\documentclass[11pt]{amsart}
\usepackage{amsmath,amssymb,amsthm,mathtools}
\usepackage{enumitem}
\usepackage{float}
\usepackage{needspace}
\usepackage{tikz}
\usetikzlibrary{arrows.meta}
\usepackage[hidelinks]{hyperref}
\usepackage{microtype}
\newtheorem{theorem}{Theorem}[section]
\newtheorem{proposition}[theorem]{Proposition}
\newtheorem{lemma}[theorem]{Lemma}
\newtheorem{corollary}[theorem]{Corollary}
\newtheorem{remark}[theorem]{Remark}
\numberwithin{equation}{section}
\newcommand{\Ric}{\operatorname{Ric}}
\newcommand{\Scal}{\operatorname{Scal}}
\newcommand{\II}{\mathrm{II}}
\newcommand{\dist}{\operatorname{dist}}
\newcommand{\Sph}{\mathbb S}

\title[Counterexample to Yau's scalar-curvature conjecture]
{A Counterexample to Yau's Conjectured\\
Asymptotic Scalar-Curvature Integral Bound}
\author{Tianze Hao}
\address{Beijing International Center for Mathematical Research, Peking University}
\email{haotz@pku.edu.cn}
\author{Jintian Zhu}
\address{Institute for Theoretical Sciences, Westlake University}
\email{zhujintian@westlake.edu.cn}
\date{August 21, 2026}

\begin{document}
\begin{abstract}
A complete one-ended three-manifold with strictly positive
Ricci curvature is constructed such that
\[
 \limsup_{R\to\infty}\frac1R\int_{B(p,R)}\Scal\,dV=+\infty.
\]
The construction combines an explicit toric lens carrying a large
scalar-curvature integral with a three-dimensional angular pair of pants and
an adaptive sequence of hybrid blocks.
\end{abstract}
\maketitle
\raggedbottom

\section{Introduction}

The Cohn--Vossen inequality bounds the total Gaussian curvature of a
complete noncompact surface \cite{CohnVossen}.  Motivated by this result,
Yau asked whether every complete noncompact $n$-manifold with
$\Ric\geq0$ satisfies
\begin{equation}\label{eq:yau-question}
 \limsup_{R\to\infty}R^{2k-n}
 \int_{B(p,R)}\sigma_k(\Ric)\,dV<\infty,
 \qquad 1\leq k\leq n,
\end{equation}
where $\sigma_k(\Ric)$ is the $k$th elementary symmetric function of the
Ricci eigenvalues \cite[Problem~9]{YauOpenProblems}.  The case $k=1$ is
the scalar-curvature problem.  In dimension three it asks whether
\begin{equation}\label{eq:yau-three-dimensional}
 \limsup_{R\to\infty}\frac1R
 \int_{B(p,R)}\Scal\,dV<\infty.
\end{equation}
Yang constructed counterexamples to the higher-order assertions in
\eqref{eq:yau-question}, but did not settle the scalar-curvature case
\cite{Yang}.

Several positive results were known under additional hypotheses.
Shi and Yau proved a scale-invariant scalar-curvature estimate for complete
noncompact K\"ahler manifolds of complex dimension at least three with bounded
pinched nonnegative holomorphic bisectional curvature \cite{ShiYau}.
Petrunin proved the local estimate
$\int_{B(p,1)}\Scal\,dV\leq C(n)$ under $\sec\geq-1$; by rescaling, this
yields the scalar estimate in \eqref{eq:yau-question} under
$\sec\geq0$ \cite{Petrunin}.  In dimension three, Xu established a
Green-function-weighted estimate on nonparabolic manifolds
\cite{XuNonparabolic}.  Subsequent comparison results of
Munteanu--Wang and Zhu developed this approach
\cite{MunteanuWangComparison,ZhuComparison}.  Zhu proved the unweighted
estimate, with constant $20\pi$, for manifolds with a pole
\cite{ZhuPole}, and Xu later proved in the pole case the exact limit
$8\pi(1-\operatorname{AVR})$ \cite{XuPole}.  Liu obtained related K\"ahler
results \cite{LiuKahler}.
Munteanu and Wang proved the sharp upper bound $8\pi$ when the scalar
curvature is bounded above and below by positive constants.  They conjectured
that these extra bounds can be removed; by the splitting theorem the
multi-ended case reduces to a product, so the substantive unresolved case in
their discussion was the one-ended case \cite{MunteanuWang}.  Further
weighted or special-case results were obtained by Chen, Xu and Zhang,
Deng, and Ma \cite{ChenXuZhang,Deng,Ma}.

Throughout the paper, $B(p,R)$ denotes the open metric ball.  For a
measurable set $E$ we call $\int_E\Scal_g\,dV_g$ its scalar action.  We use
the convention
\[
 \operatorname{AVR}(M,g)
 :=\lim_{R\to\infty}\frac{\operatorname{Vol}_g B(p,R)}{\omega_3R^3},
 \qquad \omega_3=\frac{4\pi}{3},
\]
whenever $\Ric_g\geq0$; Bishop--Gromov monotonicity gives the limit and
shows that it is independent of $p$.

The following theorem gives a negative answer to
\eqref{eq:yau-three-dimensional}, even under strict positivity of the
Ricci tensor.  It also disproves the unrestricted one-ended conjecture of
Munteanu and Wang.

\begin{theorem}\label{thm:main}
There exists a complete one-ended three-manifold $(M^3,g)$, diffeomorphic
to $\mathbb R^3$, with $\Ric_g>0$ such that, for every $p\in M$,
\begin{equation}\label{eq:limsup}
 \limsup_{R\to\infty}\frac1R
 \int_{B(p,R)}\Scal_g\,dV_g=+\infty.
\end{equation}
\end{theorem}

\begin{remark}
\label{rem:higher-dimensional}
The scalar-curvature assertion in \eqref{eq:yau-question} therefore also
fails in every dimension $n\geq4$.  Indeed, write $n=3+m$ and equip
$X^n=M^3\times\mathbb R^m$ with the product metric
$G=g+g_{\mathrm E}$.  This metric is complete, has
$\Ric_G=\Ric_g\oplus0\geq0$, and is one-ended on a manifold
diffeomorphic to $\mathbb R^n$.  If $r_j\to\infty$ is a sequence along
which the quotient in \eqref{eq:limsup} diverges, then
\[
 B_g(p,r_j)\times B_{g_{\mathrm E}}(0,r_j)
 \subset B_G((p,0),\sqrt2\,r_j).
\]
Since $\Scal_G=\Scal_g$ and $n-2=m+1$, it follows that
\[
 \begin{aligned}
 &(\sqrt2\,r_j)^{2-n}
 \int_{B_G((p,0),\sqrt2 r_j)}\Scal_G\,dV_G\\
 &\qquad\geq \omega_m2^{-(m+1)/2}
 \left(\frac1{r_j}\int_{B_g(p,r_j)}\Scal_g\,dV_g\right)
 \longrightarrow+\infty.
 \end{aligned}
\]
Here $\omega_m$ denotes the volume of the Euclidean unit $m$-ball.
Thus the product gives the required higher-dimensional counterexamples
with nonnegative Ricci curvature.  For $m>0$, this argument does not
preserve strict positivity of the Ricci tensor.
\end{remark}

\begin{proposition}
\label{prop:constructed-collapse}
The manifold constructed in Theorem~\ref{thm:main} has
\[\operatorname{AVR}(M,g)=0.\]  More precisely, there are a constant
$C_{\mathrm{col}}>0$, points
$x_j\in M$, radii $s_j>0$, and packet parameters $Q_j\to\infty$ such that
\begin{equation}\label{eq:collapsing-balls-intro}
 \frac{\operatorname{Vol}B(x_j,s_j)}{s_j^3}\leq\frac{C_{\mathrm{col}}}{Q_j}.
\end{equation}
\end{proposition}

Naber proposed a different, local noncollapsing statement: a lower Ricci
bound together with a fixed lower bound for the volume of a unit ball
should control the average scalar curvature on that ball
\cite[Conjecture~2.18]{Naber}.  This endpoint conjecture remains open
\cite{CucinottaMondino}.  Proposition~\ref{prop:constructed-collapse}
shows that the large-action stages in the present construction are
collapsing after rescaling; consequently Theorem~\ref{thm:main} is not a
counterexample to Naber's conjecture.  Recent results prove volume-growth
bounds from positive scalar or intermediate curvature
\cite{HeatKernelVolumeGrowth,ScalarVolumeGrowth,
UniversalIntermediateVolumeGrowth}.  Those results
concern Gromov's volume-growth conjecture rather than the scalar-curvature
integral in \eqref{eq:yau-three-dimensional}.

A prior construction combines a positive-curvature angular pair of pants
and Perelman necks with relative boundary and path-cylinder tools in a
counterexample to Milnor's conjecture
\cite[Lemmas~3.1--3.2, Theorem~3.3, Proposition~4.2, Theorem~4.4 and
Appendix~A]{PriorAngularConstruction}.
We use this architecture in dimension three and replace one neck by an
anisotropic filling.

Start with a three-dimensional angular pair of pants.
Fill one inner boundary by the high-action toric packet and attach a
companion neck to the other; the round far end of the neck is the input of
the block.  An outgoing path cylinder attached to the outer boundary ends
in a round strictly convex sphere.  Thus the completed hybrid block is an
annulus.  Gluing its round input to the preceding compact prefix and
iterating these annuli produces the one-ended manifold.  At each stage the
packet is chosen after the preceding geometric cost is known, and a
protected subset of that packet supplies the large scalar action.

Section~2 constructs the high-action lens.  Section~3 realizes its boundary
as an angular sphere and constructs the companion neck.  Section~4 forms
the quantitative hybrid blocks and iterates them to prove
Theorem~\ref{thm:main} and Proposition~\ref{prop:constructed-collapse}.
Appendix~\ref{app:boundary-tools} records the relative gluing and
path-cylinder tools used in those sections.

We repeatedly use the standard relative form of Perelman's Ricci-positive
gluing lemma~\cite[Section~4]{PerelmanGluing}, recorded as
Proposition~\ref{prop:relative-gluing} in
Appendix~\ref{app:boundary-tools}; see also
\cite[Lemma~3.2]{PriorAngularConstruction}.  The appendix includes a proof
only to fix the localized collar and support properties used here.

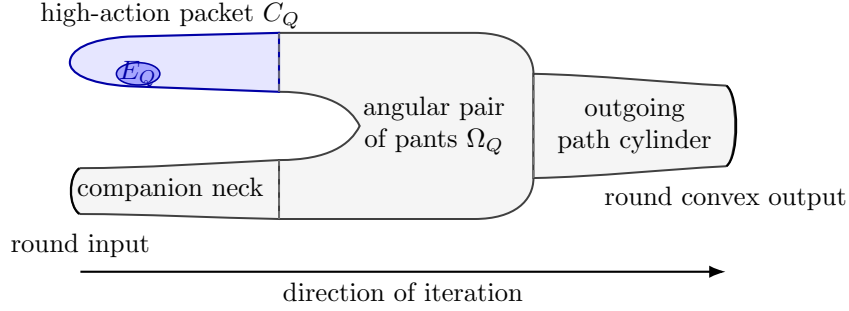
\begin{figure}[H]
\centering
\begin{tikzpicture}[
 x=0.82cm,y=0.82cm,
 component/.style={draw=black!75,fill=black!4,line width=0.75pt},
 packet/.style={draw=blue!65!black,fill=blue!10,line width=0.8pt},
 seam/.style={draw=black!55,densely dashed,line width=0.65pt},
 flow/.style={-{Latex[length=2.2mm]},line width=0.8pt},
 every node/.style={font=\small}
]
\path[component]
 (-2.35,1.50) -- (0.80,1.50)
 .. controls (1.45,1.50) and (1.75,1.20) .. (1.75,0.84)
 -- (1.75,-0.84)
 .. controls (1.75,-1.20) and (1.45,-1.50) .. (0.80,-1.50)
 -- (-2.35,-1.50) -- (-2.35,-0.55)
 .. controls (-1.55,-0.55) and (-1.20,-0.30) .. (-1.05,0)
 .. controls (-1.20,0.30) and (-1.55,0.55) .. (-2.35,0.55)
 -- cycle;

\path[packet]
 (-2.35,0.55) -- (-4.78,0.64)
 .. controls (-5.45,0.68) and (-5.72,0.86) .. (-5.72,1.02)
 .. controls (-5.72,1.22) and (-5.43,1.40) .. (-4.78,1.44)
 -- (-2.35,1.50) -- cycle;

\path[component]
 (-5.55,-1.42)
 .. controls (-4.55,-1.42) and (-3.40,-1.48) .. (-2.35,-1.50)
 -- (-2.35,-0.55)
 .. controls (-3.40,-0.60) and (-4.55,-0.68) .. (-5.55,-0.68)
 .. controls (-5.75,-0.78) and (-5.75,-1.32) .. (-5.55,-1.42)
 -- cycle;

\path[component]
 (1.75,0.84)
 .. controls (2.80,0.82) and (3.85,0.70) .. (4.85,0.65)
 .. controls (5.05,0.54) and (5.05,-0.54) .. (4.85,-0.65)
 .. controls (3.85,-0.70) and (2.80,-0.82) .. (1.75,-0.84)
 -- cycle;

\draw[seam] (-2.35,0.55) -- (-2.35,1.50);
\draw[seam] (-2.35,-1.50) -- (-2.35,-0.55);
\draw[seam] (1.75,-0.84) -- (1.75,0.84);
\draw[line width=0.75pt] (-5.57,-1.42)
  .. controls (-5.75,-1.28) and (-5.75,-0.82) .. (-5.57,-0.68);
\draw[line width=0.75pt] (4.86,-0.65)
  .. controls (5.07,-0.49) and (5.07,0.49) .. (4.86,0.65);

\filldraw[draw=blue!65!black,fill=blue!35,line width=0.6pt]
 (-4.62,0.84) ellipse [x radius=0.35,y radius=0.18];

\node[align=center] at (-4.10,1.78) {high-action packet $C_Q$};
\node[blue!65!black] at (-4.62,0.84) {$E_Q$};
\node[align=center] at (-4.10,-1.02) {companion neck};
\node[align=center] at (0.15,0) {angular pair\\of pants $\Omega_Q$};
\node[align=center] at (3.38,0) {outgoing\\path cylinder};
\node[align=center] at (-5.55,-1.93) {round input};
\node[align=center] at (4.85,-1.16) {round convex output};
\draw[flow] (-5.55,-2.35) -- (4.85,-2.35)
 node[midway,below] {direction of iteration};
\end{tikzpicture}
\caption{Schematic of one hybrid block.  One inner boundary
of the angular pair of pants is filled by the high-action packet, while the
companion neck turns the other into a round input.  The outgoing path
cylinder turns the outer boundary into a round strictly convex output.
Thus the completed block is an annulus, and the construction iterates such
blocks.}
\label{fig:hybrid-block}
\end{figure}

\subsection*{AI usage}
The authors made substantial use of ChatGPT (GPT-5.6 Pro). In
particular, ChatGPT generated the main theorems and their proofs after
several conversations with the author.  In the course
of conversations with ChatGPT, the model pointed out that no
universal upper bound exists for the scalar-curvature integral over unit
geodesic balls in closed manifolds with nonnegative Ricci curvature.  The
authors then suggested that the iterative method used to construct
counterexamples to Milnor's conjecture might be useful here and directed the
model to the construction in \cite{PriorAngularConstruction}.  From these
inputs, the model generated the entire construction.  The authors
have verified all mathematical statements and proofs and take full
responsibility for the correctness of the paper.

\section*{Acknowledgments}
The second-named author thanks Professor Wenshuai Jiang for his hospitality
during a summer school at Zhejiang University, where he
taught a mini-course.  During this period, Gromov's volume-growth
conjecture was resolved in the works
\cite{HeatKernelVolumeGrowth,ScalarVolumeGrowth,
UniversalIntermediateVolumeGrowth}.  In a conversation with Professors
Wenshuai Jiang and Wenlong Wang, the second-named author learned that
the stronger conjecture of Yau studied here was believed to fail when the
manifold is collapsing.

\section{A high-action convex lens}

Throughout this section, the angular radius $r_{\mathrm W}$ is fixed with
\begin{equation}\label{eq:r-range}
 0<r_{\mathrm W}<\frac{\pi}{10}.
\end{equation}
We first
choose $U\in(0,U_0)$ and then take $\varepsilon/U$ sufficiently small.
Constants implicit in $O_U(\cdot)$ or $\asymp_U$ may depend on $U$ and on
$r_{\mathrm W}$, but not on $\varepsilon$.  Both angular variables have
period $2\pi$, and ``waist'' means orbit radius rather than orbit
circumference.
Plain terms $O(U^m)$ below have constants uniform for $0<U<U_0$ in the
iterated regime in which first $U\downarrow0$ and then
$\varepsilon/U\downarrow0$; only $O_U(\cdot)$ permits dependence on the
fixed value of $U$.

\begin{proposition}
\label{prop:high-action-lens}
Fix $0<r_{\mathrm W}<\pi/10$.  There is $U_0>0$ such that, for every
$U\in(0,U_0)$ and all sufficiently small $\varepsilon>0$, there is a
smooth Ricci-positive three-ball $(C_{\varepsilon,U},g_{\varepsilon,U})$
whose metric is normalized so that its rotational two-sphere boundary has
full meridian length $\pi\cos r_{\mathrm W}$ and waist
$w_{\varepsilon,U}\asymp_U\varepsilon$.  Moreover, it contains a compact
set $\mathcal E_{\varepsilon,U}$ and constants
$c,C>0$, independent of $\varepsilon$, such that

\begin{equation}\label{eq:high-action-summary}
 \int_{\mathcal E_{\varepsilon,U}}\Scal\,dV
 \geq \frac{c}{w_{\varepsilon,U}},\qquad
 \sup_{x\in C_{\varepsilon,U}}
 \dist(x,\partial C_{\varepsilon,U})\leq C.
\end{equation}
Here $c$ and $C$ may depend on $U$ and $r_{\mathrm W}$, but not on
$\varepsilon$.  The set $\mathcal E_{\varepsilon,U}$ is disjoint from a
sufficiently thin boundary collar.
\end{proposition}

\begin{lemma}\label{lem:packet-data}
For every sufficiently small $\varepsilon>0$, there is a $C^1$ toric
cap--core metric
\begin{equation}\label{eq:packetmetric}
 g_\varepsilon=dr^2+a_\varepsilon(r)^2d\theta^2
                 +f_\varepsilon(r)^2d\phi^2.
\end{equation}
Its cap has positive Ricci curvature, the cap and core agree to first order,
and the transition radius is $O(e^{-c/\varepsilon^3})$.
\end{lemma}

\begin{proof}
We use the doubly warped metric \eqref{eq:packetmetric}.
Where $a_\varepsilon,f_\varepsilon>0$, the standard connection formulas
show that the orthonormal frame
\[
 e_r=\partial_r,\qquad
 e_\theta=a_\varepsilon^{-1}\partial_\theta,\qquad
 e_\phi=f_\varepsilon^{-1}\partial_\phi
\]
diagonalizes the Ricci tensor.  Thus Ricci positivity is determined by its
three diagonal eigenvalues; at a smooth collapsed orbit these eigenvalues
extend by continuity.
The packet cap is joined to the following explicit core.
Put
\begin{equation}\label{eq:constants}
 C_\varepsilon=1-\frac{\varepsilon^3}{6},\qquad
 c_\varepsilon=\sqrt{2C_\varepsilon+\frac7{24}}.
\end{equation}
On the core, use $z\in[2\varepsilon,1]$ as the parameter and define
\begin{equation}\label{eq:A}
 A_\varepsilon(z)
 =z^{1/6}\exp\left\{\frac{8C_\varepsilon}{3}(1-z^{-3})\right\}.
\end{equation}
Let $r=r(z)$ be the increasing radial coordinate, determined up to an
additive constant by
\begin{equation}\label{eq:dr}
 dr=4c_\varepsilon\varepsilon A_\varepsilon(z)z^{-3}\,dz.
\end{equation}
Define the core profiles by
\begin{equation}\label{eq:af}
 a_\varepsilon(r(z))=c_\varepsilon\varepsilon^2A_\varepsilon(z),
 \qquad
 f_\varepsilon(r(z))=\frac{2\varepsilon}{z}.
\end{equation}
Equivalently, the reciprocal fiber-radius variable is
\[
 x:=\frac{1}{f_\varepsilon(r(z))}=\frac{z}{2\varepsilon}.
\]
Differentiating \eqref{eq:af} and applying the chain rule gives
\begin{equation}\label{eq:aprime}
\begin{aligned}
 a_\varepsilon'(r)
 &=\frac{c_\varepsilon\varepsilon^2A_\varepsilon'(z)}
         {4c_\varepsilon\varepsilon A_\varepsilon(z)z^{-3}}
  =\varepsilon
       \left(\frac{2C_\varepsilon}{z}+\frac{z^2}{24}\right).
\end{aligned}
\end{equation}
The same chain-rule calculation, using $f_\varepsilon=1/x$, gives
\[
 f_\varepsilon'(r)=-\frac{\varepsilon^3x}{a_\varepsilon(r)}.
\]
Define the transition radius by
\[
 r_\varepsilon^{\mathrm{tr}}
 :=c_\varepsilon\varepsilon^2A_\varepsilon(2\varepsilon),
\]
and choose the additive constant in $r(z)$ so that
$r(2\varepsilon)=r_\varepsilon^{\mathrm{tr}}$, equivalently
$z(r_\varepsilon^{\mathrm{tr}})=2\varepsilon$.  On the cap, for
$0\leq r\leq r_\varepsilon^{\mathrm{tr}}$, set
\begin{equation}\label{eq:cap}
 \begin{aligned}
 a_\varepsilon(r)&=r,&
 f_\varepsilon(r)&=f_\varepsilon^{\mathrm{pole}}
       (1-\beta_\varepsilon^{\mathrm{cap}}r^2),\\
 \beta_\varepsilon^{\mathrm{cap}}
       (r_\varepsilon^{\mathrm{tr}})^2
   &=\frac{\varepsilon^3}{2+\varepsilon^3},&
 f_\varepsilon^{\mathrm{pole}}
   &=\frac{2+\varepsilon^3}{2}.
 \end{aligned}
\end{equation}
The cap is smooth at $r=0$: $a_\varepsilon(r)=r$ is odd with
$a_\varepsilon'(0)=1$, while
$f_\varepsilon(r)=f_\varepsilon^{\mathrm{pole}}
(1-\beta_\varepsilon^{\mathrm{cap}}r^2)$ is positive and even.
Here $A_\varepsilon'$ denotes differentiation with respect to $z$, whereas
primes on $a_\varepsilon$ and $f_\varepsilon$ denote differentiation with
respect to $r$.  At $r=r_\varepsilon^{\mathrm{tr}}$,
\begin{equation}\label{eq:transition-first-jet}
 (a_\varepsilon,f_\varepsilon,a_\varepsilon',f_\varepsilon')
 =\left(r_\varepsilon^{\mathrm{tr}},1,1,
 -\frac{\varepsilon^3}{r_\varepsilon^{\mathrm{tr}}}\right),
\end{equation}
which agrees to first order with the core at $x=1$.  With respect to
$e_r,e_\theta,e_\phi$, respectively, the cap Ricci eigenvalues for $r>0$
are
\[
 \frac{2\beta_\varepsilon^{\mathrm{cap}}
 f_\varepsilon^{\mathrm{pole}}}{f_\varepsilon},\qquad
 \frac{2\beta_\varepsilon^{\mathrm{cap}}
 f_\varepsilon^{\mathrm{pole}}}{f_\varepsilon},\qquad
 \frac{4\beta_\varepsilon^{\mathrm{cap}}
 f_\varepsilon^{\mathrm{pole}}}{f_\varepsilon}.
\]
They are positive and extend to positive limits at $r=0$.
Thus the transition radius satisfies
\begin{equation}\label{eq:transition-radius-decay}
 r_\varepsilon^{\mathrm{tr}}
 =c_\varepsilon\varepsilon^2(2\varepsilon)^{1/6}
 \exp\left\{\frac{8C_\varepsilon}{3}
       -\frac{C_\varepsilon}{3\varepsilon^3}\right\}
 =O(e^{-c/\varepsilon^3}).
\end{equation}
The piecewise metric is $C^1$.
This proves the stated cap--core properties.
\end{proof}

\begin{remark}
Although it is not needed to prove Lemma~\ref{lem:packet-data}, the cap
formula \eqref{eq:cap} also gives the identity
\begin{equation}\label{eq:pole-flux}
 f_\varepsilon(0)a_\varepsilon'(0)
 =f_\varepsilon^{\mathrm{pole}}=1+\frac{\varepsilon^3}{2},
\end{equation}
and hence $f_\varepsilon(0)a_\varepsilon'(0)\to1$.  This pole-flux identity
is used in the proofs of Lemmas~\ref{lem:action} and
\ref{lem:protected-set}.
\end{remark}

Until the smoothing in Lemma~\ref{lem:cap-smoothing} has been fixed, after
$\varepsilon$ is fixed we abbreviate the piecewise profiles by
$a:=a_\varepsilon$ and $f:=f_\varepsilon$.

\begin{lemma}\label{lem:core-curvature}
The explicit core constructed in Lemma~\ref{lem:packet-data} has positive
Ricci curvature.  Moreover,
$K_{r\theta}>0$ and $K_{\phi\theta}>0$ throughout the core.
\end{lemma}

\begin{proof}
For this doubly warped metric, the orthonormal frame
$e_r=\partial_r$, $e_\theta=a^{-1}\partial_\theta$, and
$e_\phi=f^{-1}\partial_\phi$ is the Ricci eigenframe identified above.
Its pairwise sectional curvatures are
\begin{equation}\label{eq:sectional}
 K_{r\theta}=-\frac{a''}{a},\qquad
 K_{r\phi}=-\frac{f''}{f},\qquad
 K_{\theta\phi}=-\frac{a'f'}{af}.
\end{equation}
Direct differentiation gives
\begin{align}
 a''&=\frac{-C_\varepsilon\varepsilon^3x+
               \frac13\varepsilon^6x^4}{a},                    \label{eq:asecond}\\
 \frac{f''}{f}
 &=\frac{C_\varepsilon\varepsilon^3x-
               \frac56\varepsilon^6x^4}{a^2},\qquad
 -\frac{a'f'}{af}=\frac{a'\varepsilon^3x^2}{a^2}.              \label{eq:fsecond}
\end{align}
Because, in dimension three, each Ricci eigenvalue is the sum of the two
sectional curvatures containing its eigendirection, one obtains
\begin{align}
 \Ric(e_r,e_r)&=\frac{\varepsilon^6x^4}{2a^2}>0,                 \label{eq:ric-r}\\
 \Ric(e_\phi,e_\phi)&=\frac{\varepsilon^6x^4}{a^2}>0,          \label{eq:ric-phi}\\
 \Ric(e_\theta,e_\theta)
 &=\frac{2C_\varepsilon\varepsilon^3x-
               \frac16\varepsilon^6x^4}{a^2}.                  \label{eq:ric-theta}
\end{align}
On the core, $x\leq(2\varepsilon)^{-1}$, so
$\varepsilon^3x^3\leq1/8$.  The expression in
\eqref{eq:ric-theta} is therefore positive for all sufficiently small
$\varepsilon$, since
\[
 2C_\varepsilon-\frac16\varepsilon^3x^3
 \geq 2\left(1-\frac{\varepsilon^3}{6}\right)-\frac1{48}>0.
\]
Indeed,
\[
 K_{r\theta}=\frac{\varepsilon^3x}{a^2}
 \left(C_\varepsilon-\frac{\varepsilon^3x^3}{3}\right)>0,
 \qquad
 K_{\phi\theta}=\frac{a'\varepsilon^3x^2}{a^2}>0.
\]
\end{proof}

\begin{lemma}\label{lem:core-comparison}
Let $I(r)=\int_0^r f(\tau)\,d\tau$.  If $a'I-af>0$ at one point of the explicit core,
then this inequality remains strict at every later core point.
\end{lemma}

\begin{proof}
Return to $x=f^{-1}$ and put
\[
 V=xa'=C_\varepsilon+\frac16\varepsilon^3x^3,
 \qquad R=\frac aV.
\]
Since $af=a/x$, the inequality $a'I\geq af$ is equivalent to $I\geq R$.
The core equations give
\[
 I_x=\frac{a}{\varepsilon^3x^4},\qquad
 \frac{a_x}{a}=\frac1{6x}+
 \frac{C_\varepsilon}{\varepsilon^3x^4},\qquad
 V_x=\frac12\varepsilon^3x^2,
\]
and hence
\begin{equation}\label{eq:R-I-comparison}
 \frac{R_x}{I_x}
 =\frac{\varepsilon^3x^4}{V}
   \left(\frac{a_x}{a}-\frac{V_x}{V}\right)
 =1-\frac{(\varepsilon^3x^3)^2}{2V^2}\leq1.
\end{equation}
Thus a strict inequality $I>R$ propagates along the core: explicitly,
\eqref{eq:R-I-comparison} gives
$(I-R)_x=I_x(1-R_x/I_x)\geq0$.
\end{proof}

\begin{lemma}
\label{lem:cap-smoothing}
For every fixed sufficiently small $\varepsilon$, the cap--core join has a
smooth modification whose change is compactly supported in an arbitrarily
short collar of $r=r_\varepsilon^{\mathrm{tr}}$,
equal to the displayed cap and core outside that collar, such that
\begin{equation}\label{eq:smoothing-monotone}
 \Ric>0,\qquad a'>0,\qquad a''\leq0,\qquad
 f'(0)=0,\quad f'(r)<0\quad(r>0).
\end{equation}
The collar can be chosen so that, for the smoothed primitive
$I(r)=\int_0^r f$,
\begin{equation}\label{eq:smoothing-W}
 a'I-af>0
\end{equation}
throughout the collar and on the ensuing unchanged core.  It can moreover be
chosen with radial width at most $(r_\varepsilon^{\mathrm{tr}})^2$.  Its
full-torus scalar-action contribution is
$O(\varepsilon^3r_\varepsilon^{\mathrm{tr}})$.
\end{lemma}

\begin{proof}
Fix a sufficiently small $\varepsilon$ such that
\[
 \nu_\varepsilon\leq1,
 \qquad
 r_\varepsilon^{\mathrm{tr}}\leq1;
\]
the second condition follows from
\eqref{eq:transition-radius-decay}.  Put
\[
 \nu:=\nu_\varepsilon=\varepsilon^3,
 \qquad R:=r_\varepsilon^{\mathrm{tr}},
 \qquad s:=r-R,
 \qquad \mathbf u:=(a,f),
\]
and use the curvature-adapted norm
\[
 \lVert(p,q)\rVert_*:=\frac{|p|}{R}+|q|.
\]
The common zero and first jet from \eqref{eq:transition-first-jet} is
\[
 (a,f,a',f')=(R,1,1,-\nu/R).
\]
The cap formula \eqref{eq:cap} and the core formulas
\eqref{eq:asecond}--\eqref{eq:fsecond}, evaluated at $x=1$, give the
one-sided second jets
\begin{equation}\label{eq:smoothing-jets}
\begin{aligned}
 Z_-&:=(a'',f'')_-=\left(0,-\frac{\nu}{R^2}\right),\\
 Z_+&:=(a'',f'')_+
 =\left(-\frac{\nu(1-\nu/2)}R,
          \frac{\nu(1-\nu)}{R^2}\right).
\end{aligned}
\end{equation}

\smallskip
\noindent\emph{Step 1. The Ricci-positive model segment.}
For the doubly warped metric $dr^2+a^2d\theta^2+f^2d\phi^2$, the natural
orthonormal frame is a Ricci eigenframe, with eigenvalues
\begin{equation}\label{eq:smoothing-Ricci-formula}
\begin{aligned}
 R_r&=-\frac{a''}a-\frac{f''}f,\\
 R_\theta&=-\frac{a''}a-\frac{a'f'}{af},\\
 R_\phi&=-\frac{f''}f-\frac{a'f'}{af}.
\end{aligned}
\end{equation}
Freeze the common zero and first jet above and interpolate only the second
jet:
\[
 Z_\tau=(1-\tau)Z_-+\tau Z_+,
 \qquad 0\leq\tau\leq1.
\]
Because \eqref{eq:smoothing-Ricci-formula} is affine in $(a'',f'')$ when
the lower jets are fixed, direct substitution gives
\begin{equation}\label{eq:smoothing-Ricci-segment}
\begin{aligned}
 R_r(Z_\tau)
 &=\frac{\nu}{R^2}\left(1-\tau(1-\nu/2)\right)
   \geq\frac{\nu^2}{2R^2},\\
 R_\theta(Z_\tau)
 &=\frac{\nu}{R^2}\left(1+\tau(1-\nu/2)\right)
   \geq\frac{\nu}{R^2},\\
 R_\phi(Z_\tau)
 &=\frac{\nu}{R^2}\left(2-\tau(2-\nu)\right)
   \geq\frac{\nu^2}{R^2}.
\end{aligned}
\end{equation}
Thus the entire segment lies strictly inside the Ricci-positive cone, with
smallest margin $\nu^2/(2R^2)$.  The rest of the construction keeps the
actual smoothed jets within a fixed fraction of this margin.

\smallskip
\noindent\emph{Step 2. The collar and the preliminary second-jet interpolation.}
Let $C\geq1$ be an absolute constant large enough to dominate all fixed
constants in the estimates below, and choose once and for all
\[
 0<c_1\leq\min\left\{1,\frac1{8C}\right\}.
\]
With $\varepsilon$ and $c_1$ now fixed, every subsequent approximation
requirement will be met by decreasing $\delta$ and then $\eta/\delta$.
Then choose
\begin{equation}\label{eq:smoothing-collar-width}
 0<\delta\leq c_1\min\{\nu R,R^2\}
\end{equation}
so small that, on each unsmoothed half-collar, the zero and first jets
satisfy
\begin{equation}\label{eq:smoothing-first-jet-error}
 \frac{|a-R|}{R}+|f-1|+|a'-1|
 +\frac R\nu\left|f'+\frac\nu R\right|
 \leq c_1\nu,
\end{equation}
and, with the appropriate sign in \eqref{eq:smoothing-jets},
\[
 \frac{|a''-a''_\pm|}{R}+|f''-f''_\pm|
 \leq c_1\frac{\nu^2}{R^2}.
\]
This is possible because $\varepsilon$, and hence $R>0$, is fixed and the
two one-sided profiles are smooth up to the join.

Extend the one-sided second-derivative germs smoothly across the whole
collar as $G_-$ and $G_+$, respectively, so that
\begin{align*}
 G_-(s)&=\mathbf u''(s) &&(-\delta\leq s<0),\\
 G_+(s)&=\mathbf u''(s) &&(0<s\leq\delta),\\
 \sup_{|s|\leq\delta}\lVert G_-(s)-Z_-\rVert_*
 &\leq 2c_1\frac{\nu^2}{R^2},\\
 \sup_{|s|\leq\delta}\lVert G_+(s)-Z_+\rVert_*
 &\leq 2c_1\frac{\nu^2}{R^2}.
\end{align*}
Their first components can also be kept nonpositive: the cap component is
identically zero, whereas the core component is strictly negative near the
join.

Choose a smooth nondecreasing cutoff $\vartheta$ which is zero on
$(-\infty,-1]$ and one on $[1,\infty)$, take $0<\eta<\delta/10$, and set
\[
 \tau(s):=\vartheta(s/\eta),\qquad
 G^0(s):=(1-\tau(s))G_-(s)+\tau(s)G_+(s).
\]
The function $G^0$ is a preliminary candidate for
$\widetilde{\mathbf u}''$.  It is smooth and agrees with the original cap
and core second derivatives for $s\leq-\eta$ and $s\geq\eta$,
respectively.

\smallskip
\noindent\emph{The endpoint obstruction and its two moment corrections.}
Initially, let $\mathbf u''$ denote the piecewise classical second
derivative of the original $C^1$ profile.  For every
$\varphi\in C_c^\infty((-\delta,\delta))$, integration by parts on the two
half-collars gives, componentwise,
\[
 \left\langle\partial_s^2\mathbf u,\varphi\right\rangle
 =\int_{-\delta}^0\mathbf u''(s)\varphi(s)\,ds
  +\int_0^\delta\mathbf u''(s)\varphi(s)\,ds
  +\bigl(\mathbf u'(0+)-\mathbf u'(0-)\bigr)\varphi(0).
\]
The last term vanishes because the cap and core have the same first jet at
$s=0$.  Thus the piecewise classical second derivative is also the weak
second derivative and carries no Dirac mass at the join.  Set
\[
 D:=G^0-\mathbf u'',
\]
so that $D$ is supported in $[-\eta,\eta]$.

The need for two moment conditions can be seen before making any
correction.  If a profile $\mathbf v$ is reconstructed from the exact
left-hand jet with $\mathbf v''=G^0$, then, to the right of the support of
$D$,
\[
 (\mathbf v-\mathbf u)'=\int_{-\delta}^{\delta}D(t)\,dt,
 \qquad
 \mathbf v-\mathbf u
 =s\int_{-\delta}^{\delta}D(t)\,dt
   -\int_{-\delta}^{\delta}tD(t)\,dt.
\]
Thus the zeroth moment of $D$ is the mismatch of the right first jet, and,
after it vanishes, the first moment is the remaining mismatch of the right
zero jet.

From \eqref{eq:smoothing-jets},
\[
 \lVert Z_+-Z_-\rVert_*
 \leq3\frac\nu{R^2}.
\]
The extension bounds therefore give
\begin{align*}
 \lVert D\rVert_*&\leq C\frac\nu{R^2},\\
 \left\lVert\int_{-\delta}^{\delta}D(s)\,ds\right\rVert_*
 &\leq C\eta\frac\nu{R^2},\\
 \left\lVert\int_{-\delta}^{\delta}sD(s)\,ds\right\rVert_*
 &\leq C\eta^2\frac\nu{R^2}.
\end{align*}

To cancel these two moments, choose fixed scalar bumps
$\psi_0,\psi_1\in C_c^\infty((1/3,2/3))$ such that

\[
 M:=\begin{pmatrix}
  \int\psi_0&\int\psi_1\\
  \int t\psi_0(t)\,dt&\int t\psi_1(t)\,dt
 \end{pmatrix}
\]
is invertible.  After replacing $\psi_j(t)$ by $\psi_j(s/\delta)$, their
moment matrix is
\[
 \begin{pmatrix}\delta&0\\0&\delta^2\end{pmatrix}M.
\]
Solving the resulting two-by-two system in each component produces a
correction $\Gamma$, supported in $(\delta/3,2\delta/3)$, such that
\begin{equation}\label{eq:smoothing-moments}
 \mathbf E:=G^0+\Gamma-\mathbf u'',\qquad
 \int_{-\delta}^{\delta}\mathbf E(s)\,ds=0,
 \qquad
 \int_{-\delta}^{\delta}s\mathbf E(s)\,ds=0.
\end{equation}
Put
\[
 m_j:=\int_{-\delta}^{\delta}s^jD(s)\,ds,
 \qquad j=0,1,
\]
and set
\[
 \lVert\Gamma\rVert_{*,\infty}
 :=\sup_{|s|\leq\delta}\lVert\Gamma(s)\rVert_*.
\]
Then inversion of the scaled moment matrix gives
\[
 \lVert\Gamma\rVert_{*,\infty}
 \leq C\left(\frac{\lVert m_0\rVert_*}{\delta}
              +\frac{\lVert m_1\rVert_*}{\delta^2}\right)
 \leq C\left(\frac\eta\delta+\frac{\eta^2}{\delta^2}\right)
             \frac\nu{R^2}.
\]
We now impose the scale separation
\begin{equation}\label{eq:smoothing-scale-hierarchy}
 \frac\eta\delta\leq c_1\nu.
\end{equation}
Since $c_1\nu<1$, this yields
\[
 \lVert\Gamma\rVert_{*,\infty}
 \leq Cc_1\frac{\nu^2}{R^2}.
\]

Set $\widetilde G:=G^0+\Gamma$ and reconstruct the profiles from the exact
left-hand jet:
\begin{equation}\label{eq:smoothing-reconstruction}
 \widetilde{\mathbf u}(s)=\mathbf u(-\delta)
 +(s+\delta)\mathbf u'(-\delta)
 +\int_{-\delta}^{s}(s-t)\widetilde G(t)\,dt.
\end{equation}
Then $\widetilde{\mathbf u}=(\widetilde a,\widetilde f)$ is smooth and
$\widetilde{\mathbf u}''=\widetilde G$.  Near the left endpoint,
$\mathbf E=0$ and the initial jets agree, so
$\widetilde{\mathbf u}=\mathbf u$.  Near the right endpoint, the support
of $\mathbf E$ has been passed, and \eqref{eq:smoothing-moments} gives
\[
 (\widetilde{\mathbf u}-\mathbf u)'
 =\int_{-\delta}^{\delta}\mathbf E(t)\,dt=0,
\]
and
\[
 \widetilde{\mathbf u}-\mathbf u
 =s\int_{-\delta}^{\delta}\mathbf E(t)\,dt
  -\int_{-\delta}^{\delta}t\mathbf E(t)\,dt=0.
\]
Hence the modified and original profiles agree identically on
neighborhoods of both collar boundaries; in particular, all endpoint jets
match.

\smallskip
\noindent\emph{Step 3. Ricci positivity and the lower-jet estimates.}
Because $Z_{\tau(s)}=(1-\tau(s))Z_-+\tau(s)Z_+$, the extension estimates
give
\[
 \lVert G^0(s)-Z_{\tau(s)}\rVert_*
 \leq2c_1\frac{\nu^2}{R^2}.
\]
Together with the estimate for $\Gamma$, this proves
\begin{equation}\label{eq:smoothing-second-jet-error}
 \frac{|\widetilde a''(s)-a''_{\tau(s)}|}{R}
 +|\widetilde f''(s)-f''_{\tau(s)}|
 \leq Cc_1\frac{\nu^2}{R^2}.
\end{equation}
More sharply, the support and size estimates for $D$ and $\Gamma$ give
\[
 \int_{-\delta}^{\delta}\lVert\mathbf E(s)\rVert_*\,ds
 \leq C\frac{\eta\nu}{R^2}.
\]
Indeed, the $D$-term has support of length $2\eta$, while the integral of
the $\Gamma$-term is bounded by
\[
 C\delta\left(\frac\eta\delta+\frac{\eta^2}{\delta^2}\right)
   \frac\nu{R^2}
 \leq C\frac{\eta\nu}{R^2}.
\]
Consequently,
\[
 \frac{|(\widetilde a-a)'|}{R}+|(\widetilde f-f)'|
 \leq C\frac{\eta\nu}{R^2},
 \qquad
 \frac{|\widetilde a-a|}{R}+|\widetilde f-f|
 \leq C\frac{\delta\eta\nu}{R^2}.
\]
In view of \eqref{eq:smoothing-collar-width} and
\eqref{eq:smoothing-scale-hierarchy}, the modified zero and first jets
satisfy \eqref{eq:smoothing-first-jet-error} with upper bound
$Cc_1\nu$.  By the fixed choice of $c_1$, these same estimates give
\[
 \widetilde a\geq\frac R2,
 \qquad
 \widetilde f\geq\frac12,
\]
so every denominator in the following Ricci comparison is uniformly
controlled at its natural scale.

Direct subtraction in \eqref{eq:smoothing-Ricci-formula}, now comparing
with the already defined point $Z_{\tau(s)}$, gives
\begin{align*}
 |\widetilde R_i-R_i(Z_{\tau(s)})|\leq C\bigg(&
 \frac{|\widetilde a''-a''_{\tau(s)}|}{R}
 +|\widetilde f''-f''_{\tau(s)}|\\
 &+\frac\nu{R^2}
  \left(\frac{|\widetilde a-R|}{R}
        +|\widetilde f-1|+|\widetilde a'-1|\right)
 +\frac1R\left|\widetilde f'+\frac\nu R\right|\bigg).
\end{align*}
By the choice of the absolute constant $C$, the preceding estimates imply
\[
 |\widetilde R_i-R_i(Z_{\tau(s)})|
 \leq Cc_1\frac{\nu^2}{R^2}.
\]
The choice of $c_1$ makes the Ricci perturbation at most
$\nu^2/(8R^2)$, strictly smaller
than the least margin in \eqref{eq:smoothing-Ricci-segment}.  Hence
$\Ric>0$ throughout the smoothing collar.

\smallskip
\noindent\emph{Step 4. Monotonicity.}
The first components of $G_-$ and $G_+$ are nonpositive, so the first
component of $G^0$ is nonpositive.  Because $\eta<\delta/10$, the
correction $\Gamma_a$ is supported where $\tau=1$ and hence where
$G^0=G_+=\mathbf u''$ is the genuine core second derivative.  On this
support, \eqref{eq:smoothing-jets} and the one-sided second-jet estimate
following \eqref{eq:smoothing-first-jet-error} give
\[
 -a''\geq
 \frac{\nu(1-\nu/2)}R-c_1\frac{\nu^2}R
 \geq\frac{\nu}{4R}.
\]
On the other hand,
\[
 |\Gamma_a|\leq R\lVert\Gamma\rVert_{*,\infty}
 \leq Cc_1\frac{\nu^2}{R}.
\]
The chosen bound on $c_1$ and $\nu\leq1$ imply
\[
 |\Gamma_a|\leq\frac{\nu}{8R}.
\]
Thus $\widetilde a''<0$ on the support of $\Gamma_a$; off that support,
$\Gamma_a=0$ and the first component of $G^0$ is nonpositive.  Therefore
$\widetilde a''\leq0$ throughout the collar.  The first-jet estimate gives
more explicitly
\[
 \widetilde a'=1+O(c_1\nu)>0,
 \qquad
 \widetilde f'=-\frac\nu R\bigl(1+O(c_1\nu)\bigr)<0.
\]
The collar does not meet the axis, so the unchanged cap still has
$\widetilde f'(0)=0$.  Outside the collar, all the asserted curvature and
sign conditions hold by the preceding cap and core calculations.

\smallskip
\noindent\emph{Step 5. The strict lens inequality.}
On the unchanged cap,
\begin{equation}\label{eq:cap-I}
 I(r)=f_\varepsilon^{\mathrm{pole}}
 \left(r-\frac{\beta_\varepsilon^{\mathrm{cap}}r^3}{3}\right),\qquad
 I(r)-rf(r)=\frac23f_\varepsilon^{\mathrm{pole}}
 \beta_\varepsilon^{\mathrm{cap}}r^3\geq0.
\end{equation}
At the unsmoothed join, $a=R$, $a'=1$, and $f=1$, so
\[
 a'I-af=I(R)-R=\frac{R\nu}{3}>0.
\]
With $\varepsilon$ and $c_1$ fixed, decrease $\delta$ if necessary.
Continuity on the two one-sided profiles then gives
\[
 a'I-af\geq\frac{R\nu}{4}
\]
throughout the unsmoothed half-collars.

Set $\widetilde I(r):=\int_0^r\widetilde f(t)\,dt$.  Since the change in
$f$ is supported in a collar of length $2\delta$,
\[
 \lVert\widetilde I-I\rVert_\infty
 \leq2\delta\lVert\widetilde f-f\rVert_\infty.
\]
Since $R\leq1$, \eqref{eq:smoothing-collar-width} gives
$\delta\leq R$.  Using $a,I,\widetilde I=O(R)$ and
$a',f,\widetilde f=O(1)$ in the collar, the preceding bound absorbs the
$\widetilde I-I$ term when the difference of the two lens expressions is
expanded.  Hence
\[
 \left|\widetilde a'\widetilde I-\widetilde a\widetilde f
       -(a'I-af)\right|
 \leq C\left(
 R\lVert\widetilde a'-a'\rVert_\infty
 +\lVert\widetilde a-a\rVert_\infty
 +R\lVert\widetilde f-f\rVert_\infty\right).
\]
The integrated estimates above and $\delta\leq R$ bound the right-hand
side by $C\eta\nu$.  With $R$ and $\delta$ fixed, reduce $\eta/\delta$ if
necessary so that $C\eta<R/8$.  Thus
\[
 \widetilde a'\widetilde I-\widetilde a\widetilde f
 \geq\frac{R\nu}{8}>0
\]
throughout the collar.  At its right endpoint, the modified warping
profiles agree with the explicit core profiles.  Since there $f=1/x$ and
$a'>0$, the last inequality is equivalent to
\[
 \widetilde I>\frac{a}{xa'}.
\]
The proof of Lemma~\ref{lem:core-comparison} applies to the new primitive
$\widetilde I$ and propagates this strict inequality through the unchanged
core.

\smallskip
\noindent\emph{Step 6. Width and scalar action.}
All choices above are made after fixing $\varepsilon$; hence $\delta$ and
$\eta$ may be reduced as far as required.  From
\eqref{eq:smoothing-collar-width} and $c_1\leq1/4$, the full collar width
satisfies
\[
 2\delta\leq\frac{R^2}{2}\leq R^2.
\]
The transition curvatures are $O(\nu/R^2)$, while $af=O(R)$ and the radial
width is $O(\delta)$.  Absorbing the angular factor, the absolute
full-torus scalar-action contribution is therefore
\[
 O\left(\frac\nu{R^2}\,R\delta\right)
 =O\left(\frac{\nu\delta}{R}\right)
 =O(\nu R)
 =O(\varepsilon^3r_\varepsilon^{\mathrm{tr}}).
\]
This proves the lemma.
\end{proof}

For each sufficiently small $\varepsilon$, fix one smoothing furnished by
Lemma~\ref{lem:cap-smoothing}.  Henceforth
$g_\varepsilon,a_\varepsilon,f_\varepsilon$ denote the resulting smooth
metric and its warping profiles; they agree with the explicit cap and core
outside the smoothing collar.  Set
$I_\varepsilon(r):=\int_0^r f_\varepsilon(\tau)\,d\tau$.  Once the relevant
parameters are fixed, we suppress their subscripts when no ambiguity can
result.

\begin{lemma}
\label{lem:lens-ode}
For each fixed $U\in(0,1)$ and all sufficiently small $\varepsilon>0$,
let $r_U$ be the core point $z=U$.
There is a continuous function $\Psi:[0,r_U]\to[0,\infty)$, smooth on
$(0,r_U)$, with $\Psi(r_U)=0$ and $\Psi(0)>0$ such that the two graphs
$\phi=\pm\Psi(r)$ have constant positive meridional principal curvature
$\kappa$ for some $\kappa>0$.  Set
$\Delta_{\varepsilon,U}:=\Psi(0)$.  If
$\Delta_{\varepsilon,U}<\pi$, then
\begin{equation}\label{eq:domain}
 \Omega_{\varepsilon,U}
 =\{0\leq r\leq r_U,\ -\Psi(r)\leq\phi\leq\Psi(r)\}
\end{equation}
is embedded in the $\phi$-circle.  On either graph face,
the induced rotational metric $h=du^2+b(u)^2d\theta^2$ satisfies
$b'>0$ before the equator, $b''(0)=0$, and $b''<0$ away from the pole.
\end{lemma}

\begin{proof}
Fix $U\in(0,1)$ independently of $\varepsilon$, take
$0<\varepsilon<U/2$ sufficiently small, and set
\[
 a:=a_\varepsilon,\qquad f:=f_\varepsilon,\qquad
 I(r)=\int_0^r f(\tau)\,d\tau.
\]
By construction and \eqref{eq:smoothing-monotone},
\[
 f>0,\qquad f'\leq0,\qquad a'>0,\qquad a''\leq0.
\]
The graph $\phi=\Psi(r)$ is generated by a curve in the meridional
metric $dr^2+f(r)^2d\phi^2$.  With the signed convention
$k_g=\langle\nabla_T\nu,T\rangle$, its geodesic curvature $k_g$ is
exactly the meridional principal curvature of the graph face.  Thus the
condition below prescribes a constant-geodesic-curvature generating
curve.
We seek the lens \eqref{eq:domain} as two symmetric graphs
$\phi=\pm\Psi(r)$.  The upper face bounds the region
$\phi-\Psi(r)\leq0$.  Normalizing the gradient of
$\phi-\Psi(r)$ gives the outward unit normal; the construction below will
show that its radial component points toward $\partial_r$:
\[
 \nu=\frac{-f\Psi'\,\partial_r+f^{-1}\partial_\phi}
 {\sqrt{1+f^2(\Psi')^2}}.
\]
For $0\leq r<r_U$, let $y$ be its radial component:
\begin{equation}\label{eq:y}
 y:=\langle\nu,\partial_r\rangle
 =-\frac{f\Psi'}{\sqrt{1+f^2(\Psi')^2}}.
\end{equation}
Thus $|y|\leq1$, and, wherever $|y|<1$,
\[
 -\Psi'=\frac{y}{f\sqrt{1-y^2}}.
\]
This variable records the normal angle and turns the prescribed-curvature
equation into a first-order equation.  Indeed, on $0<r<r_U$, introduce the
orthonormal frame
\[
 e_r=\partial_r,\qquad e_\theta=a^{-1}\partial_\theta,
 \qquad e_\phi=f^{-1}\partial_\phi.
\]
Then the unit normal and the unit meridional tangent directed toward
increasing $r$ are
\[
 \nu=y e_r+\sqrt{1-y^2}\,e_\phi,
 \qquad
 T=\sqrt{1-y^2}\,e_r-y e_\phi.
\]
The relevant connection identities for the doubly warped metric are
\[
 \begin{aligned}
 \nabla_{e_r}e_r&=0,
 &\qquad \nabla_{e_r}e_\phi&=0,\\
 \nabla_{e_\phi}e_r&=\frac{f'}f e_\phi,
 &\qquad \nabla_{e_\phi}e_\phi&=-\frac{f'}f e_r,\\
 \nabla_{e_\theta}e_r&=\frac{a'}a e_\theta,
 &\qquad \nabla_{e_\theta}e_\phi&=0.
 \end{aligned}
\]
Since $T(r)=\sqrt{1-y^2}$ and
$(\sqrt{1-y^2})'=-yy'/\sqrt{1-y^2}$, these identities give
\[
 \begin{aligned}
 \nabla_T\nu
 &=\sqrt{1-y^2}\left(y'+\frac{f'}fy\right)e_r
   -y\left(y'+\frac{f'}fy\right)e_\phi \\
 &=\left(y'+\frac{f'}fy\right)T,\\
 \nabla_{e_\theta}\nu
 &=\frac{a'}ay\,e_\theta.
 \end{aligned}
\]
Thus $T$ and $e_\theta$ are principal directions.  With the convention
$k(E)=\langle\nabla_E\nu,E\rangle$, their principal curvatures are
\begin{equation}\label{eq:shapes}
 k_m=y'+\frac{f'}fy=\frac{(fy)'}f,
 \qquad k_\theta=\frac{a'}ay.
\end{equation}
A radial graph over the collapsed $\theta$-disk can be differentiable at
its center only if $\Psi'(0)=0$; otherwise its directional derivative at
the center would depend on $\theta$.
Thus first-order pole regularity requires $\Psi'(0)=0$, equivalently
$y(0)=0$.  We prescribe
$k_m\equiv\kappa>0$.  Then $(fy)'=\kappa f$, so the pole condition gives
\[
 y(r)=\frac{\kappa I(r)}{f(r)}.
\]
First-order matching at $r=r_U$ requires the two graphs to become vertical
as $r$-graphs.  Since $y=\kappa I/f>0$ for $r>0$, this is equivalent to
$y(r_U)=1$, which determines the curvature:
\begin{equation}\label{eq:y-solution}
 \kappa=\frac{f(r_U)}{I(r_U)},\qquad
 y(r)=\frac{f(r_U)I(r)}{I(r_U)f(r)},\qquad y(r_U)=1.
\end{equation}

Since $f>0$ and $f'\leq0$,
\[
 \left(\frac If\right)'=1-\frac{If'}{f^2}>0,
 \qquad y'=\kappa-\frac{f'}fy\geq\kappa>0.
\]
Thus $y$ increases strictly from $0$ to $1$, so $0<y<1$ on
$(0,r_U)$ and the slope identity gives $\Psi'<0$ there.  In particular,
$y'(r_U)>0$.  Near $r_U$,
\[
 1-y(r)^2=2y'(r_U)(r_U-r)+O((r_U-r)^2),
\]
so the integrand below is $O((r_U-r)^{-1/2})$ and hence integrable at the
equator.  At the pole, $I(r)=O(r)$ and $f$ stays positive, so $y(r)=O(r)$
and the same integrand is $O(r)$; in particular, the resulting function
will satisfy $\Psi'(0)=0$.  The slope identity, together with the closing
condition $\Psi(r_U)=0$, therefore defines
\begin{equation}\label{eq:Psi}
 \Psi(r)=\int_r^{r_U}
 \frac{y(\tau)}{f(\tau)\sqrt{1-y(\tau)^2}}\,d\tau.
\end{equation}
This is the advantage of prescribing constant geodesic curvature:
$k_m\equiv\kappa$ integrates once to $fy=\kappa I$, the closing condition
determines $\kappa$, and \eqref{eq:y-solution}--\eqref{eq:Psi} give
explicit formulas for both the normal angle and the opening
$\Delta_{\varepsilon,U}=\Psi(0)$.  Thus
$\Delta_{\varepsilon,U}$ is not chosen independently: it is the total
angular displacement forced by the constant-curvature equation and the
closing normalization $\Psi(r_U)=0$.
Differentiating \eqref{eq:Psi} recovers the slope identity, so the radial
component of the graph normal is the prescribed function $y$; consequently
\eqref{eq:shapes} gives $k_m\equiv\kappa$.
The integrand is positive, so $\Psi$ decreases from
\[
 \Delta_{\varepsilon,U}:=\Psi(0)>0
\]
at the pole to $0$ at the equator.

For every $r$, the retained $\phi$-interval has length
$2\Psi(r)\leq2\Delta_{\varepsilon,U}$.  Hence, whenever
$\Delta_{\varepsilon,U}<\pi$, it has length less than the $2\pi$ period of
the $\phi$-circle and is embedded.  Pulling the ambient metric back to one
graph face gives
\[
 h=\bigl(1+f^2(\Psi')^2\bigr)dr^2+a(r)^2d\theta^2
   =\frac{dr^2}{1-y^2}+a(r)^2d\theta^2.
\]
Thus, after introducing meridional arclength $u$, this becomes
\begin{equation}\label{eq:lensmetric}
 h=du^2+b(u)^2d\theta^2,\qquad
 du=\frac{dr}{\sqrt{1-y^2}},\qquad b(u(r))=a(r),
 \qquad0\leq u\leq\ell.
\end{equation}
The endpoint estimates used for \eqref{eq:Psi} also show that
\[
 \ell=\int_0^{r_U}\frac{dr}{\sqrt{1-y(r)^2}}<\infty.
\]
The full meridian has length $2\ell$.  Differentiating gives
\begin{equation}\label{eq:bderivatives}
 b'=a'\sqrt{1-y^2},\qquad
 b''=a''(1-y^2)-a'yy'.
\end{equation}
Here primes on $b$ mean $d/du$, whereas primes on $a,f,y$ mean $d/dr$.
The right side of \eqref{eq:bderivatives} extends continuously to both
endpoints.  Before the equator, $a'>0$ and $0\leq y<1$, so $b'>0$.  At the
pole,
$a''(0)=0$ and $y(0)=0$, hence $b''(0)=0$.  For $0<u\leq\ell$, the first
term in $b''$ is nonpositive, whereas the second is strictly negative
because $a'>0$, $y>0$, and $y'>0$.  Therefore $b''<0$ away from the pole.
Reflection under $\phi\mapsto-\phi$ gives the same curvature and induced
metric conclusions for the lower graph face.  This proves the lemma.
\end{proof}

\begin{lemma}\label{lem:lens-smooth}
Assume \(\Delta_{\varepsilon,U}<\pi\).
The two graphs $\phi=\pm\Psi(r)$ join to a smooth rotational
two-sphere, and $\Omega_{\varepsilon,U}$ is a smooth three-ball.
\end{lemma}

\begin{proof}
Near the $\theta$-bolt, put
$f_0:=f_\varepsilon^{\mathrm{pole}}$ and
$\beta:=\beta_\varepsilon^{\mathrm{cap}}$.  The exact cap formulas give
\[
 a=r,\qquad f=f_0(1-\beta r^2),\qquad
 I=f_0\left(r-\frac{\beta}{3}r^3\right).
\]
Consequently
\[
 y=\kappa r\frac{1-\beta r^2/3}{1-\beta r^2}
   =\kappa r+O(r^3).
\]
Thus $y$ is smooth and odd in $r$.  Since $f$ and
$\sqrt{1-y^2}$ are smooth and even near $r=0$, the slope identity
\[
 -\Psi'=\frac{y}{f\sqrt{1-y^2}}
\]
shows that $\Psi'$ is smooth and odd.  Hence $\Psi$ is smooth and even;
in particular,
\[
 \Psi(r)=\Delta_{\varepsilon,U}
 -\frac{\kappa}{2f_0}r^2+O(r^4).
\]
Thus each graph is a smooth rotational disk at its pole.

At $r=r_U$, let $s$ be boundary arclength on the upper graph.  Then
\[
 \frac{dr}{ds}=\sqrt{1-y^2},\qquad
 \frac{d\phi}{ds}=-\frac yf.
\]
The condition $y(r_U)=1$ gives $r_s(\ell)=0$, which is only the
first-order tangent-matching condition.
Continue past the equator using the negative square root for $dr/ds$
and the same second equation.  On both sides,
\begin{equation}\label{eq:r-equator-ode}
 \frac{d^2r}{ds^2}=-y(r)y_r(r).
\end{equation}
Equation~\eqref{eq:y-solution} gives a smooth $y$ up to $r_U$.
Thus the right side of \eqref{eq:r-equator-ode} is smooth there.  This is a
smooth autonomous equation with $r(\ell)=r_U$ and $r_s(\ell)=0$.  Reflection
about $s=\ell$ gives another solution with the same initial data, so
uniqueness shows that $r-r_U$ is even in $s-\ell$.  It follows from
$\phi_s=-y(r(s))/f(r(s))$ that $\phi-\phi(\ell)$ is odd.  The two graphs
therefore join smoothly.  The orbit domain under
the $\theta$-rotation is a disk whose $r=0$ segment is the rotation
axis.  Rotating it gives a three-ball with the displayed two-sphere as
boundary.
\end{proof}

\medskip
\noindent\textit{Core-coordinate formulas.}
We now rewrite the lens quantities in the explicit core coordinate $z$;
for a radial quantity $Q$, write $Q(z):=Q(r(z))$.  Let
$z_{\mathrm{pre}}$ be the first unchanged-core point after the cap and
smoothing collar, and define
\begin{equation}\label{eq:H}
 \begin{aligned}
 H_\varepsilon(z)
 &=H_{\mathrm{pre},\varepsilon}+
   \int_{z_{\mathrm{pre}}}^zA_\varepsilon(v)v^{-4}\,dv,\\
 H_{\mathrm{pre},\varepsilon}
 &:=\frac{I(r_{\mathrm{pre}})}{8c_\varepsilon\varepsilon^2}>0,
 & I(z)&=8c_\varepsilon\varepsilon^2H_\varepsilon(z).
 \end{aligned}
\end{equation}
Indeed, \eqref{eq:af} and \eqref{eq:dr} give
\[
 \frac{dI}{dz}
 =f_\varepsilon(r(z))\frac{dr}{dz}
 =\frac{2\varepsilon}{z}
  \bigl(4c_\varepsilon\varepsilon A_\varepsilon(z)z^{-3}\bigr)
 =8c_\varepsilon\varepsilon^2A_\varepsilon(z)z^{-4}.
\]
The collar-width bound \eqref{eq:smoothing-collar-width} and
\eqref{eq:dr} give
$z_{\mathrm{pre}}=2\varepsilon+o(\varepsilon)$.  Indeed,
$dr/dz$ is increasing on the core and
$(dr/dz)(2\varepsilon)=r_\varepsilon^{\mathrm{tr}}/(2\varepsilon^4)$,
whereas
$r_{\mathrm{pre}}-r_\varepsilon^{\mathrm{tr}}=O(\delta)$ and
$\delta\leq(r_\varepsilon^{\mathrm{tr}})^2$.
In both the piecewise and smoothed constructions,
$H_{\mathrm{pre},\varepsilon}=O(e^{-c/\varepsilon^3})$ for some $c>0$, and
$H_\varepsilon'=A_\varepsilon z^{-4}$ exactly on the unchanged core.
Equation~\eqref{eq:y-solution}, together with \eqref{eq:af} and
\eqref{eq:H}, gives
\begin{equation}\label{eq:y-core}
 \kappa\varepsilon
 =\frac{1}{4c_\varepsilon U H_\varepsilon(U)},\qquad
 y(z)=\frac{zH_\varepsilon(z)}{UH_\varepsilon(U)}.
\end{equation}
Using \eqref{eq:dr} to change variables in \eqref{eq:Psi} and
\eqref{eq:lensmetric} then gives
\begin{align}
 \ell&=\ell_{\mathrm{cap},\varepsilon}+
 4c_\varepsilon\varepsilon
 \int_{z_{\mathrm{pre}}}^U\frac{A_\varepsilon(z)z^{-3}}
                    {\sqrt{1-y(z)^2}}\,dz,                       \label{eq:ell}\\
 \Psi(z)&=2c_\varepsilon\int_z^U
 \frac{A_\varepsilon(v)v^{-2}y(v)}
                    {\sqrt{1-y(v)^2}}\,dv
 \quad(z_{\mathrm{pre}}\leq z\leq U),                            \label{eq:Psi-z}\\
 \Delta_{\varepsilon,U}
 &=\Delta_{\mathrm{cap},\varepsilon}
 +2c_\varepsilon\int_{z_{\mathrm{pre}}}^U
 \frac{A_\varepsilon(v)v^{-2}y(v)}
                    {\sqrt{1-y(v)^2}}\,dv.                       \label{eq:Delta-split}
\end{align}
Here $\ell_{\mathrm{cap},\varepsilon}$ and
$\Delta_{\mathrm{cap},\varepsilon}$ are the
positive cap and transition-collar contributions.

We next record the two smallness estimates used below.  Let
$\delta\leq(r_\varepsilon^{\mathrm{tr}})^2$ be the collar half-width from
Lemma~\ref{lem:cap-smoothing}.  For small $\varepsilon$,
$z_{\mathrm{pre}}<U/2$, and the unchanged segment $[U/2,U]$ gives
\begin{equation}\label{eq:H-lower}
 H_\varepsilon(U)\geq
 \int_{U/2}^{U}A_\varepsilon(z)z^{-4}\,dz\geq h_U>0,
\end{equation}
where $h_U$ is independent of $\varepsilon$.  Since $c_\varepsilon$ is
bounded away from zero, \eqref{eq:y-core} and \eqref{eq:H-lower} give
$\kappa=O_U(\varepsilon^{-1})$.  On the cap and collar,
$r_{\mathrm{pre}}=r_\varepsilon^{\mathrm{tr}}+O(\delta)$,
$I=O(r_\varepsilon^{\mathrm{tr}})$, and $f\asymp1$.  Therefore, using
$y=\kappa I/f$,
\begin{equation}\label{eq:y-cap-transition}
 \sup_{0\leq r\leq r_{\mathrm{pre}}}y(r)
 =O_U\!\left(\frac{r_\varepsilon^{\mathrm{tr}}}{\varepsilon}\right)
 =o_U(1).
\end{equation}
Hence $(1-y^2)^{-1/2}=O_U(1)$ there.  Since the collar has $r$-width
$O(\delta)$, its contributions to $I$, $\ell$, and
$\Delta_{\varepsilon,U}$ are respectively $O(\delta)$,
$O_U(\delta)$, and
$O_U(\delta r_\varepsilon^{\mathrm{tr}}/\varepsilon)$.
On the exact cap, the corresponding bounds are
$O(r_\varepsilon^{\mathrm{tr}})$,
$O_U(r_\varepsilon^{\mathrm{tr}})$, and
$O_U((r_\varepsilon^{\mathrm{tr}})^2/\varepsilon)$.  Since
$\delta\leq(r_\varepsilon^{\mathrm{tr}})^2$, equation
\eqref{eq:transition-radius-decay} absorbs every polynomial factor in
$\varepsilon^{-1}$; consequently
\begin{equation}\label{eq:cap-transition-negligible}
 r_{\mathrm{pre}}+\ell_{\mathrm{cap},\varepsilon}
 +\Delta_{\mathrm{cap},\varepsilon}
 =O_U(e^{-c/\varepsilon^3}).
\end{equation}
Thus the cap and collar do not affect any fixed-$U$ endpoint asymptotic.

For later use we also retain the corresponding estimate on a fixed initial
core segment.  If $K>2$ is fixed and $r_K$ denotes the point
$z=K\varepsilon$, then
$A_\varepsilon(K\varepsilon)=O_K(e^{-c_K/\varepsilon^3})$.
Monotonicity of $A_\varepsilon$, together with \eqref{eq:H},
\eqref{eq:dr}, and \eqref{eq:H-lower}, therefore gives
\begin{equation}\label{eq:initial-core-smallness}
 \begin{gathered}
 H_\varepsilon(K\varepsilon)=O_K(e^{-c_K/\varepsilon^3}),
 \qquad r_K=O_K(e^{-c_K/\varepsilon^3}),\\
 \sup_{0\leq r\leq r_K}y(r)=y(r_K)
 =\frac{K\varepsilon H_\varepsilon(K\varepsilon)}
        {UH_\varepsilon(U)}
 =O_{U,K}(e^{-c_K/\varepsilon^3}).
 \end{gathered}
\end{equation}
\medskip

\noindent
The next three lemmas separate the quantitative information according to
its later use.  Lemma~\ref{lem:uniform-endpoint} evaluates the endpoint-layer
integrals, Lemma~\ref{lem:packet-asymptotics} converts them into the global
scales needed for the high-action rescaling, and
Lemma~\ref{lem:normalized-identities} supplies the pointwise boundary
identities used for the angular attachment in Section~3.  We begin with the
uniform endpoint estimates.

\begin{lemma}\label{lem:uniform-endpoint}
Fix $U\in(0,1)$ sufficiently small and put
\begin{equation}\label{eq:endpoint-sigma}
 \sigma=\frac{U^3}{8C_\varepsilon}.
\end{equation}
After $U$ is fixed, take $\varepsilon/U$ sufficiently small.  Uniformly
for $U/2\leq z\leq U$,
\begin{equation}\label{eq:H-asymptotic}
 H_\varepsilon(z)=\frac{A_\varepsilon(z)}{8C_\varepsilon}
 \left(1-\frac{z^3}{48C_\varepsilon}+O(U^6)\right).
\end{equation}
The radial-coordinate integral and the two endpoint integrals in
\eqref{eq:ell} and \eqref{eq:Delta-split} satisfy
\begin{align}
 \int_{z_{\mathrm{pre}}}^U A_\varepsilon(z)z^{-3}\,dz
 &=\frac{UA_\varepsilon(U)}{8C_\varepsilon}
   \bigl(1+O(\sigma)\bigr),                                      \label{eq:endpoint-radial-integral}\\
 \int_{z_{\mathrm{pre}}}^U
 \frac{A_\varepsilon(z)z^{-3}}{\sqrt{1-y(z)^2}}\,dz
 &=\frac{\pi UA_\varepsilon(U)}{16C_\varepsilon}
   \left(1-\frac76\sigma+O(\sigma^2)\right),                    \label{eq:endpoint-length-integral}\\
 \int_{z_{\mathrm{pre}}}^U
 \frac{A_\varepsilon(z)z^{-2}y(z)}{\sqrt{1-y(z)^2}}\,dz
 &=\frac{U^2A_\varepsilon(U)}{8C_\varepsilon}
   \bigl(1+O(\sigma)\bigr).                                     \label{eq:endpoint-angle-integral}
\end{align}
All error constants are independent of $\varepsilon$ and of the endpoint
variable.  The cap and smoothing terms in \eqref{eq:ell} and
\eqref{eq:Delta-split} are $O(\sigma^2)$ relative to the corresponding
leading terms.
\end{lemma}

\begin{proof}
On the unchanged core,
\[
 \delta_\varepsilon
 :=H_\varepsilon(z)-\int_0^zA_\varepsilon(v)v^{-4}\,dv
\]
is constant.  Equations \eqref{eq:H}, \eqref{eq:A}, and
\eqref{eq:cap-transition-negligible} show that this constant and all
non-core contributions are super-exponentially small.  Thus, after $U$ is
fixed, we may decrease $\varepsilon/U$ until
\begin{equation}\label{eq:endpoint-cap-error}
 \frac{|\delta_\varepsilon|}{A_\varepsilon(U/2)}
 +\frac{r_{\mathrm{pre}}+\ell_{\mathrm{cap},\varepsilon}}
       {\varepsilon UA_\varepsilon(U)}
 +\frac{\Delta_{\mathrm{cap},\varepsilon}}{U^2A_\varepsilon(U)}
 \leq\sigma^2.
\end{equation}
The denominators are the leading scales of the quantities concerned, so it
remains to estimate the explicit core integrals.

Since
\[
 \frac{A_\varepsilon'}{A_\varepsilon}
 =z^{-4}\left(8C_\varepsilon+\frac{z^3}{6}\right),
\]
setting $m_\varepsilon(z)=(8C_\varepsilon+z^3/6)^{-1}$ and integrating by
parts gives
\[
 \int_0^zA_\varepsilon(v)v^{-4}\,dv
 =m_\varepsilon(z)A_\varepsilon(z)
  -\int_0^zm_\varepsilon'(v)A_\varepsilon(v)\,dv.
\]
There is no lower boundary term because $A_\varepsilon(v)$ decays like
$e^{-c/v^3}$ as $v\downarrow0$.  Since
$|m_\varepsilon'(v)|\leq Cv^2$ and
$A_\varepsilon(v)\leq A_\varepsilon(z)e^{-c(z-v)/z^4}$ on $[z/2,z]$,
the contribution of this interval is at most
\[
 Cz^2A_\varepsilon(z)\int_{z/2}^ze^{-c(z-v)/z^4}\,dv
 \leq Cz^6A_\varepsilon(z).
\]
For $0<v\leq z/2$, the explicit formula \eqref{eq:A} gives
\[
 \frac{A_\varepsilon(v)}{A_\varepsilon(z)}
 =\left(\frac vz\right)^{1/6}
 \exp\!\left\{\frac{8C_\varepsilon}{3}(z^{-3}-v^{-3})\right\}
 \leq e^{-c/z^3}.
\]
Consequently,
\[
 \int_0^{z/2}|m_\varepsilon'(v)|A_\varepsilon(v)\,dv
 \leq Cz^3e^{-c/z^3}A_\varepsilon(z)
 \leq Cz^6A_\varepsilon(z).
\]
The last inequality uses $e^{-c/z^3}=O(z^3)$.
Combining the two intervals gives
$\int_0^z|m_\varepsilon'|A_\varepsilon\,dv
\leq Cz^6A_\varepsilon(z)$.
Finally,
\[
 m_\varepsilon(z)=\frac1{8C_\varepsilon}
 \left(1-\frac{z^3}{48C_\varepsilon}+O(z^6)\right).
\]
Together with \eqref{eq:endpoint-cap-error}, these estimates prove
\eqref{eq:H-asymptotic}.

Set
\begin{equation}\label{eq:endpoint-variable}
 z=U(1-\sigma\xi),\qquad t:=1-\sigma\xi.
\end{equation}
This is the natural endpoint scale because
\[
 -\frac{d}{d\xi}\log A_\varepsilon(Ut)
 =t^{-4}+\frac{\sigma}{6t}.
\]
Also, \eqref{eq:H-asymptotic} and $H_\varepsilon'=A_\varepsilon z^{-4}$
give
\[
 \frac{H_\varepsilon'}{H_\varepsilon}
 =8C_\varepsilon z^{-4}+\frac1{6z}+O(U^2).
\]
Using \eqref{eq:y-core}, we therefore obtain, uniformly for
$U/2\leq z\leq U$,
\[
 -\frac{d}{d\xi}\log y(Ut)
 =t^{-4}+\frac{7\sigma}{6t}+O(\sigma^2).
\]
Both negative logarithmic derivatives are therefore bounded below by a
constant $c_0>0$.  Both normalized functions equal $1$ at $\xi=0$, so
integration gives
\[
 \frac{A_\varepsilon(z)}{A_\varepsilon(U)}\leq e^{-c_0\xi},
 \qquad y(z)\leq e^{-c_0\xi}.
\]
\begin{equation}\label{eq:endpoint-majorant}
 \frac{A_\varepsilon(z)}{A_\varepsilon(U)}+y(z)
 \leq 2e^{-c_0\xi},\qquad
 1-y(z)^2\geq c\min\{\xi,1\}.
\end{equation}
Thus every normalized integrand in
\eqref{eq:endpoint-radial-integral}--
\eqref{eq:endpoint-angle-integral} is bounded by
\[
 Ce^{-c\xi}\bigl(1+\xi^{-1/2}\bigr),
\]
which is independent of $\varepsilon$ and integrable at both endpoints.
To control the lower part, put $b=U/2$.  Since
$A_\varepsilon'/A_\varepsilon\geq cz^{-4}$, for $p=2,3$,
\[
 \int_{z_{\mathrm{pre}}}^{b}A_\varepsilon(z)z^{-p}\,dz
 \leq Cb^{4-p}\int_{z_{\mathrm{pre}}}^{b}A_\varepsilon'(z)\,dz
 \leq Cb^{4-p}A_\varepsilon(b).
\]
The same differential inequality on $[b,U]$ gives
\[
 \frac{A_\varepsilon(b)}{A_\varepsilon(U)}
 \leq\exp\!\left(-c\int_b^Us^{-4}\,ds\right)
 \leq e^{-c/\sigma}.
\]
At $b$, one has $\xi=(2\sigma)^{-1}$, so
\eqref{eq:endpoint-majorant} and the monotonicity of $y$ give
$y(z)\leq y(b)\leq e^{-c/\sigma}$ for $z\leq b$; hence
$(1-y^2)^{-1/2}=O(1)$ there.  Thus the lower parts of the radial and
length integrals are $O(UA_\varepsilon(U)e^{-c/\sigma})$, while that of
the angular integral is $O(U^2A_\varepsilon(U)e^{-c/\sigma})$.
The substitution \eqref{eq:endpoint-variable} maps the remaining interval
$[U/2,U]$ exactly onto $[0,(2\sigma)^{-1}]$.

Integrating the two logarithmic-derivative formulas from $\xi=0$ and using
$y(U)=1$ now gives, uniformly for
$0\leq\xi\leq\sigma^{-1/3}$,
\begin{align}
 \frac{A_\varepsilon(z)}{A_\varepsilon(U)}
 &=e^{-\xi}\left[1-\sigma\left(2\xi^2+\frac\xi6\right)
   +O\!\left(\sigma^2\xi(1+\xi)^3\right)\right],               \label{eq:A-layer}\\
 y(z)
 &=e^{-\xi}\left[1-\sigma\left(2\xi^2+\frac{7\xi}{6}\right)
   +O\!\left(\sigma^2\xi(1+\xi)^3\right)\right].              \label{eq:y-layer}
\end{align}
The factor $\xi$ in each remainder comes from integrating the corresponding
logarithmic-derivative error over $[0,\xi]$; thus the error vanishes at the
normalized endpoint, where
$A_\varepsilon(U)/A_\varepsilon(U)=y(U)=1$.  For the expansion of $y$,
this vanishing is needed because $1-e^{-2\xi}\asymp\xi$ near $\xi=0$.
For the length integral, put
\[
 q(\xi)=e^{-2\xi},\qquad
 W(\xi)=\frac{e^{-\xi}}{\sqrt{1-q(\xi)}},\qquad
 B(\xi)=2\xi^2+\frac{7\xi}{6}.
\]
Since $t^{-3}=1+3\sigma\xi+O(\sigma^2\xi^2)$, equation
\eqref{eq:y-layer} gives
\[
 1-y^2
 =1-q+2\sigma qB
 +O\!\left(\sigma^2q\xi(1+\xi)^3\right).
\]
Using $1-q\asymp\min\{\xi,1\}$ and the expansion of
$(1+s)^{-1/2}$, equations \eqref{eq:A-layer}--\eqref{eq:y-layer} yield
\[
 \frac{A_\varepsilon(z)}{A_\varepsilon(U)}
 \frac{t^{-3}}{\sqrt{1-y(z)^2}}
 =
 W\left[1+\sigma\left(
 3\xi-\left(2\xi^2+\frac{\xi}{6}\right)
 -\frac{qB}{1-q}\right)\right]+\mathcal R_\sigma(\xi).
\]
Because $B=2\xi^2+7\xi/6$, the first-order coefficient equals
\[
 3\xi-\left(2\xi^2+\frac{\xi}{6}\right)-\frac{qB}{1-q}
 =4\xi-\frac{B}{1-q}.
\]
On $[0,\sigma^{-1/3}]$, the function $qB/(1-q)$ is bounded.  The
first-order coefficients above are at most quadratic in $\xi$, while the
remainders in \eqref{eq:A-layer}--\eqref{eq:y-layer} are
$O(\sigma^2(1+\xi)^4)$.  Moreover,
$\sigma(1+\xi)^2=O(\sigma^{1/3})$ on this interval, so higher products
are absorbed into the same bound.  Consequently
\[
 |\mathcal R_\sigma(\xi)|
 \leq C\sigma^2e^{-c\xi}(1+\xi)^4(1+\xi^{-1/2}).
\]
The displayed bound is integrable, so the remainder contributes
$O(\sigma^2)$.  On the rest of the actual interval,
$[\sigma^{-1/3},(2\sigma)^{-1}]$,
\eqref{eq:endpoint-majorant} makes the contribution smaller than every
power of $\sigma$.  The tails of the model profiles beyond
$(2\sigma)^{-1}$ are $O(e^{-c/\sigma})$.  We may therefore extend only
the model integrals, not the original change of variables, to
$[0,\infty)$ with an $O(\sigma^2)$ error.
The change of variables in the length integral now gives
\[
 \begin{aligned}
 &\int_{z_{\mathrm{pre}}}^U
 \frac{A_\varepsilon(z)z^{-3}}{\sqrt{1-y(z)^2}}\,dz\\
 &\quad=\frac{UA_\varepsilon(U)}{8C_\varepsilon}
 \left\{\int_0^\infty W\,d\xi
 +\sigma\int_0^\infty
 W\left(4\xi-\frac{B}{1-e^{-2\xi}}\right)d\xi
 +O(\sigma^2)\right\}.
 \end{aligned}
\]
The substitution $u=e^{-\xi}$ gives
$\int_0^\infty W\,d\xi=\pi/2$, and differentiation gives
$W'=-W/(1-e^{-2\xi})$.  Moreover,
$B'=4\xi+7/6$.  Since $B=O(\xi)$ and $W=O(\xi^{-1/2})$ at $0$, while
$W$ decays exponentially at infinity, $B(\xi)W(\xi)\to0$ at both
endpoints.  Therefore
\[
 \begin{aligned}
 \int_0^\infty\bigl(4\xi W+BW'\bigr)\,d\xi
 &=[BW]_0^\infty+\int_0^\infty(4\xi-B')W\,d\xi\\
 &=-\frac76\int_0^\infty W\,d\xi=-\frac{7\pi}{12}.
 \end{aligned}
\]
This proves \eqref{eq:endpoint-length-integral}.  The same change of
variables has prefactors
\[
 U\sigma U^{-3}=\frac{U}{8C_\varepsilon},\qquad
 U\sigma U^{-2}=\frac{U^2}{8C_\varepsilon}.
\]
The radial and angular leading profiles are, respectively,
\[
 e^{-\xi}\qquad\text{and}\qquad
 \frac{e^{-2\xi}}{\sqrt{1-e^{-2\xi}}}.
\]
Their integrals are
\[
 \int_0^\infty e^{-\xi}\,d\xi=1,\qquad
 \int_0^\infty\frac{e^{-2\xi}}{\sqrt{1-e^{-2\xi}}}\,d\xi=1,
\]
where the second identity follows again from $u=e^{-\xi}$.  These identities
prove \eqref{eq:endpoint-radial-integral} and
\eqref{eq:endpoint-angle-integral}.  The cap assertion follows from
\eqref{eq:endpoint-cap-error}.
\end{proof}

\medskip
\noindent
We now convert the uniform endpoint estimates into global lens data.  The
length, radial size, and opening angle will determine the normalized waist,
depth, and scalar-action scale.  For the angular attachment, it is convenient
to normalize the total meridional bending.  Here
$\kappa=\kappa_{\varepsilon,U}$ is the value determined by the closing
condition $y(r_U)=1$ in Lemma~\ref{lem:lens-ode}, not an independently fixed
parameter.  Set
\[
 \mathfrak m:=\frac{2\kappa\ell}{\pi}.
\]

\begin{lemma}
\label{lem:packet-asymptotics}
Fix $U\in(0,1)$ sufficiently small and then let
$\varepsilon/U\to0$.  The lens
data satisfy
\begin{align*}
 \ell&=\frac{\pi c_\varepsilon\varepsilon U A_\varepsilon(U)}
                   {4C_\varepsilon}
 \left(1-\frac{7U^3}{48C_\varepsilon}+O(U^6)\right),\\
 r_U&=\frac{c_\varepsilon\varepsilon U A_\varepsilon(U)}
                   {2C_\varepsilon}\bigl(1+O(U^3)\bigr),\\
 \Delta_{\varepsilon,U}
 &=\frac{c_\varepsilon U^2A_\varepsilon(U)}
                   {4C_\varepsilon}\bigl(1+O(U^3)\bigr),\\
 \mathfrak m&=1-\frac{U^3}{8C_\varepsilon}+O(U^6).
\end{align*}
In particular, $\Delta_{\varepsilon,U}<\pi$ when $U$ and then
$\varepsilon/U$ are sufficiently small.
\end{lemma}

\begin{proof}
Let $\sigma$ be as in \eqref{eq:endpoint-sigma}.  Substitution of
the uniform endpoint estimates of Lemma~\ref{lem:uniform-endpoint}, in
particular \eqref{eq:endpoint-length-integral}, in \eqref{eq:ell}, including the
cap estimate \eqref{eq:endpoint-cap-error}, gives
\begin{align}
 \ell&=\frac{\pi c_\varepsilon\varepsilon U A_\varepsilon(U)}
 {4C_\varepsilon}
 \left(1-\frac76\sigma+O(\sigma^2)\right).                       \label{eq:ell-asymp}
\end{align}
Equation \eqref{eq:dr}, the radial estimate
\eqref{eq:endpoint-radial-integral}, and the $r_{\mathrm{pre}}$ term in
\eqref{eq:endpoint-cap-error} give the first formula below.  Equations
\eqref{eq:Delta-split}, \eqref{eq:endpoint-angle-integral}, and the
$\Delta_{\mathrm{cap},\varepsilon}$ term in
\eqref{eq:endpoint-cap-error} give the second:
\begin{align}
 r_U&=\frac{c_\varepsilon\varepsilon U A_\varepsilon(U)}
                   {2C_\varepsilon}\bigl(1+O(U^3)\bigr),            \label{eq:rU-asymp}\\
 \Delta_{\varepsilon,U}
 &=\frac{c_\varepsilon U^2A_\varepsilon(U)}
                   {4C_\varepsilon}\bigl(1+O(U^3)\bigr).            \label{eq:Delta-asymp}
\end{align}
At $z=U$, equation \eqref{eq:H-asymptotic} reads
\[
 H_\varepsilon(U)=\frac{A_\varepsilon(U)}{8C_\varepsilon}
 \left(1-\frac\sigma6+O(\sigma^2)\right).
\]
Since $\kappa\varepsilon=(4c_\varepsilon UH_\varepsilon(U))^{-1}$,
\begin{equation}\label{eq:charge-asymp}
 \mathfrak m=\frac{2\kappa\ell}{\pi}
 =\frac{\ell}{2\pi c_\varepsilon\varepsilon UH_\varepsilon(U)}
 =1-\sigma+O(\sigma^2)
 =1-\frac{U^3}{8C_\varepsilon}+O(U^6).
\end{equation}
In particular, after first fixing \(U\) sufficiently small and then taking
\(\varepsilon/U\) sufficiently small, \eqref{eq:Delta-asymp} gives
\(\Delta_{\varepsilon,U}<\pi\), as required in
Lemma~\ref{lem:lens-smooth}.
\end{proof}

\medskip
\noindent
The preceding asymptotics control the global size and normalized bending of
the lens.  For the pointwise concavity and seam comparisons in Section~3,
we must also relate meridional arclength to the turning of the outward normal.
On either pole-to-equator graph face, orient meridional arclength $u$ from
$u=0$ at the pole to $u=\ell$ at the equator, and define
\[
 \alpha(u):=\arcsin y(r(u)),\qquad
 \chi(u):=\frac{\pi u}{2\ell}.
\]
Thus
\[
 \nu=\sin\alpha\,e_r+\cos\alpha\,e_\phi,
 \qquad
 \alpha(0)=\chi(0)=0,
 \qquad
 \alpha(\ell)=\chi(\ell)=\frac\pi2.
\]
The lens equation gives $d\alpha/du=dy/dr>0$, so $\alpha$ is a valid
meridional coordinate.  The variable $\chi$ is the normalized meridional
arclength; at this stage it does not refer to an ambient angular sphere.

\begin{lemma}\label{lem:normalized-identities}
On the explicit core define
\begin{equation}\label{eq:T-E}
 T(z)=\frac{z^3H_\varepsilon(z)}{A_\varepsilon(z)},\qquad
 E_{\mathrm{core}}(z)=\frac1{1+T(z)},\qquad
 \Xi(z)=
 \frac{1-z^3/(24C_\varepsilon)}
      {1+z^3/(48C_\varepsilon)}.
\end{equation}
On the entire pole-to-equator graph face define
\[
 \widehat E(\alpha):=\kappa\frac{du}{d\alpha}.
\]
Then $\widehat E(\alpha)=E_{\mathrm{core}}(z(\alpha))$ on the explicit
core, and on the entire graph face
\begin{equation}
 \frac{d\chi}{d\alpha}=\frac{\widehat E(\alpha)}{\mathfrak m}.
                                                               \label{eq:chi-alpha}
\end{equation}
Finally, at interior points of the explicit core,
\begin{equation}
 -\frac{2\ell}{\pi}\frac{b''}{b'}
 =\mathfrak m\left((1+T)\tan\alpha+T\Xi\cot\alpha\right).
                                                               \label{eq:exact-L}
\end{equation}
After $U$ is fixed sufficiently small and $\varepsilon/U$ is sufficiently
small, one has, for $0<\alpha\leq\pi/2$,
\begin{equation}\label{eq:M-refined}
 \widehat E(\alpha)
 =E_{\mathrm{core}}(U)+O\!\left(U^6|\log\sin\alpha|\right),\qquad
 \mathfrak m=E_{\mathrm{core}}(U)+O(U^6).
\end{equation}
\end{lemma}

\begin{proof}
On the explicit core write $z=z(\alpha)$.  Since
$\sin\alpha=y$, $\cos\alpha=\sqrt{1-y^2}$, and
$du/dr=(1-y^2)^{-1/2}$, equations \eqref{eq:dr}, \eqref{eq:H}, and
\eqref{eq:y-core} give
\[
 \frac{du}{dz}=\frac{4c_\varepsilon\varepsilon
 A_\varepsilon z^{-3}}{\cos\alpha},\qquad
 y_z=\frac{A_\varepsilon z^{-3}(1+T)}
 {UH_\varepsilon(U)},
 \qquad
 \frac{d\alpha}{dz}
 =\frac{A_\varepsilon z^{-3}(1+T)}
 {UH_\varepsilon(U)\cos\alpha}.
\]
Taking the quotient of the first and third expressions and using
$4c_\varepsilon\varepsilon U H_\varepsilon(U)=\kappa^{-1}$ from
\eqref{eq:y-core} gives
\[
 \frac{du}{d\alpha}
 =4c_\varepsilon\varepsilon U H_\varepsilon(U)
 E_{\mathrm{core}}(z(\alpha))
 =\kappa^{-1}E_{\mathrm{core}}(z(\alpha)).
\]
Thus $\widehat E(\alpha)=E_{\mathrm{core}}(z(\alpha))$ there.  The strict
monotonicity of $\alpha$ recorded above makes $\widehat E$ well defined
also on the cap and smoothing collar.  Direct differentiation of
$\chi=\pi u/(2\ell)$ gives
\[
 \frac{d\chi}{d\alpha}
 =\frac{\pi}{2\ell}\frac{du}{d\alpha}
 =\frac{\widehat E(\alpha)}{\mathfrak m},
\]
which proves \eqref{eq:chi-alpha}.  Integrating
$\widehat E\,d\alpha=\kappa\,du$ from the pole to the equator also gives
\[
 \mathfrak m=\frac2\pi\int_0^{\pi/2}\widehat E(\beta)\,d\beta,
\]
which will be used at the end.

The reciprocal quotient above gives
$y_r=d\alpha/du=\kappa(1+T)$.  Equation \eqref{eq:bderivatives} therefore
implies
\[
 -\frac1\kappa\frac{b''}{b'}
 =\frac{y_r}{\kappa}\tan\alpha
  -\frac{a''}{\kappa a'}\cos\alpha.
\]
To identify the second term, equations \eqref{eq:af},
\eqref{eq:aprime}, \eqref{eq:asecond}, and \eqref{eq:y-core} give
\[
\begin{aligned}
 a''&=-\frac{C_\varepsilon\varepsilon^2z}{2a}
 \left(1-\frac{z^3}{24C_\varepsilon}\right),\\
 a'&=\frac{2C_\varepsilon\varepsilon}{z}
 \left(1+\frac{z^3}{48C_\varepsilon}\right),\\
 \frac{y}{\kappa}&=4c_\varepsilon\varepsilon zH_\varepsilon(z).
\end{aligned}
\]
Since $a=c_\varepsilon\varepsilon^2A_\varepsilon$, these formulas yield
\[
 -\frac{a''y}{\kappa a'}=T\Xi.
\]
Using $y=\sin\alpha$ in the preceding expression for $b''/b'$ now gives
\[
 -\frac1\kappa\frac{b''}{b'}
 =(1+T)\tan\alpha+T\Xi\cot\alpha.
\]
Multiplication by $2\kappa\ell/\pi=\mathfrak m$ proves
\eqref{eq:exact-L}.

It remains to estimate $\widehat E$.  Define
$H_\varepsilon^{\mathrm{core}}(z)
:=\int_0^zA_\varepsilon(v)v^{-4}\,dv$.  Since
\[
 A_\varepsilon(v)v^{-4}
 =\frac{A_\varepsilon'(v)}
 {8C_\varepsilon+v^3/6}
\]
and the denominator is increasing, while $A_\varepsilon(0)=0$, one has
\[
 \frac{A_\varepsilon(z)}{8C_\varepsilon+z^3/6}
 <H_\varepsilon^{\mathrm{core}}(z)
 <\frac{A_\varepsilon(z)}{8C_\varepsilon}.
\]
The difference
$H_\varepsilon-H_\varepsilon^{\mathrm{core}}$ is independent of $z$ on
the unchanged core and is super-exponentially small by \eqref{eq:H}.
After fixing $U$ and reducing $\varepsilon/U$, this correction is absorbed
by the strict margins above.  Using $(1+s)^{-1}\geq1-s$ gives, uniformly
for $U/2\leq z\leq U$,
\[
 \frac{z^3}{8C_\varepsilon}
 \left(1-\frac{z^3}{48C_\varepsilon}\right)
 \leq T(z)\leq\frac{z^3}{8C_\varepsilon}.
\]
In particular, $T\asymp z^3$ on this interval.  Since
$\sin\alpha=y=zH_\varepsilon(z)/(UH_\varepsilon(U))$ and
$T=z^3H_\varepsilon/A_\varepsilon$, direct differentiation gives the
core identities
\begin{equation}\label{eq:core-differential-identities}
 \frac{d}{dz}\log\sin\alpha
 =\frac1z\left(1+\frac1T\right),
 \qquad
 zT_z=1+\frac{17}{6}T-\frac{8C_\varepsilon T}{z^3}.
\end{equation}
On $U/2\leq z\leq U$, the right-hand side of the first identity is at
least $cU^{-4}$.  Integration from $z$ to $U$, using $y(U)=1$, therefore
gives
\[
 0\leq U-z(\alpha)
 \leq CU^4|\log\sin\alpha|.
\]
The second identity in \eqref{eq:core-differential-identities} and the
two-sided estimate for $T$ show that
\[
 0\leq1-\frac{8C_\varepsilon T}{z^3}\leq Cz^3,
 \qquad T=O(z^3).
\]
Thus $T_z=O(z^2)=O(U^2)$ and
$(E_{\mathrm{core}})_z=-T_z/(1+T)^2=O(U^2)$ on this endpoint region.
The mean-value theorem now gives
\[
 |E_{\mathrm{core}}(z(\alpha))-E_{\mathrm{core}}(U)|
 \leq CU^6|\log\sin\alpha|.
\]

On the entire graph face, $y=\kappa I/f$ and the lens equation give
\[
 \frac{d\alpha}{du}=y_r
 =\kappa\left(1-\frac{f_rI}{f^2}\right),\qquad
 \widehat E
 =\left(1-\frac{f_rI}{f^2}\right)^{-1}.
\]

To justify the estimate uniformly on the entire core, note from
\eqref{eq:H} that
\[
 H_\varepsilon(z)
 \leq H_{\mathrm{pre},\varepsilon}
      +H_\varepsilon^{\mathrm{core}}(z).
\]
The cap and collar estimates give
$H_{\mathrm{pre},\varepsilon}=O(A_\varepsilon(2\varepsilon))$,
while monotonicity of $A_\varepsilon$ and the upper bound above for
$H_\varepsilon^{\mathrm{core}}$ give
$H_\varepsilon(z)/A_\varepsilon(z)=O(1)$ for
$z_{\mathrm{pre}}\leq z\leq U$.  Consequently
$T(z)=z^3H_\varepsilon(z)/A_\varepsilon(z)=O(U^3)$ throughout the core.
Thus, on the core, $-f_rI/f^2=T=O(U^3)$; on the cap and smoothing collar,
the cap formulas and the smoothing first-jet estimates give
$-f_rI/f^2=O(\varepsilon^3)=O(U^3)$.  Hence
$\widehat E=1+O(U^3)$ on the entire graph face.  Moreover,
$T(U)=O(U^3)$ implies
\[
 E_{\mathrm{core}}(U)=\frac1{1+T(U)}=1+O(U^3).
\]
The triangle inequality therefore gives
\[
 |\widehat E-E_{\mathrm{core}}(U)|
 \leq |\widehat E-1|+|1-E_{\mathrm{core}}(U)|
 =O(U^3)
\]
uniformly on the entire graph face.  On the complement
of the endpoint core segment, \eqref{eq:endpoint-majorant}, evaluated at
$z=U/2$, and the monotonicity of $y$ give
\[
 \sin\alpha=y\leq e^{-c/U^3},
 \qquad
 |\log\sin\alpha|\geq cU^{-3}.
\]
Thus the global $O(U^3)$ error is bounded by
$CU^6|\log\sin\alpha|$ there.  Together with the endpoint estimate, this
proves the first assertion of \eqref{eq:M-refined}.  Finally, the average
formula for $\mathfrak m$ and the first assertion give
\[
\begin{aligned}
 \left|\mathfrak m-E_{\mathrm{core}}(U)\right|
 &\leq \frac{2CU^6}{\pi}
 \int_0^{\pi/2}|\log\sin\alpha|\,d\alpha \\
 &=CU^6\log 2=O(U^6).
\end{aligned}
\]
This proves the second assertion of \eqref{eq:M-refined}.
\end{proof}

\medskip
\noindent
With these pointwise identities recorded for the later angular matching,
we return to the packet itself and estimate its scalar action and the
localization of that action away from the future gluing collar.

\begin{lemma}\label{lem:action}
Fix $U\in(0,1)$ sufficiently small, as in
Lemma~\ref{lem:uniform-endpoint}.  For all sufficiently small
$\varepsilon>0$ for which $\Delta_{\varepsilon,U}<\pi$, one has
\begin{equation}\label{eq:action-limit}
 \int_{\Omega_{\varepsilon,U}}\Scal_{g_\varepsilon}\,dV
 \geq4\pi f_\varepsilon(0)a_\varepsilon'(0)\Delta_{\varepsilon,U}
       -O_U\!\left(\varepsilon^2A_\varepsilon(U)\right).
\end{equation}
Moreover, $\Delta_{\varepsilon,U}\to\Delta_U>0$ as
$\varepsilon\downarrow0$ for some constant $\Delta_U$.  The pole flux
satisfies $f_\varepsilon(0)a_\varepsilon'(0)=1+\varepsilon^3/2$ by
\eqref{eq:pole-flux}.
\end{lemma}

\begin{proof}
Because $\Ric>0$, one has
$\Scal\geq\Ric(e_\theta,e_\theta)$, so it suffices to integrate the
$\theta$-eigenvalue.  From
\eqref{eq:sectional},
\[
 \Ric(e_\theta,e_\theta)=-\frac{(fa')'}{af},\qquad
 dV=af\,dr\,d\theta\,d\phi.
\]
The angular range at $r$ has length $2\Psi(r)$, hence
we first integrate by parts on $[\eta,r_U]$ and then let
$\eta\downarrow0$.  Using
$\Psi(0)=\Delta_{\varepsilon,U}$, $\Psi(r_U)=0$, and the pole-flux
identity \eqref{eq:pole-flux}, we obtain
\begin{align}
 \int_{\Omega_{\varepsilon,U}}\Ric(e_\theta,e_\theta)\,dV
 &=4\pi\int_0^{r_U}\Psi(r)\bigl(-(fa')'(r)\bigr)\,dr \notag\\
 &=4\pi f_\varepsilon(0)a_\varepsilon'(0)\Delta_{\varepsilon,U}
   +4\pi\int_0^{r_U}\Psi'(r)f(r)a'(r)\,dr.                        \label{eq:flux-lens}
\end{align}
The final integral equals
\[
 -4\pi\int_0^{r_U}\frac{a'(r)y(r)}
              {\sqrt{1-y(r)^2}}\,dr.
\]
On the unchanged core, equations
\eqref{eq:dr}, \eqref{eq:aprime}, and \eqref{eq:y-core} rewrite this term as
\[
 -16\pi c_\varepsilon\varepsilon^2
 \int_{z_{\mathrm{pre}}}^{U}A_\varepsilon(z)
 \left(2C_\varepsilon z^{-4}+\frac{z^{-1}}{24}\right)
 \frac{y(z)}{\sqrt{1-y(z)^2}}\,dz.
\]
Since $H_\varepsilon$ is increasing, \eqref{eq:y-core} gives
\[
 \frac{y(z)}z=\frac{H_\varepsilon(z)}{UH_\varepsilon(U)}\leq\frac1U.
\]
Also $z^{-1}\leq Uz^{-2}$ for $z\leq U$.  Therefore
\begin{align*}
 &\int_{z_{\mathrm{pre}}}^{U}A_\varepsilon(z)
 \left(2C_\varepsilon z^{-4}+\frac{z^{-1}}{24}\right)
 \frac{y(z)}{\sqrt{1-y(z)^2}}\,dz\\
 &\quad\leq\frac{2C_\varepsilon}{U}
 \int_{z_{\mathrm{pre}}}^{U}
 \frac{A_\varepsilon(z)z^{-3}}{\sqrt{1-y(z)^2}}\,dz
 +\frac{U}{24}\int_{z_{\mathrm{pre}}}^{U}
 \frac{A_\varepsilon(z)z^{-2}y(z)}{\sqrt{1-y(z)^2}}\,dz\\
 &=O_U\!\left(A_\varepsilon(U)\right)
\end{align*}
by \eqref{eq:endpoint-length-integral} and
\eqref{eq:endpoint-angle-integral}.  Thus the core correction is
$O_U(\varepsilon^2A_\varepsilon(U))$.

On the cap and transition collar, $0<a'\leq1$.  Equations
\eqref{eq:y-cap-transition} and \eqref{eq:cap-transition-negligible}
give
\[
 \left|\int_0^{r_{\mathrm{pre}}}
 \frac{a'(r)y(r)}{\sqrt{1-y(r)^2}}\,dr\right|
 =O_U\!\left(
 \frac{r_{\mathrm{pre}}r_\varepsilon^{\mathrm{tr}}}{\varepsilon}\right)
 =O_U(e^{-c/\varepsilon^3}).
\]
Because $A_\varepsilon(U)\asymp_U1$, this term is absorbed into
$O_U(\varepsilon^2A_\varepsilon(U))$.  This proves
\eqref{eq:action-limit}.

For the asserted limit, put
\[
 \begin{gathered}
 c_{\mathrm{lim}}=\sqrt{\frac{55}{24}},\qquad
 \mathcal A_{\mathrm{lim}}(z)
 =z^{1/6}\exp\!\left\{\frac83(1-z^{-3})\right\},\\
 H_{\mathrm{lim}}(z)
 =\int_0^z\mathcal A_{\mathrm{lim}}(v)v^{-4}\,dv,\qquad
 y_{\mathrm{lim}}(z)
 =\frac{zH_{\mathrm{lim}}(z)}{UH_{\mathrm{lim}}(U)}.
 \end{gathered}
\]
The facts $z_{\mathrm{pre}}\to0$ and
$H_{\mathrm{pre},\varepsilon}=O(e^{-c/\varepsilon^3})$ give
$A_\varepsilon\to\mathcal A_{\mathrm{lim}}$,
$H_\varepsilon\to H_{\mathrm{lim}}$, and $y\to y_{\mathrm{lim}}$
uniformly on compact subintervals of $(0,U]$.  On $[U/2,U]$,
\eqref{eq:endpoint-majorant} controls the square-root singularity at $U$.
At $U/2$ the same estimate makes $y$ uniformly smaller than $1/2$ once
$U$ is fixed sufficiently small; monotonicity of $y$ and \eqref{eq:A}
then control the lower core.  Together these estimates give the
integrable bound
\[
 \frac{A_\varepsilon(z)z^{-2}y(z)}{\sqrt{1-y(z)^2}}
 \leq C_Uz^{-11/6}e^{-cz^{-3}}(U-z)^{-1/2}
 \qquad(z_{\mathrm{pre}}\leq z<U).
\]
Extend this integrand by zero on $(0,z_{\mathrm{pre}})$.  Since
$\Delta_{\mathrm{cap},\varepsilon}\to0$, dominated convergence gives
\[
 \Delta_{\varepsilon,U}\longrightarrow
 \Delta_U:=2c_{\mathrm{lim}}\int_0^U
 \frac{\mathcal A_{\mathrm{lim}}(z)z^{-2}y_{\mathrm{lim}}(z)}
      {\sqrt{1-y_{\mathrm{lim}}(z)^2}}\,dz>0.
\]
The same $L^1$ convergence in \eqref{eq:Psi-z} gives local uniform
convergence of $\Psi$ on $(0,U)$ to the positive function obtained by
replacing $c_\varepsilon,A_\varepsilon$, and $y$ by their limiting
counterparts.
This proves the asserted limit and completes the proof.
\end{proof}

\begin{lemma}
\label{lem:depth}
If $\Delta_{\varepsilon,U}<\pi$, then every point of
$\Omega_{\varepsilon,U}$ has distance at most $r_U$ from
$\partial\Omega_{\varepsilon,U}$.
\end{lemma}

\begin{proof}
At $(r,\theta,\phi)$ increase $r$ while holding $(\theta,\phi)$ fixed.
Since $\Psi$ decreases continuously from $\Delta_{\varepsilon,U}$ to zero,
there is $r_*\in[r,r_U]$ with $\Psi(r_*)=|\phi|$, and the curve first
meets the boundary there.  Its length is
$r_*-r\leq r_U$; in particular, the path does not traverse the
$f\,d\phi$ direction.
\end{proof}

\begin{lemma}
\label{lem:protected-set}
Fix $U\in(0,1)$ sufficiently small, as in Lemma~\ref{lem:action}, and
then fix $K>6$.  For all sufficiently small $\varepsilon$ satisfying
$K\varepsilon<U$ and $\Delta_{\varepsilon,U}<\pi$, let $r_K$ be the
unchanged-core point $z=K\varepsilon$ and set
\begin{equation}\label{eq:protected-set}
 \mathcal E_{\varepsilon,U,K}
 =\left\{0\leq r\leq r_K,\quad
          |\phi|\leq\frac{\Delta_{\varepsilon,U}}3\right\}.
\end{equation}
Then
$\mathcal E_{\varepsilon,U,K}\Subset
\Omega_{\varepsilon,U}$ and,
for a constant $c_0>0$ independent of $\varepsilon$,
\begin{equation}\label{eq:protected-action}
 \int_{\mathcal E_{\varepsilon,U,K}}\Scal\,dV
 \geq c_0\Delta_{\varepsilon,U}.
\end{equation}
A sufficiently thin relative gluing collar is disjoint from
$\mathcal E_{\varepsilon,U,K}$.
\end{lemma}

\begin{proof}
Equation \eqref{eq:initial-core-smallness} gives
\[
 r_K+\sup_{0\leq r\leq r_K}y(r)
 =O_{U,K}\!\left(e^{-c_K/\varepsilon^3}\right).
\]
Thus $y\leq1/2$ on $[0,r_K]$ for small $\varepsilon$.  Moreover,
$f$ is decreasing and $f(r_K)=2/K$, so $f\geq2/K$ there.  Hence
\[
 \begin{aligned}
 0\leq\Delta_{\varepsilon,U}-\Psi(r_K)
 &=\int_0^{r_K}\frac{y(r)}{f(r)\sqrt{1-y(r)^2}}\,dr\\
 &\leq\frac{K}{\sqrt3}\,r_K\sup_{0\leq r\leq r_K}y(r)
 =O_{U,K}\!\left(e^{-c_K/\varepsilon^3}\right).
 \end{aligned}
\]
Because $\Delta_{\varepsilon,U}\to\Delta_U>0$ and $\Psi$ is decreasing,
for all sufficiently small $\varepsilon$,
\[
 \Psi(r)\geq\Psi(r_K)>\frac{\Delta_{\varepsilon,U}}3
 \qquad(0\leq r\leq r_K).
\]
Also $r_K<r_U$ because $K\varepsilon<U$.  Thus
$\mathcal E_{\varepsilon,U,K}$ is compactly contained in
$\Omega_{\varepsilon,U}$; the locus $r=0$ is a smooth $\theta$-bolt,
not a boundary component.

Since $\Ric>0$, one has $\Scal\geq\Ric(e_\theta,e_\theta)$.  From
\eqref{eq:sectional},
$\Ric(e_\theta,e_\theta)af=-(fa')'$.  Integrating first over
$[\eta,r_K]$ and then letting $\eta\downarrow0$, with angular ranges
$2\pi$ and $2\Delta_{\varepsilon,U}/3$, gives
\[
 \begin{aligned}
 \int_{\mathcal E_{\varepsilon,U,K}}\Scal\,dV
 &\geq\int_{\mathcal E_{\varepsilon,U,K}}
          \Ric(e_\theta,e_\theta)\,dV\\
 &=\frac{4\pi\Delta_{\varepsilon,U}}3
   \left(f_\varepsilon(0)a_\varepsilon'(0)-f(r_K)a'(r_K)\right)\\
 &=\frac{4\pi\Delta_{\varepsilon,U}}3
   \left(1-\frac4{K^2}+O_K(\varepsilon^3)\right).
 \end{aligned}
\]
Here \eqref{eq:af} and \eqref{eq:aprime} at $z=K\varepsilon$ give
$f(r_K)a'(r_K)=4C_\varepsilon/K^2+K\varepsilon^3/12$, and the pole term
is \eqref{eq:pole-flux}.  Since $K>6$, the last parenthesis is at least
$\frac12(1-4/K^2)>0$ after decreasing $\varepsilon$.  This proves
\eqref{eq:protected-action}, for example with
$c_0=\frac{2\pi}{3}(1-4/K^2)$.

Finally, compact containment gives positive distance from
$\mathcal E_{\varepsilon,U,K}$ to the boundary for each fixed
$\varepsilon$.  Proposition~\ref{prop:relative-gluing} permits an arbitrarily
thin gluing collar, which can therefore be chosen disjoint from this set.
\end{proof}

\medskip
\noindent
We now normalize the boundary meridian for the angular attachment.  Define
\begin{equation}\label{eq:D}
 D_{\varepsilon,U}=\frac{\pi\cos r_{\mathrm W}}{2\ell},
\end{equation}
so that the full meridian of the scaled boundary has length
$2D_{\varepsilon,U}\ell=\pi\cos r_{\mathrm W}$, and set
\begin{equation}\label{eq:w}
 w_{\varepsilon,U}
 =D_{\varepsilon,U}a_\varepsilon(r_U)
 =D_{\varepsilon,U}c_\varepsilon\varepsilon^2A_\varepsilon(U).
\end{equation}
Equations \eqref{eq:ell-asymp} and \eqref{eq:rU-asymp} give
\begin{align}
 w_{\varepsilon,U}
 &=\frac{2C_\varepsilon\cos r_{\mathrm W}}{U}\,
       \varepsilon\bigl(1+O(U^3)\bigr),                             \label{eq:w-asymp}\\
 D_{\varepsilon,U}r_U
 &=\cos r_{\mathrm W}\bigl(1+O(U^3)\bigr).                         \label{eq:depth-scaled}
\end{align}
Thus $w_{\varepsilon,U}\asymp_U\varepsilon$, while the rescaled radial
scale is bounded.  Combining this normalization with the preceding depth
and protected-action estimates gives the following scaled form.

\begin{corollary}\label{cor:scaled-protected}
Fix $K>6$.  Under the choices of
Lemma~\ref{lem:packet-asymptotics} and the normalization \eqref{eq:D},
there are constants
$c(U,r_{\mathrm W},K),C(U,r_{\mathrm W})>0$ such that,
after scaling the packet by $D_{\varepsilon,U}^2$, the set
$\mathcal E_{\varepsilon,U,K}$
satisfies
\begin{equation}\label{eq:protected-scaled}
 \begin{aligned}
  \int_{\mathcal E_{\varepsilon,U,K}}
  \Scal_{D_{\varepsilon,U}^2g_\varepsilon}
       \,dV_{D_{\varepsilon,U}^2g_\varepsilon}
  &\geq\frac{c(U,r_{\mathrm W},K)}{w_{\varepsilon,U}},\\
  \sup_{p\in\mathcal E_{\varepsilon,U,K}}
  d_{D_{\varepsilon,U}^2g_\varepsilon}
       (p,\partial\Omega_{\varepsilon,U})
  &\leq C(U,r_{\mathrm W}).
 \end{aligned}
\end{equation}
\end{corollary}

\begin{proof}
Lemma~\ref{lem:protected-set} gives the unscaled lower bound
$c_0\Delta_{\varepsilon,U}$ on the protected set.  Under this homothety,
a scalar action is multiplied by $D_{\varepsilon,U}$.
Equations \eqref{eq:D}, \eqref{eq:ell-asymp},
\eqref{eq:Delta-asymp}, and \eqref{eq:w-asymp} give
$D_{\varepsilon,U}\Delta_{\varepsilon,U}
\geq c/w_{\varepsilon,U}$.  The distance
estimate follows from Lemma~\ref{lem:depth} and
\eqref{eq:depth-scaled}.
\end{proof}

\begin{proof}[Proof of Proposition~\ref{prop:high-action-lens}]
Take $C_{\varepsilon,U}=\Omega_{\varepsilon,U}$ with metric
$g_{\varepsilon,U}=D_{\varepsilon,U}^2g_\varepsilon$.  The cap smoothing
and Lemma~\ref{lem:core-curvature} give $\Ric>0$.
Lemma~\ref{lem:lens-smooth} identifies the domain as a smooth three-ball,
while \eqref{eq:D} gives the boundary normalization and
\eqref{eq:w-asymp} gives the asserted waist.  Finally take
$\mathcal E_{\varepsilon,U}
=\mathcal E_{\varepsilon,U,7}$.  This is admissible because $7>6$.
Corollary~\ref{cor:scaled-protected} gives the
action estimate,
while Lemma~\ref{lem:protected-set} permits the gluing collar to avoid this
set.  Lemma~\ref{lem:depth}, together with
\eqref{eq:depth-scaled}, gives the distance estimate for every point of
$C_{\varepsilon,U}$.
\end{proof}

\section{The angular port and companion neck}

Let
\[
 h_+=du^2+b(u)^2d\theta^2,\qquad 0\leq u\leq\ell,
\]
be the unscaled metric induced on one graph face of
$\Omega_{\varepsilon,U}$, and let $h$ denote its smooth double.  The
boundary metric of the normalized lens in
Proposition~\ref{prop:high-action-lens} is $D_{\varepsilon,U}^2h$.
Our immediate purpose is to realize this metric as an inner angular
sphere of the attachment model, so that the two boundaries can be
identified isometrically.  On one scaled graph face, the meridional
line element and the $\theta$-orbit radius are, respectively,
$D_{\varepsilon,U}\,du$ and $D_{\varepsilon,U}b(u)$.

The attachment model uses the ambient metric
\begin{equation}\label{eq:ambient}
 g_{\mathrm{amb}}=\cot^2r_{\mathrm W}
 \left(dt^2+\cos^2t\,d\zeta^2+\mathcal R(t)^2d\theta^2\right),
\end{equation}
where $t$ is latitude and $\mathcal R(t)$ is the unscaled
$\theta$-warping profile.  The ambient domain used below has
$0\leq t\leq5r_{\mathrm W}$, with $\zeta$ and $\theta$ periodic; the
$\theta$-circle collapses smoothly along $t=0$, so $t$ is a nonnegative
radial coordinate from that collapsed orbit.  For $\zeta_0\in\mathbb R$
(understood modulo the $\zeta$-period), let
$o_{\zeta_0}$ denote the point with $(t,\zeta)=(0,\zeta_0)$ at which
the $\theta$-orbit collapses, and define
\[
 \begin{aligned}
 \varrho_{\zeta_0}(t,\zeta)
 &:=\arccos\!\bigl(\cos t\cos(\zeta-\zeta_0)\bigr),\\
 \mathcal B_s(o_{\zeta_0})
 &:=\{(t,\zeta,\theta):\varrho_{\zeta_0}(t,\zeta)\leq s\},\\
 \mathcal S_s(o_{\zeta_0})
 &:=\partial\mathcal B_s(o_{\zeta_0})
   =\{(t,\zeta,\theta):\varrho_{\zeta_0}(t,\zeta)=s\}
 \qquad(0<s\leq4r_{\mathrm W}).
 \end{aligned}
\]
We call $\mathcal S_s(o_{\zeta_0})$ the \emph{angular sphere} of
angular radius $s$ centered at $o_{\zeta_0}$, and
$\mathcal B_s(o_{\zeta_0})$ its angular ball.  Thus an angular sphere is
obtained by lifting the round geodesic circle
$\varrho_{\zeta_0}=s$ and adjoining its $\theta$-orbits.  Since the
$\theta$-orbit collapses at $o_{\zeta_0}$,
$\mathcal B_s(o_{\zeta_0})$ is the closed $g_{\mathrm{amb}}$-metric
ball of radius $s\cot r_{\mathrm W}$; the term \emph{angular radius}
distinguishes $s$ from this rescaled metric radius.

To match the lens boundary, consider on the unit round surface with
coordinates $(t,\zeta)$ the geodesic circle of angular radius
$r_{\mathrm W}$ centered at $(0,0)$.  Use the normalized meridional
coordinate $\chi=\pi u/(2\ell)$ introduced in Section~2 to parametrize
its pole-to-equator half.  Under the standard embedding
\[
 (t,\zeta)\longmapsto
 (\cos t\cos\zeta,\cos t\sin\zeta,\sin t)
\]
this half-circle is
\[
 (\cos r_{\mathrm W},
   \sin r_{\mathrm W}\cos\chi,
   \sin r_{\mathrm W}\sin\chi).
\]
Hence
$\sin t=\sin r_{\mathrm W}\sin\chi$, and its unscaled arclength
element is $\sin r_{\mathrm W}\,d\chi$.  In \eqref{eq:ambient}, its
lift has meridional line element $\cos r_{\mathrm W}\,d\chi$ and
$\theta$-orbit radius $\cot r_{\mathrm W}\mathcal R(t(\chi))$.

The earlier normalization
$D_{\varepsilon,U}=\pi\cos r_{\mathrm W}/(2\ell)$ was chosen precisely
so that the two pole-to-equator lengths agree:
$D_{\varepsilon,U}\ell=\pi\cos r_{\mathrm W}/2$.  Pointwise matching of
the meridional line elements and the $\theta$-orbit radii therefore
requires
\[
 \cos r_{\mathrm W}\,d\chi=D_{\varepsilon,U}\,du,
 \qquad
 \cot r_{\mathrm W}\mathcal R(t(u))
 =D_{\varepsilon,U}b(u).
\]
The definition of $\chi$ and the normalization of $D_{\varepsilon,U}$
make the first equality automatic.  The half-circle determines $t$ from
$\chi$, and the second equality determines $\mathcal R$.  Explicitly,
\begin{equation}\label{eq:ambient-profile}
 \chi=\frac{\pi u}{2\ell},\qquad
 t(u)=\arcsin\!\bigl(\sin r_{\mathrm W}\sin\chi\bigr),\qquad
 \mathcal R(t(u))
 =\tan r_{\mathrm W}\,D_{\varepsilon,U}\,b(u).
\end{equation}
The map $u\mapsto t(u)$ is strictly increasing from $[0,\ell]$ onto
$[0,r_{\mathrm W}]$, so \eqref{eq:ambient-profile} uniquely defines
$\mathcal R$ on $[0,r_{\mathrm W}]$.  Let $g_{\mathrm{ang}}$ denote the
metric induced by $g_{\mathrm{amb}}$ on the particular angular sphere
$\mathcal S_{r_{\mathrm W}}(o_0)$.  Its identification with
$D_{\varepsilon,U}^2h$ is verified in Lemma~\ref{lem:inverse-realization}.
The smoothness of $\mathcal R$ at the collapsed orbit and equator, and
the concavity needed for positive sectional curvature, are also proved
below.  Thus the boundary-isometry requirement reduces the attachment
problem to controlling and extending this reconstructed profile.

We first construct the angular pair of pants and record its own boundary
geometry; the following corollary then records the packet and neck
attachments.

\begin{theorem}
\label{thm:attachment}
Fix $r_{\mathrm W}$ as in \eqref{eq:r-range}.  There is
$U_0=U_0(r_{\mathrm W})>0$ such that, for every fixed
$U\in(0,U_0)$, there is $\varepsilon_0=\varepsilon_0(U,r_{\mathrm W})>0$
for which the following holds whenever
$0<\varepsilon<\varepsilon_0$.

The reconstructed profile $\mathcal R$ extends smoothly from
$[0,r_{\mathrm W}]$ to $[0,5r_{\mathrm W}]$, and the metric
\eqref{eq:ambient} is smooth across the collapsed $\theta$-orbit at
$t=0$.  The initial data at the collapsed orbit are
\[
 \mathcal R(0)=0,\qquad \mathcal R'(0)=1,\qquad
 \mathcal R''(0)=0.
\]
On the extended interval the profile is positive, strictly increasing, and
strictly concave away from the collapsed orbit:
\[
 \mathcal R(t)>0,\qquad \mathcal R'(t)>0,\qquad
 \mathcal R''(t)<0
 \qquad(0<t\leq5r_{\mathrm W}).
\]
Using the angular balls and spheres defined above, set
\[
 \mathcal P_{\varepsilon,U}
 :=\mathcal B_{4r_{\mathrm W}}(o_0)\setminus
 \left(
  \operatorname{int}\mathcal B_{r_{\mathrm W}}(o_{2r_{\mathrm W}})
  \cup
  \operatorname{int}\mathcal B_{r_{\mathrm W}}(o_{-2r_{\mathrm W}})
 \right),
 \qquad
 g_{\mathcal P}:=g_{\mathrm{amb}}|_{\mathcal P_{\varepsilon,U}}.
\]
Then $(\mathcal P_{\varepsilon,U},g_{\mathcal P})$ is a compact
sectionally positive three-manifold, diffeomorphic to a
three-ball with the interiors of two disjoint three-balls removed.  Its
boundary components are
\[
 \Sigma_{\mathrm{out}}=\mathcal S_{4r_{\mathrm W}}(o_0),\qquad
 \Sigma_{\mathrm{lens}}=\mathcal S_{r_{\mathrm W}}(o_{2r_{\mathrm W}}),
 \qquad
 \Sigma_{\mathrm{neck}}=\mathcal S_{r_{\mathrm W}}(o_{-2r_{\mathrm W}}),
\]
and they have the following properties.
\begin{enumerate}[label=\textup{(\roman*)}]
\item Both inner boundary metrics are isometric to the prescribed
normalized rotational metric:
\[
 (\Sigma_{\mathrm{lens}},g_{\mathcal P})
 \cong(\Sigma_{\mathrm{neck}},g_{\mathcal P})
 \cong(\mathbb S^2,D_{\varepsilon,U}^2h).
\]
With respect to the meridional--circle principal frame, the shape
operator of either removed angular ball, for the normal pointing from
that ball into $\mathcal P_{\varepsilon,U}$, is
\[
 S^{\mathrm{hole}}=\operatorname{diag}(1,q_\theta),\qquad
 q_\theta=\tan t\,\frac{\mathcal R'(t)}{\mathcal R(t)},\qquad
 0<q_\theta\leq q_*:=\frac{\tan r_{\mathrm W}}{r_{\mathrm W}},
\]
where $q_\theta$ is understood by its smooth limit at a pole.

\item The outer boundary $\Sigma_{\mathrm{out}}$ is strictly convex,
and its induced metric has positive Gaussian curvature.
\end{enumerate}
\end{theorem}

The theorem constructs the central three-boundary piece.  To use it in
the final gluing, one must also match the packet to one inner boundary
and a round-ended neck to the other, with the strict shape inequalities
required by Proposition~\ref{prop:relative-gluing}.  The next corollary
records these two compatible attachments; its proof is given after the
supporting boundary and neck estimates below.

\begin{corollary}
\label{cor:angular-attachments}
After decreasing $U_0$ in Theorem~\ref{thm:attachment} if necessary,
and then decreasing $\varepsilon_0$ for each fixed $U\in(0,U_0)$, the
following consequences hold for every $0<\varepsilon<\varepsilon_0$.
\begin{enumerate}[label=\textup{(\roman*)}]
\item The normalized packet boundary is isometric to
$\Sigma_{\mathrm{lens}}$.  Under this isometry, the outward shape
operator $S^{\mathrm{pkt}}$ of the packet lens satisfies
\[
 S^{\mathrm{pkt}}-S^{\mathrm{hole}}>0.
\]
Since the outward shape operator of $\mathcal P_{\varepsilon,U}$ along
$\Sigma_{\mathrm{lens}}$ is $-S^{\mathrm{hole}}$, the packet and the
angular pair of pants satisfy the relative Ricci-positive gluing
hypothesis at this seam.

\item There is a Ricci-positive companion neck whose angular end is
isometric to $\Sigma_{\mathrm{neck}}$ and whose outward shape operator
there satisfies
\[
 S^{\mathrm{neck}}>q_*\operatorname{Id}
 \geq S^{\mathrm{hole}}.
\]
Thus the neck and the angular pair of pants also satisfy the relative
Ricci-positive gluing hypothesis.  For some $\rho,\lambda>0$, the other
neck end is round of radius $\rho/\lambda$, has outward shape operator
$-\lambda\operatorname{Id}$ with
$\lambda>\rho/w_{\varepsilon,U}$, and is at distance less than $1$ from
the angular end.
\end{enumerate}
\end{corollary}

We now establish the ingredients needed to prove the theorem and its
corollary.  The first step is to verify that the reconstruction
\eqref{eq:ambient-profile} reproduces the normalized lens boundary, is
smooth at the collapsed orbit and the equator, and converts the required
ambient concavity into an inequality involving the lens profile.

\begin{lemma}
\label{lem:inverse-realization}
With $h_+$, $h$, and $\mathcal R$ as defined above, the sphere
$(\mathbb S^2,D_{\varepsilon,U}^2h)$ is isometric to
$(\mathcal S_{r_{\mathrm W}}(o_0),g_{\mathrm{ang}})$.
The reconstructed profile $\mathcal R$ is smooth on
$[0,r_{\mathrm W}]$, extends smoothly and oddly across $t=0$ with
$\mathcal R'(0)=1$, and satisfies $\mathcal R'>0$ on
$[0,r_{\mathrm W}]$.  In particular, $g_{\mathrm{amb}}$ is smooth
across the collapsed $\theta$-orbit.  At interior points of one
meridional half, $0<u<\ell$, the inequality $\mathcal R''<0$ is
equivalent to
\begin{equation}\label{eq:concavity-criterion}
 -\frac{2\ell}{\pi}\frac{b''}{b'}>
 \frac{\cos^2r_{\mathrm W}\tan\chi}
 {1-\sin^2r_{\mathrm W}\sin^2\chi}.
\end{equation}
At every regular point $0<t\leq r_{\mathrm W}$ where
$\mathcal R''<0$, the ambient metric has positive sectional curvature.
\end{lemma}

\begin{proof}
The two matching identities preceding \eqref{eq:ambient-profile}
identify the metric on the pole-to-equator half with
$D_{\varepsilon,U}^2h_+$.  Smooth reflection therefore gives
\begin{equation}\label{eq:exact-metric}
 g_{\mathrm{ang}}=D_{\varepsilon,U}^2h.
\end{equation}

At the rotational pole $u=0$, the smoothness of $h_+$ means that
$b$ extends smoothly and oddly, with $b'(0)=1$.  The middle identity
in \eqref{eq:ambient-profile} makes $t(u)$ an odd local coordinate with
$t'(0)=\pi\sin r_{\mathrm W}/(2\ell)$.  Since
$\tan r_{\mathrm W}D_{\varepsilon,U}
=\pi\sin r_{\mathrm W}/(2\ell)$, the last identity there shows that
$\mathcal R(t)$ is smooth and odd at the axis, with
\[
 \mathcal R(0)=0,\qquad \mathcal R'(0)=1,\qquad
 \mathcal R^{(2j)}(0)=0\quad(j\geq1).
\]
Thus $g_{\mathrm{amb}}$ is smooth across the collapsed orbit.

At the equator $u=\ell$, put $v=\ell-u$.  Smooth doubling makes
$b(\ell-v)$ a smooth even function of $v$.  Equation
\eqref{eq:ambient-profile} gives
\[
 t(\ell-v)
 =\arcsin\!\left(
   \sin r_{\mathrm W}\cos\frac{\pi v}{2\ell}
  \right),\qquad
 r_{\mathrm W}-t(\ell-v)
 =\frac{\pi^2\tan r_{\mathrm W}}{8\ell^2}v^2+O(v^4).
\]
Both $r_{\mathrm W}-t(\ell-v)$ and
$\mathcal R(t(\ell-v))$ are smooth functions of $v^2$, and the first
has positive derivative at $v^2=0$.  The inverse-function theorem
therefore shows that $\mathcal R$ is smooth as a function of $t$ at
$t=r_{\mathrm W}$.

In the following formulas, primes on $b$ denote $u$-derivatives and
primes on $\mathcal R$ denote $t$-derivatives.  Differentiating on
$0<u<\ell$ gives
\begin{align}
 \mathcal R'(t)
 &=b'(u)\frac{\cos t}{\cos\chi},                                  \label{eq:R-prime}\\
 \mathcal R''(t)
 &=\frac{b'(u)\cos^2t}
 {\sin r_{\mathrm W}\cos^2\chi}
 \left(
  \frac{2\ell}{\pi}\frac{b''(u)}{b'(u)}
  +\frac{\cos^2r_{\mathrm W}\tan\chi}{\cos^2t}
 \right).                                                        \label{eq:R-second}
\end{align}
Lemma~\ref{lem:lens-ode} gives $b'>0$ on $0<u<\ell$, so
\eqref{eq:R-prime} yields $\mathcal R'>0$ there.  At the equator,
$b'(\ell-v)=-b''(\ell)v+O(v^3)$ and
$\cos\chi=\pi v/(2\ell)+O(v^3)$; hence
\[
 \mathcal R'(r_{\mathrm W})
 =-\frac{2\ell}{\pi}b''(\ell)\cos r_{\mathrm W}>0.
\]
Here $b''(\ell)<0$ by Lemma~\ref{lem:lens-ode}.  Together with
$\mathcal R'(0)=1$, this proves $\mathcal R'>0$ on the closed interval.

All factors outside the parentheses in \eqref{eq:R-second} are positive.
Since
$\cos^2t=1-\sin^2r_{\mathrm W}\sin^2\chi$, the condition
$\mathcal R''<0$ is therefore exactly
\eqref{eq:concavity-criterion}.

The three ambient sectional curvatures, apart from the common positive
factor $\tan^2r_{\mathrm W}$, are
\begin{equation}\label{eq:ambient-sections}
 1,\qquad-\frac{\mathcal R''}{\mathcal R},\qquad
 \frac{\mathcal R'}{\mathcal R}\tan t.
\end{equation}
For $0<t\leq r_{\mathrm W}$ one has $\mathcal R>0$ and
$\mathcal R'>0$, so all three are positive wherever
$\mathcal R''<0$.  At the collapsed orbit the quotients are interpreted
by their smooth limits.
\end{proof}

The reconstructed profile is presently defined only for
$0\leq t\leq r_{\mathrm W}$, whereas the pair-of-pants construction
requires it for $0\leq t\leq5r_{\mathrm W}$.  Once strict concavity on
the initial interval has been established, the following elementary
lemma extends the profile to the full interval without changing its
reconstructed boundary data.

\begin{lemma}\label{lem:R-extension}
Let $\mathcal R$ be smooth on $[0,r]$ and suppose
\[
\mathcal R(0)=0,\qquad \mathcal R''(0)=0,
\]
and
\[
 \mathcal R(t)>0,\qquad
 \mathcal R'(t)>0,\qquad
 \mathcal R''(t)<0
 \qquad(0<t\leq r).
\]
Then it has a smooth extension to $[0,5r]$ satisfying
\[
 \mathcal R(0)=0,\qquad \mathcal R''(0)=0,
\]
and
\[
 \mathcal R(t)>0,\qquad
 \mathcal R'(t)>0,\qquad
 \mathcal R''(t)<0
 \qquad(0<t\leq5r).
\]
\end{lemma}

\begin{proof}
Put $\psi=-\mathcal R''$.  Extend the positive germ of $\psi$ at $r$ to a
positive smooth function $\widetilde\psi$ on $[r,r+2\delta]$.  Choose a
cutoff which is one near $r$ and zero near $r+2\delta$, and interpolate
there between $\widetilde\psi$ and a positive constant $\eta$.  This
produces a smooth $\psi_{\mathrm{ext}}>0$ on $[r,5r]$, agreeing with the
full germ of $\psi$ at $r$.  First choose $\delta$ small and then $\eta$
small so that
\[
 \int_r^{5r}\psi_{\mathrm{ext}}(s)\,ds<\frac12\mathcal R'(r).
\]
For $t\geq r$ define
\[
 \mathcal R_{\mathrm{ext}}(t)
 =\mathcal R(r)+\mathcal R'(r)(t-r)
 -\int_r^t(t-s)\psi_{\mathrm{ext}}(s)\,ds.
\]
It agrees smoothly with the given function at $r$, and
\[
 \mathcal R_{\mathrm{ext}}''=-\psi_{\mathrm{ext}}<0,\qquad
 \mathcal R_{\mathrm{ext}}'(t)
 =\mathcal R'(r)-\int_r^t\psi_{\mathrm{ext}}(s)\,ds
 >\frac12\mathcal R'(r)>0.
\]
Hence $\mathcal R_{\mathrm{ext}}(t)\geq\mathcal R(r)>0$, proving the
claim.
\end{proof}

The extension lemma applies once the reconstructed profile is known to be
strictly concave on its initial interval.  The next lemma proves exactly
this missing concavity statement.  The normalized lens identities from
Section~2 reduce it to a pointwise comparison between the normal angle
$\alpha$ and normalized arclength $\chi$.

\begin{lemma}
\label{lem:uniform-concavity}
Fix $r_{\mathrm W}$ as in \eqref{eq:r-range}.  For every sufficiently
small fixed $U\in(0,1)$, and then every sufficiently small
$\varepsilon>0$, inequality \eqref{eq:concavity-criterion} holds at every
interior point $0<u<\ell$, including the interiors of the cap and the
cap-to-core smoothing collar.  The associated ambient profile satisfies
\[
 \mathcal R''(t)<0\quad(0<t\leq r_{\mathrm W}),
 \qquad \mathcal R''(0)=0,
 \qquad \lim_{t\downarrow0}-\frac{\mathcal R''(t)}{\mathcal R(t)}>0.
\]
\end{lemma}

\begin{proof}
For $0<u<\ell$, set
\[
 L:=-\frac{2\ell}{\pi}\frac{b''}{b'},\qquad
 \mathcal F(\chi):=
 \frac{\cos^2r_{\mathrm W}\tan\chi}
 {1-\sin^2r_{\mathrm W}\sin^2\chi}.
\]
By Lemma~\ref{lem:inverse-realization}, proving
$\mathcal R''<0$ is equivalent to proving $L>\mathcal F(\chi)$.
By the lens equation in Lemma~\ref{lem:lens-ode},
$d\alpha/du=y_r\geq\kappa$ on the entire pole-to-equator graph face.
Since $\mathfrak m=2\kappa\ell/\pi$ and
$\chi=\pi u/(2\ell)$, integration from the pole gives
\begin{equation}\label{eq:alpha-lower}
 \alpha\geq\mathfrak m\chi.
\end{equation}
At the equator $\alpha=\chi=\pi/2$, so this inequality also shows that
$\mathfrak m\leq1$.  Moreover, \eqref{eq:bderivatives}, $a''\leq0$,
$a'>0$, and $y_r\geq\kappa$ give
\begin{equation}\label{eq:L-polar}
 L
 \geq\mathfrak m\tan\alpha
 \geq\mathfrak m\tan(\mathfrak m\chi).
\end{equation}
This estimate will be used only on the cap and smoothing collar.  Their
endpoint has normalized angle
\[
 \chi_{\mathrm{cap}}
 :=\frac{\pi\ell_{\mathrm{cap},\varepsilon}}{2\ell}
 \longrightarrow0
\]
by \eqref{eq:cap-transition-negligible} and \eqref{eq:ell-asymp}.
Uniformly for $\mathfrak m$ near $1$,
\[
 \frac{\mathfrak m\tan(\mathfrak m\chi)}{\mathcal F(\chi)}
 \longrightarrow
 \frac{\mathfrak m^2}{\cos^2r_{\mathrm W}}
 \qquad(\chi\downarrow0).
\]
Equation~\eqref{eq:charge-asymp} therefore allows us first to choose $U$
so that $\mathfrak m>\cos r_{\mathrm W}$ with a uniform margin for all
sufficiently small $\varepsilon$, and then to reduce $\varepsilon$ so
that \eqref{eq:L-polar} implies $L>\mathcal F(\chi)$ throughout the
punctured cap and smoothing collar.

On the explicit core put
\[
 s=\frac{z^3}{48C_\varepsilon},\qquad v=\tan\alpha,\qquad
 \mathcal B=(1+T)v+\frac{T\Xi}{v}.
\]
Equation~\eqref{eq:exact-L} gives $L=\mathfrak m\mathcal B$, while
\eqref{eq:T-E} and \eqref{eq:chi-alpha} give
\[
 \chi'=\frac{1}{\mathfrak m(1+T)},
\]
where a prime now denotes $d/d\alpha$.  On the core we prove the
stronger comparison $L>\tan\chi$.  This is sufficient because
\begin{equation}\label{eq:target-gap}
 \tan\chi-\mathcal F(\chi)
 =
 \frac{\sin^2r_{\mathrm W}\sin\chi\cos\chi}
 {1-\sin^2r_{\mathrm W}\sin^2\chi}>0.
\end{equation}
To compare $L$ with $\tan\chi$, define
\[
 \Theta:=\arctan L=\arctan(\mathfrak m\mathcal B).
\]
Since $0<\chi<\pi/2$, the desired inequality $L>\tan\chi$ is
equivalent to $\Theta>\chi$.  Moreover, $v=\tan\alpha\to\infty$ and
$\mathcal B\sim(1+T)v$ at the equator, so both $\Theta$ and $\chi$
tend to $\pi/2$.  It therefore suffices to prove
$\Theta'<\chi'$ and then integrate backward from the equator.  For
this derivative comparison, it is enough to prove
\begin{equation}\label{eq:Q-less}
 \mathcal Q:=(1+T)\mathcal B'-\mathcal B^2<1
\end{equation}
Indeed, this inequality implies
\begin{align*}
 \Theta'
 &=\frac{\mathfrak m\mathcal B'}
 {1+\mathfrak m^2\mathcal B^2}\notag\\
 &<\frac{\mathfrak m(1+\mathcal B^2)}
 {(1+T)(1+\mathfrak m^2\mathcal B^2)}
 \leq\frac{1}{\mathfrak m(1+T)}=\chi'.
\end{align*}
The last inequality uses $0<\mathfrak m\leq1$.  We now
verify \eqref{eq:Q-less}.  The first identity in
\eqref{eq:core-differential-identities} gives the reciprocal relation
\[
 z_\alpha=\frac{zT}{(1+T)\tan\alpha}.
\]
Moreover, $\Xi=(1-2s)/(1+s)$ and $zs_z=3s$, so
\[
 z\Xi_z=-\frac{9s}{(1+s)^2}.
\]
Using these relations, the second identity in
\eqref{eq:core-differential-identities}, rewritten using
$8C_\varepsilon/z^3=1/(6s)$, and $v'=1+v^2$, expansion of
$\mathcal Q-1$ gives
\begin{align}
 \mathcal Q-1={}&
 \frac{9Ts}{1+s}
 +T^2\left(\frac56+\frac{9s}{1+s}-\frac1{6s}\right)                \label{eq:Q}\\
 &+\frac{T^2}{v^2}
 \left[
 \Xi\left(\frac{11}{6}-\frac1{6s}-\Xi\right)
 -\frac{9s}{(1+s)^2}
 \right].\notag
\end{align}
The strict lens comparison \eqref{eq:smoothing-W}, already established
in Section~2 on the smoothing collar and ensuing unchanged core, implies
after substitution of \eqref{eq:af}, \eqref{eq:aprime}, and \eqref{eq:H}
the exact lower bound
\begin{equation}\label{eq:T-exact-lower}
 T(z)\geq\frac{z^3}{8C_\varepsilon(1+s)}
 =\frac{6s}{1+s}
\end{equation}
for the actual primitive, with no discarded cap term.

After reducing $U$, we have $0<s<1/17$ throughout the core.  In
particular, $T>0$, $v>0$, and $\Xi>0$.  The coefficient of $T^2/v^2$
in \eqref{eq:Q} is negative because
\[
 \frac{11}{6}-\frac1{6s}-\Xi<0.
\]
The remaining two terms have negative sum.  Indeed,
\[
 \frac56+\frac{9s}{1+s}-\frac1{6s}<0,
\]
and equation \eqref{eq:T-exact-lower} gives
\begin{align*}
 \frac1T\left\{\frac{9Ts}{1+s}
 +T^2\left(\frac56+\frac{9s}{1+s}-\frac1{6s}\right)\right\}
 &\leq\frac{9s}{1+s}+\frac{6s}{1+s}
 \left(\frac56+\frac{9s}{1+s}-\frac1{6s}\right)\\
 &=\frac{s}{1+s}
 \left(14+\frac{54s}{1+s}-\frac1s\right)<0.
\end{align*}
This proves \eqref{eq:Q-less}, and hence $\Theta'<\chi'$.  Integrating
backward from the common equatorial limit, for every interior core
point, gives
\[
 \frac\pi2-\Theta(\alpha)
 =\int_\alpha^{\pi/2}\Theta'(\beta)\,d\beta
 <\int_\alpha^{\pi/2}\chi'(\beta)\,d\beta
 =\frac\pi2-\chi(\alpha).
\]
Thus $\Theta>\chi$ and therefore $L>\tan\chi$ on the explicit core.
Together with \eqref{eq:target-gap}, this proves
$L>\mathcal F(\chi)$ throughout the open core.

At the equator, where the gap in \eqref{eq:target-gap} tends to zero,
rearranging \eqref{eq:R-second} gives
\begin{equation}\label{eq:Rsecond-factorized}
 \mathcal R''
 =-\frac{b'}{\cos\chi}\,
 \frac{\cos^2t}{\sin r_{\mathrm W}}\,
 \frac{L-\mathcal F}{\cos\chi}.
\end{equation}
Smoothness of the doubled lens boundary across the equator and
$b''(\ell)<0$ give
\[
 \lim_{\chi\uparrow\pi/2}\frac{b'}{\cos\chi}
 =-\frac{2\ell}{\pi}b''(\ell)>0.
\]
Moreover $L\geq\tan\chi$ implies
\[
 \liminf_{\chi\uparrow\pi/2}\frac{L-\mathcal F}{\cos\chi}
 \geq
 \lim_{\chi\uparrow\pi/2}
 \frac{\tan\chi-\mathcal F}{\cos\chi}
 =\tan^2r_{\mathrm W}>0.
\]
Since $\mathcal R''$ is continuous at $r_{\mathrm W}$,
\eqref{eq:Rsecond-factorized} therefore gives
$\mathcal R''(r_{\mathrm W})<0$.

On the exact cap, $a''=0$, so \eqref{eq:bderivatives} gives
\[
 L=\mathfrak m\frac{y_r}{\kappa}\tan\alpha.
\]
As $u\downarrow0$, one has $y_r\to\kappa$ and
$\alpha/\chi\to\mathfrak m$.  Hence
$L/\tan\chi\to\mathfrak m^2$, whereas
$\mathcal F(\chi)/\tan\chi\to\cos^2r_{\mathrm W}$.  Using
\eqref{eq:Rsecond-factorized}, $t/\chi\to\sin r_{\mathrm W}$, and
$\mathcal R(t)/t\to1$ gives
\begin{equation}\label{eq:axis-curvature-limit}
 \lim_{t\downarrow0}-\frac{\mathcal R''(t)}{\mathcal R(t)}
 =\frac{\mathfrak m^2-\cos^2r_{\mathrm W}}
 {\sin^2r_{\mathrm W}}>0.
\end{equation}
Smooth oddness of $\mathcal R$ at the axis gives
$\mathcal R''(0)=0$.  The last display also shows that the ambient
sectional curvature represented by $-\mathcal R''/\mathcal R$ has a
strictly positive axis limit.
\end{proof}

Together with Lemma~\ref{lem:R-extension}, this completes the construction
of the strictly concave ambient profile.  The following lemma gives the
intrinsic curvature margin for the matched angular sphere that will be
used in the companion-neck construction.

\begin{lemma}
\label{lem:Kgreater}
Suppose that
\begin{equation}\label{eq:charge-window}
 2\kappa\ell>\pi\cos r_{\mathrm W},
\end{equation}
then the Gaussian curvature of
$g_{\mathrm{ang}}=D_{\varepsilon,U}^2h$ satisfies the quantitative bound
\begin{equation}\label{eq:angular-curvature-lower}
 K_{g_{\mathrm{ang}}}
 \geq\left(\frac{2\kappa\ell}
 {\pi\cos r_{\mathrm W}}\right)^2>1.
\end{equation}
\end{lemma}

\begin{proof}
Let $e_u$ be the unit meridional tangent of the lens in the packet.  The
Gauss equation and \eqref{eq:shapes} give
\[
 K_h=K_{g_\varepsilon}(e_u,e_\theta)+\kappa k_\theta.
\]
Since the curvature operator of the doubly warped packet metric is
diagonal in the coordinate two-planes and
$e_u=\cos\alpha\,e_r-\sin\alpha\,e_\phi$, one has
\[
 K_{g_\varepsilon}(e_u,e_\theta)
 =\cos^2\alpha\,K_{r\theta}
  +\sin^2\alpha\,K_{\phi\theta}\geq0.
\]
Here nonnegativity follows on the exact cap from \eqref{eq:cap}, on the
smoothing collar from \eqref{eq:smoothing-monotone}, and on the core from
Lemma~\ref{lem:core-curvature}.  Equation~\eqref{eq:y-solution} gives
\[
 k_\theta=\kappa\frac{a'I}{af}.
\]
On the exact cap, $a=r$ and \eqref{eq:cap-I} gives $a'I-af\geq0$;
on the smoothing collar and subsequent core, the strict inequality is
\eqref{eq:smoothing-W}.  Hence $k_\theta\geq\kappa$ and
$K_h\geq\kappa^2$.  These inequalities extend to the pole and equator by
smoothness.  Scaling by $D_{\varepsilon,U}^2$ and using
\eqref{eq:D} gives
\[
 K_{g_{\mathrm{ang}}}
 \geq D_{\varepsilon,U}^{-2}\kappa^2
 =\left(\frac{2\kappa\ell}
 {\pi\cos r_{\mathrm W}}\right)^2.
\]
The last quantity is greater than one by \eqref{eq:charge-window}.
\end{proof}

Intrinsic curvature alone does not verify the relative gluing condition:
the packet's outward shape operator must also dominate the shape operator
of the removed angular ball.  Using the normal-angle and
normalized-arclength coordinates $\alpha$ and $\chi$ from Section~2, the
next lemma reduces this tensor inequality to two scalar inequalities.

\begin{lemma}
\label{lem:seam}
Assume \eqref{eq:charge-window}.  Suppose that the quotient
$\tan\alpha/\tan\chi$, initially defined on $0<u<\ell$, has continuous
endpoint limits and satisfies
\begin{equation}\label{eq:angle-window}
 \frac{\tan\alpha(u)}{\tan\chi(u)}>\cos r_{\mathrm W}
 \qquad(0\leq u\leq\ell),
\end{equation}
where the endpoint values are understood as these limits.
Under the boundary isometry, let $S^{\mathrm{pkt}}$ be the outward shape
operator of the scaled packet and let $S^{\mathrm{hole}}$ be the shape
operator of the removed angular ball for the normal pointing from that
ball into the perforated region.  Then
\[
 S^{\mathrm{pkt}}-S^{\mathrm{hole}}>0.
\]
Equivalently, the sum of the outward second fundamental forms of the
scaled packet and of the perforated angular region is positive definite.
\end{lemma}

\begin{proof}
For the normal used to define $S^{\mathrm{hole}}$, the angular sphere has
meridional principal curvature $1$.  Its circle principal curvature
follows from \eqref{eq:ambient-profile}, \eqref{eq:R-prime}, and
\eqref{eq:D}:
\begin{equation}\label{eq:qtheta}
 q_\theta=\tan t\frac{\mathcal R'}{\mathcal R}
 =\frac{2\ell}{\pi}\frac{b'}b\tan\chi.
\end{equation}
The scaled packet lens has principal curvatures
\[
 \frac{\kappa}{D_{\varepsilon,U}},\qquad
 \frac1{D_{\varepsilon,U}}\frac{a'}a\sin\alpha.
\]
Their meridional difference is
\[
 \frac{\kappa}{D_{\varepsilon,U}}-1
 =\frac{2\kappa\ell-\pi\cos r_{\mathrm W}}
 {\pi\cos r_{\mathrm W}}>0
\]
by \eqref{eq:charge-window}.  Since $y=\sin\alpha$,
equation \eqref{eq:bderivatives} gives $b'=a'\cos\alpha$.  Together with
$b=a$ and $2\ell/\pi=\cos r_{\mathrm W}/D_{\varepsilon,U}$, this
factorizes the circle difference as
\begin{equation}\label{eq:circle-seam-factorization}
 \frac{b'\tan\chi}{aD_{\varepsilon,U}}
 \left(\frac{\tan\alpha}{\tan\chi}-\cos r_{\mathrm W}\right).
\end{equation}
The first factor is positive for $0<u<\ell$ and has the positive endpoint
limits
\[
 \lim_{u\downarrow0}\frac{b'\tan\chi}{aD_{\varepsilon,U}}
 =\frac1{\cos r_{\mathrm W}},\qquad
 \lim_{u\uparrow\ell}\frac{b'\tan\chi}{aD_{\varepsilon,U}}
 =-\frac{2\ell}{\pi}
   \frac{b''(\ell)}{a(\ell)D_{\varepsilon,U}}>0.
\]
The second factor in \eqref{eq:circle-seam-factorization} is positive on
the entire closed interval by \eqref{eq:angle-window}, with its endpoint
values interpreted as in the statement.  Thus
\eqref{eq:circle-seam-factorization} is positive everywhere, which proves
both eigenvalues of $S^{\mathrm{pkt}}-S^{\mathrm{hole}}$ are positive.
\end{proof}

It remains to verify the normalized-bending and angle inequalities in a common
parameter regime.  The next lemma obtains both from the lens asymptotics
of Section~2, thereby supplying the hypotheses of
Lemmas~\ref{lem:Kgreater} and~\ref{lem:seam} simultaneously.

\begin{lemma}\label{lem:docking-windows}
Fix $r_{\mathrm W}$ as in \eqref{eq:r-range}.  There is $U_1>0$ such
that, for every fixed $U\in(0,U_1)$ and every sufficiently small
$\varepsilon>0$, the quotient $\tan\alpha/\tan\chi$ extends continuously
to the two endpoints and satisfies \eqref{eq:angle-window}.  In the same
parameter regime,
\begin{equation}\label{eq:strong-charge}
 \mathfrak m=\frac{2\kappa\ell}{\pi}>
 \frac{\sin r_{\mathrm W}}{r_{\mathrm W}}>\cos r_{\mathrm W}.
\end{equation}
\end{lemma}

\begin{proof}
Choose $U_1$ below the small-$U$ thresholds in
Lemmas~\ref{lem:packet-asymptotics}, \ref{lem:normalized-identities},
and \ref{lem:uniform-concavity}.  By \eqref{eq:constants},
$C_\varepsilon\to1$, and \eqref{eq:charge-asymp} therefore implies,
uniformly for all sufficiently small $\varepsilon$,
\[
 \mathfrak m
 =1-\frac{U^3}{8C_\varepsilon}+O(U^6)
 \geq1-C_0U^3.
\]
Because $\sin r_{\mathrm W}/r_{\mathrm W}<1$, reducing $U_1$ gives the
first inequality in \eqref{eq:strong-charge}; the second follows from
$\tan r_{\mathrm W}>r_{\mathrm W}$.

On $0<\chi\leq\pi/3$, equation \eqref{eq:alpha-lower} gives
\[
 \frac{\tan\alpha}{\tan\chi}
 \geq\frac{\tan(\mathfrak m\chi)}{\tan\chi}.
\]
As observed immediately after \eqref{eq:alpha-lower},
$0<\mathfrak m\leq1$; the preceding estimate for $\mathfrak m$ then gives
$1-\mathfrak m=O(U^3)$.  The mean-value theorem, together with
$\tan\chi\geq\chi$ and the uniform bound for $\sec^2\chi$ on this
interval, therefore gives
\begin{equation}\label{eq:angle-away-equator}
 \frac{\tan\alpha}{\tan\chi}\geq1-C_1U^3.
\end{equation}
At the pole, \eqref{eq:cap}, \eqref{eq:y-solution}, and
\eqref{eq:lensmetric} give $y=\kappa r+O(r^3)$ and
$u=r+O(r^3)$, hence $\alpha=\arcsin y=\kappa u+O(u^3)$.
Together with the definition $\chi=\pi u/(2\ell)$, this shows that the
quotient extends continuously with value $\mathfrak m$.

Suppose now that $\chi\geq\pi/3$.  After reducing $U_1$ so that
$\mathfrak m\geq3/4$, equation \eqref{eq:alpha-lower} gives
$\alpha\geq\pi/4$.  Set
\[
 x=\frac\pi2-\alpha,\qquad X=\frac\pi2-\chi.
\]
On $0\leq x\leq\pi/4$, one has
$-\log\sin\alpha=-\log\cos x\leq Cx^2$.  Combining
\eqref{eq:M-refined} with \eqref{eq:chi-alpha}, and using the preceding
uniform lower bound for $\mathfrak m$, yields
\[
 \frac{d\chi}{d\alpha}
 =\frac{E_{\mathrm{core}}(U)}{\mathfrak m}+O(U^6x^2),
 \qquad
 \frac{E_{\mathrm{core}}(U)}{\mathfrak m}=1+O(U^6).
\]
Using the endpoint conditions $\alpha(\ell)=\chi(\ell)=\pi/2$ stated
before Lemma~\ref{lem:normalized-identities},
\begin{align*}
 X
 &=\int_\alpha^{\pi/2}\frac{d\chi}{d\beta}\,d\beta\notag\\
 &=\frac{E_{\mathrm{core}}(U)}{\mathfrak m}x+O(U^6x^3),
 \qquad |X-x|\leq C U^6x.
\end{align*}
The function $\tan t/t$, extended by the value $1$ at $t=0$, is smooth
and positive on the relevant compact interval.  Consequently, uniformly,
\begin{equation}\label{eq:angle-near-equator}
 \frac{\tan\alpha}{\tan\chi}
 =\frac{\tan X}{\tan x}
 =\frac{X}{x}\frac{\tan X/X}{\tan x/x}
 =1+O(U^6)
\end{equation}
Moreover, its equatorial endpoint value is
\[
 \lim_{u\uparrow\ell}\frac{\tan\alpha}{\tan\chi}
 =\frac{E_{\mathrm{core}}(U)}{\mathfrak m}.
\]
Equations \eqref{eq:angle-away-equator} and
\eqref{eq:angle-near-equator} give lower bounds $1-C_1U^3$ and
$1-C_2U^6$, respectively.  Since $\cos r_{\mathrm W}<1$, a final
reduction of $U_1$ makes both bounds, including their endpoint limits,
strictly greater than $\cos r_{\mathrm W}$.  This proves
\eqref{eq:angle-window}.
\end{proof}

Together with Lemma~\ref{lem:seam}, the preceding lemma supplies the
required shape comparison at the packet-to-pants seam.  To construct the
companion neck at the other inner boundary, we must rewrite the matched
metric in the global latitude form used in Proposition~\ref{prop:neck}.
The next lemma supplies this coordinate representation, the pole behavior
needed for smoothness, and the meridional-length identity used in the neck
interpolation.

\begin{lemma}
\label{lem:latitude-form}
Put $w:=w_{\varepsilon,U}$.  The metric
$g_{\mathrm{ang}}=D_{\varepsilon,U}^2h$ admits a latitude parameter
$\varphi\in[-\pi/2,\pi/2]$, with $\varphi=0$ at the equator and
$\varphi=\pm\pi/2$ at the two poles, and a smooth positive function
$A_{\mathrm{ang}}$ such that
\[
 g_{\mathrm{ang}}
 =A_{\mathrm{ang}}(\varphi)^2d\varphi^2
  +w^2\cos^2\varphi\,d\theta^2,
\]
and
\[
 \begin{gathered}
 A_{\mathrm{ang}}(\pm\pi/2)=w,\qquad
 \dfrac{A_{\mathrm{ang}}(\varphi)}w
 =1+O(\cos^2\varphi)
 \quad\text{as }\varphi\to\pm\pi/2,\\
 \displaystyle\int_{-\pi/2}^{\pi/2}A_{\mathrm{ang}}\,d\varphi
 =\pi\cos r_{\mathrm W}.
 \end{gathered}
\]
\end{lemma}

\begin{proof}
By Lemma~\ref{lem:lens-ode} and \eqref{eq:lensmetric}, on one
pole-to-equator half
\[
 h=du^2+b(u)^2d\theta^2,\qquad b'>0\quad(0<u<\ell),
 \qquad b''(\ell)<0.
\]
Moreover, $b(\ell)=a(r_U)$, so \eqref{eq:w} gives
$w=D_{\varepsilon,U}b(\ell)$.  Define
\[
 \cos\varphi=\frac{b(u)}{b(\ell)},\qquad
 0\leq\varphi\leq\frac{\pi}{2},
\]
where $\varphi=\pi/2$ at the pole and $\varphi=0$ at the equator.
Differentiation gives
\[
 -\sin\varphi\,\frac{d\varphi}{du}
 =\frac{b'(u)}{b(\ell)},
\]
and hence
\[
 A_{\mathrm{ang}}(\varphi)
 =D_{\varepsilon,U}\left|\frac{du}{d\varphi}\right|
 =\frac{w\sin\varphi}{b'(u(\varphi))}.
\]
This proves the asserted metric formula away from the equator.

Put $v=\ell-u$ on this half and $c=-b''(\ell)>0$.  The smooth even
reflection from Lemma~\ref{lem:lens-smooth} gives
\[
 b(\ell-v)=b(\ell)-\frac c2v^2+O(v^4).
\]
Comparing this with
$\cos\varphi=1-\varphi^2/2+O(\varphi^4)$ gives
\[
 \varphi=\sqrt{\frac{c}{b(\ell)}}\,v+O(v^3).
\]
Furthermore,
\[
 b'(\ell-v)=cv+O(v^3),\qquad
 \sin\varphi=\sqrt{\frac{c}{b(\ell)}}\,v+O(v^3),
\]
so
\[
 A_{\mathrm{ang}}(0)
 =\frac{w}{\sqrt{c\,b(\ell)}}>0.
\]
With the signed equatorial variable, the numerator and denominator in
the quotient defining $A_{\mathrm{ang}}$ are smooth odd functions;
their quotient is therefore smooth and even.  Assigning the opposite
sign to $\varphi$ on the reflected half gives the global range
$[-\pi/2,\pi/2]$ and a smooth positive even function
$A_{\mathrm{ang}}$.

At the pole $\varphi=\pi/2$, rotational smoothness from
Lemma~\ref{lem:lens-smooth} gives
\[
 b(u)=u+O(u^3),\qquad b'(u)=1+O(u^2),
 \qquad \cos\varphi=\frac{u}{b(\ell)}+O(u^3).
\]
Hence $\sin\varphi=1+O(u^2)$ and
\[
 \frac{A_{\mathrm{ang}}(\varphi)}w
 =\frac{\sin\varphi}{b'(u)}
 =1+O(u^2)=1+O(\cos^2\varphi),
\]
which gives $A_{\mathrm{ang}}(\pi/2)=w$ and the stated asymptotic there.
Evenness gives the same conclusions at $-\pi/2$.  Finally, since
$A_{\mathrm{ang}}=D_{\varepsilon,U}|du/d\varphi|$ and the two halves
are reflections of one another,
\[
 \int_{-\pi/2}^{\pi/2}A_{\mathrm{ang}}(\varphi)\,d\varphi
 =2D_{\varepsilon,U}\int_0^\ell du
 =2D_{\varepsilon,U}\ell
 =\pi\cos r_{\mathrm W}
\]
by \eqref{eq:D}.
\end{proof}

The latitude form places the matched boundary metric in the class treated
by the rotational boundary-metric interpolation in
\cite[Appendix~A]{PriorAngularConstruction}.  The following proposition
gives the version needed here: it joins such a metric to a small round
sphere while making the angular-end shape operator large enough to
dominate $S^{\mathrm{hole}}$.

\begin{proposition}
\label{prop:neck}
Fix parameters
\[
 0<r_{\mathrm W}<\pi/10,
 \qquad 1<q_*<V<\Gamma,
\]
and choose $\rho,w>0$ such that
\begin{equation}\label{eq:rho-w}
 0<\rho<\min\{r_{\mathrm W}/\pi,\Gamma^{-1}\},\qquad
 0<w<\rho^2,
 \qquad w<\frac{\rho^2\Gamma}{V^2}.
\end{equation}
Let $A_{\mathrm{ang}}:[-\pi/2,\pi/2]\to(0,\infty)$ be smooth, and suppose
that
\[
 g_{\mathrm{ang}}
 =A_{\mathrm{ang}}(\varphi)^2d\varphi^2
  +w^2\cos^2\varphi\,d\theta^2
\]
extends smoothly across the two poles $\varphi=\pm\pi/2$ to a rotational
metric on $\mathbb S^2$ and satisfies
\begin{equation}\label{eq:neck-hyp}
 K_{g_{\mathrm{ang}}}>\Gamma^2\quad\text{on }\mathbb S^2,\qquad
 A_{\mathrm{ang}}(\pm\pi/2)=w,\qquad
 \int_{-\pi/2}^{\pi/2}A_{\mathrm{ang}}\,d\varphi
       =\pi\cos r_{\mathrm W}.
\end{equation}

Then there exist $\ell_{\mathrm{neck}},\lambda>0$ and smooth positive
functions $R_{\mathrm{neck}}(s)$ and $F(s,\varphi)$ such that the following
is a smooth Ricci-positive metric on
$[0,\ell_{\mathrm{neck}}]\times\mathbb S^2$:
\[
 g_{\mathrm{neck}}
 =ds^2+R_{\mathrm{neck}}(s)^2
 \left(F(s,\varphi)^2d\varphi^2
       +\cos^2\varphi\,d\theta^2\right).
\]
Its two boundary ends have
the following properties.
\begin{enumerate}[label=\textup{(\roman*)}]
\item At the angular end $s=\ell_{\mathrm{neck}}$,
\[
 R_{\mathrm{neck}}(\ell_{\mathrm{neck}})=w,
 \qquad
 F(\ell_{\mathrm{neck}},\varphi)
 =\frac{A_{\mathrm{ang}}(\varphi)}w.
\]
Thus the induced metric is $g_{\mathrm{ang}}$, and, with respect to the
outward normal $+\partial_s$, both principal curvatures are greater than
$q_*$.

\item At the round end $s=0$,
\[
 R_{\mathrm{neck}}(0)=\frac\rho\lambda,
 \qquad F(0,\varphi)=1.
\]
Thus the induced metric is the round metric of radius $\rho/\lambda$, and,
with respect to the outward normal $-\partial_s$, both principal curvatures
are equal to $-\lambda$.
\end{enumerate}
Throughout the neck,
\begin{equation}\label{eq:neck-radius}
 R_{\mathrm{neck}}'>0,\qquad R_{\mathrm{neck}}''<0.
\end{equation}
Moreover,
\begin{equation}\label{eq:lambda-neck-length}
 \lambda>\frac{\rho}{w},\qquad \ell_{\mathrm{neck}}<1.
\end{equation}
\end{proposition}

\begin{proof}
\smallskip
\noindent\emph{Step 1. High-curvature slice path.}
Normalize the angular deformation by setting
\begin{equation}\label{eq:neck-u-alpha}
 \tau_*=\max_\varphi\frac{A_{\mathrm{ang}}(\varphi)}w,\qquad
 \eta=\frac{A_{\mathrm{ang}}/w-1}{\tau_*-1},\qquad
 p_0=\frac{\log \tau_*}{\log(\rho/w)}.
\end{equation}
The meridian identity in \eqref{eq:neck-hyp} and
$\rho<r_{\mathrm W}/\pi<\cos r_{\mathrm W}$ give
\[
 \tau_*w\geq\frac1\pi
 \int_{-\pi/2}^{\pi/2}A_{\mathrm{ang}}\,d\varphi
 =\cos r_{\mathrm W}>\rho.
\]
Thus $\tau_*>\rho/w>1$, so the definition of $p_0$ gives $p_0>1$.
A maximum of $A_{\mathrm{ang}}$ occurs away from the poles.  At such a
maximum the Gaussian curvature is
$(\tau_*w)^{-2}$, so the curvature hypothesis gives
\begin{equation}\label{eq:umax-bound}
 \tau_*w<\Gamma^{-1}<1.
\end{equation}
Moreover, \eqref{eq:umax-bound} and $w<\rho^2$ imply
\[
 \tau_*<\frac1w<\left(\frac\rho w\right)^2,
\]
so
\[
 1<p_0<2,\qquad \rho\tau_*^{-1/p_0}=w.
\]
The last identity is the reason for the choice of $p_0$: it makes the
slice path below start at radius $\rho$ and end at radius $w$.

For $F_\tau=1+(\tau-1)\eta$, set
\begin{equation}\label{eq:slice-path}
 \widetilde g_{\tau,v}
 =v^2\left(F_\tau^2d\varphi^2+\cos^2\varphi\,d\theta^2\right),
 \qquad v_0(\tau)=\rho \tau^{-1/p_0}.
\end{equation}
For $1\leq\tau\leq\tau_*$,
\[
 F_\tau
 =\frac{\tau_*-\tau}{\tau_*-1}
  +\frac{\tau-1}{\tau_*-1}\frac{A_{\mathrm{ang}}}{w}>0.
\]
By the definition of $p_0$,
\[
 \widetilde g_{1,v_0(1)}=\rho^2
 (d\varphi^2+\cos^2\varphi\,d\theta^2),
 \qquad
 \widetilde g_{\tau_*,v_0(\tau_*)}=g_{\mathrm{ang}}.
\]
At either pole, smooth rotational extension means that
$A_{\mathrm{ang}}/w$ has a smooth even expansion in
$x=\cos\varphi$, with value $1$ at $x=0$; equivalently, near each pole,
\[
 \frac{A_{\mathrm{ang}}}{w}
 =1+\cos^2\varphi\,\omega_\pm(\cos^2\varphi)
\]
for smooth functions $\omega_\pm$.  The affine interpolation
$F_\tau$ preserves these pole jets, so every metric in the path is
smooth.  In particular,
\[
 \eta=O(\cos^2\varphi),\qquad \eta'=O(\cos\varphi).
\]
With meridional arclength $dr=vF_\tau\,d\varphi$ and orbit radius
$f=v\cos\varphi$, the identity $K=-f_{rr}/f$ gives
\begin{equation}\label{eq:slice-K}
 K_{\tau,v}=\frac1{v^2}
 \left(\frac1{F_\tau^2}
 -\frac{(\tau-1)\eta'\tan\varphi}{F_\tau^3}\right).
\end{equation}
We claim that this curvature remains greater than $\Gamma^2$ along the
path $v=v_0(\tau)$.  If $\eta'\tan\varphi<0$, then the second term in
\eqref{eq:slice-K} is nonnegative.  Since $\eta\leq1$, one also has
$F_\tau\leq\tau$, and therefore
\[
 K_{\tau,v_0}\geq\frac1{v_0^2F_\tau^2}
 \geq\frac1{\tau^2v_0^2}
 \geq\frac1{\tau_*^2w^2}>\Gamma^2,
\]
where $\tau v_0(\tau)$ is increasing and \eqref{eq:umax-bound} was used.
It remains to exclude an interior minimum when
$q:=\eta'\tan\varphi\geq0$.  Put
\[
 N(\tau,\varphi)=1+(\tau-1)(\eta-q).
\]
Then
\[
 K_{\tau,v_0}=\frac{N}{v_0^2F_\tau^3}.
\]
For fixed $\varphi$, the numerator $N$ is affine in $\tau$, with
$N(1,\varphi)=1$ and $N(\tau_*,\varphi)>0$ because the terminal Gaussian
curvature is positive.  Hence $N>0$ on $[1,\tau_*]$, and logarithmic
differentiation is legitimate.  It gives
\begin{equation}\label{eq:K-log-derivative}
 \frac{d}{d\tau}\log K_{\tau,v_0}
 =\frac{\mathcal D(\tau,\varphi)}{p_0\tau},\qquad
 \mathcal D
 :=2+\frac{p_0\tau(\eta-q)}{N}
 -\frac{3p_0\tau\eta}{F_\tau}.
\end{equation}
For the following sign analysis, fix $\varphi$, so that
$\eta=\eta(\varphi)$ and $q=q(\varphi)$ are constants.  We show that
every zero of $\mathcal D$ has negative derivative.  Recall that
\[
 p_0>1,\qquad F_\tau>0,\qquad N>0,
 \qquad q\geq0,\qquad \eta\leq1.
\]
If $\eta>0$, direct simplification gives
\[
 F_\tau\mathcal D
 =2(1-\eta)+2\eta\tau(1-p_0)-\frac{p_0\tau q}{N}.
\]
Hence
\[
 \frac{d}{d\tau}(F_\tau\mathcal D)
 =2\eta(1-p_0)
 -p_0q\frac{1-\eta+q}{N^2}<0.
\]
The first term is strictly negative, while the second is nonpositive.
Thus, at any zero $\tau_0$ of $\mathcal D$,
\[
 F_{\tau_0}\mathcal D'(\tau_0)
 =\left.\frac{d}{d\tau}(F_\tau\mathcal D)\right|_{\tau=\tau_0}<0,
\]
and therefore $\mathcal D'(\tau_0)<0$.

If $\eta\leq0$, write
\[
 \mathcal D=2-\frac{p_0\tau}{F_\tau}P,
 \qquad P:=2\eta+\frac qN.
\]
Here
\[
 \frac d{d\tau}\left(\frac{\tau}{F_\tau}\right)
 =\frac{1-\eta}{F_\tau^2}>0,
 \qquad
 P'=-\frac{q(\eta-q)}{N^2}\geq0.
\]
At any zero $\tau_0$ of $\mathcal D$, one has
$P(\tau_0)=2F_{\tau_0}/(p_0\tau_0)>0$, and therefore
\[
 \mathcal D'(\tau_0)
 =-p_0\left[
   \left(\frac{\tau}{F_\tau}\right)'P
   +\frac{\tau}{F_\tau}P'
  \right]_{\tau=\tau_0}<0.
\]
Thus every zero of $\mathcal D$ has negative derivative.  Since
$K'_{\tau,v_0}/K_{\tau,v_0}=\mathcal D/(p_0\tau)$, every interior
critical point of $K_{\tau,v_0}$ is a strict local maximum.  Its minimum
on $[1,\tau_*]$ is therefore attained at an endpoint.  The endpoint
curvatures are $\rho^{-2}>\Gamma^2$ and
$K_{g_{\mathrm{ang}}}>\Gamma^2$; the pole values follow by continuity.
Consequently
\begin{equation}\label{eq:base-path-Gamma}
 K_{\tau,v_0}>\Gamma^2\qquad(1\leq \tau\leq \tau_*).
\end{equation}

\smallskip
\noindent\emph{Step 2. Slow parametrization.}
The identity defining $p_0$ and \eqref{eq:umax-bound} give
\begin{equation}\label{eq:V-room}
 \tau_*^{1/p_0-1/2}
 =\frac{\rho}{w\sqrt{\tau_*}}
 >\rho\sqrt{\frac{\Gamma}{w}}>V,
\end{equation}
where the final inequality is \eqref{eq:rho-w}.  The map
$p\mapsto\tau_*^{1/p_0-1/p}$ is continuous and strictly increasing on
$[p_0,2]$; its value at $p_0$ is $1<V$, while its value at $2$ is
greater than $V$ by \eqref{eq:V-room}.  Hence there is a unique
$p\in(p_0,2)$ satisfying the endpoint relation
\begin{equation}\label{eq:p-choice}
 \tau_*^{1/p_0-1/p}=V.
\end{equation}
Moreover,
\[
 \frac{\rho \tau^{-1/p}}{v_0(\tau)}
 =\tau^{1/p_0-1/p}\leq V,
\]
for $1\leq\tau\leq\tau_*$.  Since Gaussian curvature scales by the
inverse square of the metric scale,
\[
 K_{\tau,\rho\tau^{-1/p}}
 =\left(\frac{v_0(\tau)}{\rho\tau^{-1/p}}\right)^2
  K_{\tau,v_0}
 >\frac{\Gamma^2}{V^2}>1.
\]
Fix $\mu$ with $0<\mu<\Gamma^2/V^2-1$.  Then
\begin{equation}\label{eq:modified-path-K}
 K_{\tau,\rho \tau^{-1/p}}\geq1+\mu
 \qquad(1\leq \tau\leq \tau_*).
\end{equation}

We first identify the differential inequality required of the
parametrization.  For $L>2$ and any smooth increasing
$\tau:[1,L]\to[1,\tau_*]$, set
\[
 v(t)=\rho\tau(t)^{-1/p},\qquad
 \beta(t):=-\frac{d}{dt}\log v(t)
 =\frac1p\frac{\tau'(t)}{\tau(t)}.
\]
Thus $\beta$ is the logarithmic rate at which the slice scale $v$
decreases.  Consider the cone-type metric
\[
 \widehat g=dt^2+t^2v(t)^2
 \left(F_{\tau(t)}^2d\varphi^2+\cos^2\varphi\,d\theta^2\right).
\]
Away from the poles, let
$e_\theta=(tv\cos\varphi)^{-1}\partial_\theta$.  The radius of the
$\theta$-circle is $tv\cos\varphi$, and $\cos\varphi$ is independent of
$t$.  Hence its radial sectional curvature is
\begin{equation}\label{eq:neck-radial-theta}
 K_{\widehat g}(\partial_t,e_\theta)
 =-\frac{(tv)''}{tv}
 =\beta'+\frac{2\beta}{t}-\beta^2.
\end{equation}
Thus the parametrization should make $\beta'+2\beta/t$ positive and
larger than $\beta^2$.  At the same time, keeping $t\beta$ small
preserves the tangential curvature margin in
\eqref{eq:modified-path-K}.  We now construct a path with both
properties.

Fix a smooth nondecreasing function $\zeta:[0,\infty)\to[0,1]$ such that
$\zeta(0)=0$, $\zeta'(0)>0$, $\zeta(x)>0$ for $x>0$, and
$\zeta=1$ on $[1,\infty)$.  Set
\[
 \gamma_L:=
 \frac{\log\tau_*}{\displaystyle
  \int_1^L\frac{\zeta(s-1)}s\,ds}
\]
and, for $1\leq t\leq L$, define
\[
 \tau(t):=
 \exp\left(\gamma_L\int_1^t\frac{\zeta(s-1)}s\,ds\right).
\]
Then
\[
 \tau(1)=1,\qquad \tau(L)=\tau_*,\qquad
 v(1)=\rho,\qquad \frac{v(L)}w=V.
\]
Moreover,
\[
 \log(L/2)
 \leq\int_1^L\frac{\zeta(s-1)}s\,ds
 \leq\log L,
\]
so, as $L\to\infty$, $\gamma_L\asymp1/\log L$.  For this choice, the
definition of $\beta$ gives
\[
 \beta(t)=\frac1p\frac{\tau'(t)}{\tau(t)}=-\frac{v'(t)}{v(t)}
 =\frac{\gamma_L}{p}\frac{\zeta(t-1)}t.
\]
The key computation is
\[
 \beta'+\frac{2\beta}{t}
 =\frac{\gamma_L}{p\,t^2}
  \bigl(t\zeta'(t-1)+\zeta(t-1)\bigr).
\]
The factor in parentheses has a positive lower bound on $[1,2]$ and
equals $1$ on $[2,L]$.  Hence, for all sufficiently large $L$,
\begin{equation}\label{eq:slow-speed-estimates}
 \begin{aligned}
 \beta'+\frac{2\beta}{t}
 &\geq\frac{c}{t^2\log L},&
 |\beta'|&\leq\frac{C}{t^2\log L},\\
 \beta^2&\leq\frac{C}{t^2(\log L)^2},&
 0\leq t\beta&\leq\frac{C}{\log L}.
 \end{aligned}
\end{equation}
At the endpoints,
\begin{equation}\label{eq:slow-endpoint-speed}
 \beta(1)=0,\qquad
 L\beta(L)=\frac{\gamma_L}{p}
 \leq\frac{\log\tau_*}{p\log(L/2)}.
\end{equation}
Thus the path still reaches the prescribed terminal slice exactly, while
both its dimensionless speed $t\beta$ and its terminal speed
$L\beta(L)$ tend to zero as $L\to\infty$.

\smallskip
\noindent\emph{Step 3. Ricci positivity.}
We now verify that the cone-type metric $\widehat g$ introduced above has
positive Ricci curvature.  The margin in
\eqref{eq:modified-path-K} controls its tangential curvature, while the
slow variation supplies positive radial Ricci curvature.
Let $e_\varphi$ denote the unit meridional vector on each slice, and
retain $e_\theta$ from above.  The slice principal curvatures in these
two directions are
\[
 h_\varphi=\frac1t-\beta+p\beta X,
 \qquad h_\theta=\frac1t-\beta,
\]
where $X:=\tau\eta/F_\tau$, which is uniformly bounded on the whole
path.
The curvature of a slice is $t^{-2}K_{\tau,v}$, so the Gauss equation
gives
\[
 K_{\widehat g}(e_\varphi,e_\theta)
 =\frac{K_{\tau,v}}{t^2}-h_\varphi h_\theta,
 \qquad
 \left|h_\varphi h_\theta-\frac1{t^2}\right|
 \leq C\left(\frac{\beta}{t}+\beta^2\right).
\]
Since $t\beta\leq C/\log L$, equation
\eqref{eq:modified-path-K} implies, after increasing $L$, that
\begin{equation}\label{eq:neck-tangential-sectional}
 K_{\widehat g}(e_\varphi,e_\theta)\geq\frac{\mu}{2t^2}.
\end{equation}

Differentiating $X$ gives
\[
 X'=p\beta X(1-X).
\]
Using
$K_{\widehat g}(\partial_t,e_i)=-(h_i'+h_i^2)$, the other radial
sectional curvature is
\begin{equation}
 K_{\widehat g}(\partial_t,e_\varphi)
 =(1-pX)\left(\beta'+\frac{2\beta}{t}\right)
  -\beta^2\bigl(1+p(p-2)X\bigr).                             \label{eq:neck-radial-E}
\end{equation}
Adding this to \eqref{eq:neck-radial-theta} gives the exact identity
\begin{equation}\label{eq:radial-Ricci-neck}
 \Ric_{\widehat g}(\partial_t,\partial_t)
 =(2-pX)\left(\beta'+\frac{2\beta}{t}\right)
 -\beta^2\bigl(2+p(p-2)X\bigr).
\end{equation}
If $\eta>0$, then $0<X\leq1$, while $X\leq0$ if $\eta\leq0$.
Since $p<2$, one has $2-pX\geq2-p>0$, and all remaining
coefficients are bounded.

For all sufficiently large $L$, \eqref{eq:slow-speed-estimates} gives
\begin{equation}\label{eq:neck-normal-Ricci-lower}
 \Ric_{\widehat g}(\partial_t,\partial_t)
 \geq\frac{c}{t^2\log L}
\end{equation}
The same formulas give
\begin{equation}\label{eq:neck-radial-sectional-bound}
 |K_{\widehat g}(\partial_t,e_\varphi)|
 +|K_{\widehat g}(\partial_t,e_\theta)|
 \leq\frac{C}{t^2\log L}.
\end{equation}
Combining this with \eqref{eq:neck-tangential-sectional}, and enlarging
$L$ if necessary, gives
\begin{equation}\label{eq:neck-tangential-Ricci-lower}
 \Ric_{\widehat g}(e_\varphi,e_\varphi)\geq\frac{c}{t^2},\qquad
 \Ric_{\widehat g}(e_\theta,e_\theta)\geq\frac{c}{t^2}.
\end{equation}

It remains to control the mixed Ricci term.  The pole expansion
$\eta=O(\cos^2\varphi)$ gives
$\eta\tan\varphi=O(\cos\varphi)$, while $F_\tau$ is uniformly bounded
away from zero.  Hence
$\eta\tan\varphi/F_\tau^2$ extends continuously to zero at the poles and
is uniformly bounded on the whole path.  The contracted Codazzi identity
now gives the only nonzero mixed Ricci term:
\[
 \Ric_{\widehat g}(\partial_t,e_\varphi)
 =-(h_\varphi-h_\theta)\frac{\tan\varphi}{tvF_\tau}
 =-\frac{\eta \tau'\tan\varphi}{tvF_\tau^2}.
\]
Since $\tau'=p\beta\tau$ and $\tau/v$ is uniformly bounded,
\eqref{eq:slow-speed-estimates} gives
\[
 |\Ric_{\widehat g}(\partial_t,e_\varphi)|
 \leq\frac{C}{t^2\log L}.
\]
The other mixed entries vanish.  For all sufficiently large $L$,
\begin{equation}\label{eq:neck-Schur}
 \begin{aligned}
 \frac{|\Ric_{\widehat g}(\partial_t,e_\varphi)|^2}
 {\Ric_{\widehat g}(e_\varphi,e_\varphi)}
 &\leq\frac{C}{t^2(\log L)^2}
 \leq\frac12\Ric_{\widehat g}(\partial_t,\partial_t),\\
 \Ric_{\widehat g}(\partial_t,\partial_t)
 -\frac{|\Ric_{\widehat g}(\partial_t,e_\varphi)|^2}
 {\Ric_{\widehat g}(e_\varphi,e_\varphi)}
 &\geq\frac12\Ric_{\widehat g}(\partial_t,\partial_t)>0.
 \end{aligned}
\end{equation}
Since $\Ric_{\widehat g}(e_\varphi,e_\varphi)>0$,
\eqref{eq:neck-Schur} says precisely that the symmetric matrix
\[
 \begin{pmatrix}
  \Ric_{\widehat g}(\partial_t,\partial_t)
  &\Ric_{\widehat g}(\partial_t,e_\varphi)\\
  \Ric_{\widehat g}(\partial_t,e_\varphi)
  &\Ric_{\widehat g}(e_\varphi,e_\varphi)
 \end{pmatrix}
\]
is positive definite.  The direction $e_\theta$ is orthogonal to this
block for the Ricci tensor, and its Ricci curvature is positive by
\eqref{eq:neck-tangential-Ricci-lower}.  Hence
$\Ric_{\widehat g}>0$.

\smallskip
\noindent\emph{Step 4. Rescaling and boundary geometry.}
Put
\begin{align*}
 \delta_{\mathrm{neck}}&=\frac{w}{Lv(L)}=\frac1{LV},&
 s&=\delta_{\mathrm{neck}}(t-1),\\
 \lambda&=\delta_{\mathrm{neck}}^{-1}=LV,&
 \ell_{\mathrm{neck}}&=\delta_{\mathrm{neck}}(L-1).
\end{align*}
For $t=1+s/\delta_{\mathrm{neck}}$, define
\[
 F(s,\varphi):=F_{\tau(t)}(\varphi),\qquad
 R_{\mathrm{neck}}(s):=\delta_{\mathrm{neck}}t v(t).
\]
Then the rescaled metric is
\[
 g_{\mathrm{neck}}
 =\delta_{\mathrm{neck}}^2\widehat g
 =ds^2+R_{\mathrm{neck}}(s)^2
 \left(F(s,\varphi)^2d\varphi^2+\cos^2\varphi\,d\theta^2\right).
\]
At $s=0$, one has $F_{\tau(1)}=1$ and
$R_{\mathrm{neck}}(0)=\rho/\lambda$, so this end is round.  Its outward
normal is $-\partial_s$; since $\beta(1)=0$, both principal
curvatures are $-\lambda$.

At $s=\ell_{\mathrm{neck}}$, one has
$F_{\tau(L)}=A_{\mathrm{ang}}/w$ and
$R_{\mathrm{neck}}(\ell_{\mathrm{neck}})=w$, so the induced metric is
$g_{\mathrm{ang}}$.  With outward normal $+\partial_s$, the circle and
meridional principal curvatures there are, respectively,
\[
 V\bigl(1-L\beta(L)\bigr),\qquad
 V\bigl(1-L\beta(L)+pL\beta(L)X\bigr).
\]
Since $X$ is uniformly bounded, both are bounded below by
\begin{equation}\label{eq:terminal-shape-bound}
 V[1-CL\beta(L)].
\end{equation}
By \eqref{eq:slow-endpoint-speed} and $V>q_*$, all the preceding
estimates hold and both terminal shapes are greater than $q_*$ for all
sufficiently large $L$.  Since $\rho$, $w$, and $V$ are already fixed
and $L$ remains free, enlarge $L$ further, if necessary, so that
$L>\rho/(wV)$.  Then $LV>\rho/w$.  Fix this value of $L$.

It remains to verify the monotonicity and concavity required in
\eqref{eq:neck-radius}.  From the definition of $\beta$,
\[
 (tv)'=v(1-t\beta)>0,\qquad
 \frac{(tv)''}{tv}
 =-\left(\beta'+\frac{2\beta}{t}\right)+\beta^2<0,
\]
where the two inequalities follow from
\eqref{eq:slow-speed-estimates} for large $L$.  Since
$R_{\mathrm{neck}}(s)=\delta_{\mathrm{neck}}t v(t)$ and
$dt/ds=\delta_{\mathrm{neck}}^{-1}$, this gives
\[
 R_{\mathrm{neck}}'=(tv)'>0,\qquad
 R_{\mathrm{neck}}''
 =\frac{(tv)''}{\delta_{\mathrm{neck}}}<0.
\]
Finally, the explicit normalization gives
\[
 \lambda=LV>\frac\rho w,\qquad
 \ell_{\mathrm{neck}}=\frac{L-1}{LV}<\frac1V<1.
\]
This proves
\eqref{eq:lambda-neck-length} and completes the construction.
\end{proof}

\begin{proof}[Proof of Theorem~\ref{thm:attachment}]
Choose $U_0=U_0(r_{\mathrm W})>0$ and, for each
$U\in(0,U_0)$, $\varepsilon_0=\varepsilon_0(U,r_{\mathrm W})>0$ so
that Lemma~\ref{lem:uniform-concavity} applies whenever
$0<\varepsilon<\varepsilon_0$, and fix such $U$ and $\varepsilon$.
Lemma~\ref{lem:inverse-realization}
applied to the reconstruction \eqref{eq:ambient-profile} identifies the
metric on $\mathcal S_{r_{\mathrm W}}(o_0)$ with
$D_{\varepsilon,U}^2h$ by \eqref{eq:exact-metric}.  It also proves that
$\mathcal R$ is smooth and odd across $t=0$, with
$\mathcal R(0)=0$, $\mathcal R'(0)=1$, all even derivatives equal to
zero there, and $\mathcal R'>0$ on $[0,r_{\mathrm W}]$.
Lemma~\ref{lem:uniform-concavity} gives
$\mathcal R''<0$ on $(0,r_{\mathrm W}]$ and the positive axis limit
\eqref{eq:axis-curvature-limit}.  These conclusions verify the
hypotheses of Lemma~\ref{lem:R-extension}, which leaves $\mathcal R$
unchanged on $[0,r_{\mathrm W}]$ and extends it to
$[0,5r_{\mathrm W}]$ with
\[
 \mathcal R>0,\qquad \mathcal R'>0,\qquad \mathcal R''<0
 \qquad(0<t\leq5r_{\mathrm W}).
\]

For the doubly warped metric \eqref{eq:ambient}, the coordinate
two-planes are eigenspaces of the curvature operator, and their
sectional curvatures are $\tan^2r_{\mathrm W}$ times the three
expressions in \eqref{eq:ambient-sections}.  Since
$5r_{\mathrm W}<\pi/2$, the three signs above make all of these
curvatures positive for $0<t\leq5r_{\mathrm W}$.  At $t=0$, the first
expression in \eqref{eq:ambient-sections} remains $1$, smooth oddness
gives
$(\mathcal R'/\mathcal R)\tan t\to1$, and
\eqref{eq:axis-curvature-limit} gives a strictly positive limit for
$-\mathcal R''/\mathcal R$.  Thus $g_{\mathrm{amb}}$ extends smoothly
across the collapsed $\theta$-orbit and is sectionally positive there as
well.  This proves the asserted profile properties and the sectional
positivity of every subdomain considered below.

We next verify the position and topology of the three angular balls.
If $\pi(t,\zeta,\theta)=(t,\zeta)$, then \eqref{eq:ambient} gives
\[
 L_{g_{\mathrm{amb}}}(\gamma)
 \geq\cot r_{\mathrm W}\,
 L_{dt^2+\cos^2t\,d\zeta^2}(\pi\circ\gamma)
\]
for every curve $\gamma$.  Hence the distance from $o_{\zeta_0}$ is
at least $\cot r_{\mathrm W}\,\varrho_{\zeta_0}$.  Conversely, a minimizing
round geodesic from $(0,\zeta_0)$ to a point with
$\varrho_{\zeta_0}\leq4r_{\mathrm W}$ stays in
$0\leq t\leq4r_{\mathrm W}<5r_{\mathrm W}$.  It therefore has a lift
with constant $\theta$ and the same length multiplied by
$\cot r_{\mathrm W}$; its initial $\theta$-value is immaterial because
the $\theta$-orbit collapses at $o_{\zeta_0}$.  Hence
\[
 d_{g_{\mathrm{amb}}}
 \bigl(o_{\zeta_0},(t,\zeta,\theta)\bigr)
 =\cot r_{\mathrm W}\,\varrho_{\zeta_0}(t,\zeta),
\]
whenever $\varrho_{\zeta_0}\leq4r_{\mathrm W}$.  Combining this
equality with the preceding lower bound shows, for every
$0<s\leq4r_{\mathrm W}$, that $\mathcal B_s(o_{\zeta_0})$ is exactly
the metric ball of radius $s\cot r_{\mathrm W}$.  Because
$4r_{\mathrm W}<\pi/2$, the round geodesic segments used here are
minimizing, and the relevant distances are
\[
 \begin{gathered}
 \operatorname{rad}(\mathcal B_{r_{\mathrm W}})
 =r_{\mathrm W}\cot r_{\mathrm W},\qquad
 \operatorname{rad}(\mathcal B_{4r_{\mathrm W}})
 =4r_{\mathrm W}\cot r_{\mathrm W},\\
 d(o_0,o_{\pm2r_{\mathrm W}})
 =2r_{\mathrm W}\cot r_{\mathrm W},\qquad
 d(o_{2r_{\mathrm W}},o_{-2r_{\mathrm W}})
 =4r_{\mathrm W}\cot r_{\mathrm W}.
 \end{gathered}
\]

The separation of the two inner centers is
$4r_{\mathrm W}\cot r_{\mathrm W}$, which is greater than the sum
$2r_{\mathrm W}\cot r_{\mathrm W}$ of their radii, so the inner
closed balls are disjoint.  Each inner center is at distance
$2r_{\mathrm W}\cot r_{\mathrm W}$ from $o_0$, and this distance plus
the inner radius is $3r_{\mathrm W}\cot r_{\mathrm W}$, strictly less
than the outer radius $4r_{\mathrm W}\cot r_{\mathrm W}$.  Hence both
inner balls lie strictly inside the outer ball.  In the
$\theta$-orbit space, an angular ball is a round half-disk whose
diameter lies on $t=0$; restoring the $\theta$-circles and collapsing
them along that diameter produces a smooth three-ball.  It follows that
$\mathcal P_{\varepsilon,U}$ is a three-ball with the interiors of two
disjoint three-balls removed.  It is compact because it is a closed
subset of the compact outer ball.

Equation \eqref{eq:exact-metric} identifies the metric of
$\mathcal S_{r_{\mathrm W}}(o_0)$ with $D_{\varepsilon,U}^2h$.
Because the coefficients of $g_{\mathrm{amb}}$ are independent of
$\zeta$, translation in $\zeta$ is an isometry and carries this sphere
to $\Sigma_{\mathrm{lens}}$ and $\Sigma_{\mathrm{neck}}$.  Thus both
inner boundary metrics are isometric to $D_{\varepsilon,U}^2h$.

It remains to check their shape operators.  For an angular sphere
$\mathcal S_s(o_{\zeta_0})$, put $\delta=\zeta-\zeta_0$ and choose the
normal pointing outward from the angular ball.  The meridional
principal curvature is the geodesic curvature of a round circle of
radius $s$, with the ambient homothety included, and the circle
principal curvature is the normal derivative of
$\log\mathcal R$.  Explicitly,
\[
 k_m=\frac{\cot s}{\cot r_{\mathrm W}},\qquad
 k_\theta=\frac1{\cot r_{\mathrm W}}
 \frac{\mathcal R'}{\mathcal R}
 \frac{\sin t\cos\delta}{\sin s}.
\]
For either inner sphere, $s=r_{\mathrm W}$ and its defining equation is
$\cos t\cos\delta=\cos r_{\mathrm W}$.  Consequently
\[
 k_m=1,\qquad
 k_\theta=\tan t\,\frac{\mathcal R'}{\mathcal R}=:q_\theta.
\]
This normal points from the removed ball into
$\mathcal P_{\varepsilon,U}$, exactly as required in the definition of
$S^{\mathrm{hole}}$.  Strict
concavity and $\mathcal R(0)=0$ imply, for $t>0$,
\[
\mathcal R(t)-t\mathcal R'(t)
 =\int_0^t\bigl(\mathcal R'(s)-\mathcal R'(t)\bigr)\,ds\geq0.
\]
Every point of an inner angular sphere has $0\leq t\leq r_{\mathrm W}$.
Since $\tan t/t$ is increasing,
\begin{equation}\label{eq:angular-circle-bound}
 q_\theta
 \leq\frac{\tan t}{t}
 \leq q_*\qquad(0<t\leq r_{\mathrm W}).
\end{equation}
The inequalities $\mathcal R'>0$ and $\mathcal R>0$ show that
$q_\theta>0$ away from a pole, while smooth oddness gives
$q_\theta\to1$ at a pole.  Therefore
$S^{\mathrm{hole}}=\operatorname{diag}(1,q_\theta)$ and
$0<S^{\mathrm{hole}}\leq q_*\operatorname{Id}$.  This completes
\textup{(i)}.

Finally, apply the same general shape formulas with $s=4r_{\mathrm W}$,
$\zeta_0=0$, and the outward normal of the outer ball.  They give
\begin{equation}\label{eq:outer-shapes}
 k_m^{\mathrm{out}}
 =\frac{\cot(4r_{\mathrm W})}{\cot r_{\mathrm W}}>0,\qquad
k_\theta^{\mathrm{out}}
 =\frac1{\cot r_{\mathrm W}}\frac{\mathcal R'}{\mathcal R}
 \frac{\sin t\cos\zeta}{\sin(4r_{\mathrm W})}>0.
\end{equation}
The first curvature is positive because $4r_{\mathrm W}<\pi/2$.
At every regular point of the orbit-space semicircle, $t>0$ and
$\mathcal R'/\mathcal R>0$; its defining equation
$\cos t\cos\zeta=\cos(4r_{\mathrm W})>0$ also gives
$\cos\zeta>0$.  Hence the second curvature is positive there.

At
either collapsed-circle pole it has the smooth limit
$\cot(4r_{\mathrm W})/\cot r_{\mathrm W}$, equal to the first
principal curvature.  Thus both principal curvatures are positive and
$\Sigma_{\mathrm{out}}$ is strictly convex.  For a principal
orthonormal frame $(e_m,e_\theta)$, the Gauss equation reads
\[
 K_{\Sigma_{\mathrm{out}}}
 =K_{g_{\mathrm{amb}}}(e_m,e_\theta)
  +k_m^{\mathrm{out}}k_\theta^{\mathrm{out}}>0.
\]
The ambient sectional-curvature term and the product of the two
principal curvatures are both positive.  This proves \textup{(ii)} and
completes the proof.
\end{proof}

\begin{proof}[Proof of Corollary~\ref{cor:angular-attachments}]
Let $U_0$ be supplied by Theorem~\ref{thm:attachment}, and let $U_1$
be the constant in Lemma~\ref{lem:docking-windows}.  Since
$\sin r_{\mathrm W}/r_{\mathrm W}<1$,
\[
 q_*:=\frac{\tan r_{\mathrm W}}{r_{\mathrm W}}
 <\frac1{\cos r_{\mathrm W}}.
\]
Choose constants
\begin{equation}\label{eq:q-Gamma-margin}
 1<q_*<V<\Gamma<\Gamma_-<\frac1{\cos r_{\mathrm W}},
\end{equation}
and then choose
\[
 0<\rho<\min\left\{\frac{r_{\mathrm W}}\pi,\Gamma^{-1}\right\}.
\]
By \eqref{eq:charge-asymp} and $C_\varepsilon\to1$, after decreasing
$U_0\leq U_1$ if necessary, for every fixed $U\in(0,U_0)$ and all
sufficiently small $\varepsilon$,
\begin{equation}\label{eq:uniform-charge-margin}
 \frac{\mathfrak m}{\cos r_{\mathrm W}}\geq\Gamma_-.
\end{equation}
Fix such a $U$, and choose $\varepsilon_0=\varepsilon_0(U,r_{\mathrm W})$
so that Theorem~\ref{thm:attachment},
Lemma~\ref{lem:docking-windows}, and
\eqref{eq:uniform-charge-margin} all hold whenever
$0<\varepsilon<\varepsilon_0$.

Proposition~\ref{prop:high-action-lens} gives the normalized
Ricci-positive packet with boundary metric $D_{\varepsilon,U}^2h$.
Let $\Phi_{\mathrm L}$ be the boundary isometry with
$\Sigma_{\mathrm{lens}}$ supplied by
Theorem~\ref{thm:attachment}\textup{(i)}.
Lemma~\ref{lem:docking-windows} gives
\eqref{eq:angle-window}, with the quotient at $u=0,\ell$ interpreted by
its continuous endpoint limits.  Equation~\eqref{eq:strong-charge} also
gives $2\kappa\ell>\pi\cos r_{\mathrm W}$.  These are exactly the two
hypotheses of Lemma~\ref{lem:seam}, which gives
\[
 S^{\mathrm{pkt}}-S^{\mathrm{hole}}>0.
\]

The normal used to define $S^{\mathrm{hole}}$ points from the removed
ball into $\mathcal P_{\varepsilon,U}$, whereas the outward normal of
$\mathcal P_{\varepsilon,U}$ along this inner boundary points into the
removed ball.  Its outward shape operator is therefore
$-S^{\mathrm{hole}}$.  Since
$\II(X,Y)=g_{\partial}(SX,Y)$, the sum in
\eqref{eq:relative-shape-hypothesis} is
\[
 \II_{\mathrm{pkt}}+\Phi_{\mathrm L}^*\II_{\mathcal P}
 =D_{\varepsilon,U}^2h
 \bigl((S^{\mathrm{pkt}}-S^{\mathrm{hole}})\,\cdot,\cdot\bigr)>0.
\]
The packet is Ricci positive by
Proposition~\ref{prop:high-action-lens}; the pair of pants is
sectionally positive, and hence Ricci positive, by
Theorem~\ref{thm:attachment}.  Thus all hypotheses of
Proposition~\ref{prop:relative-gluing} hold at the packet seam.  This
proves \textup{(i)}.

For the other inner boundary, Lemma~\ref{lem:latitude-form} gives
\[
 g_{\mathrm{ang}}
 =A_{\mathrm{ang}}(\varphi)^2d\varphi^2
  +w_{\varepsilon,U}^2\cos^2\varphi\,d\theta^2,
\]
where the metric extends smoothly across the two poles and
\[
 A_{\mathrm{ang}}(\pm\pi/2)=w_{\varepsilon,U},
 \qquad
 \int_{-\pi/2}^{\pi/2}A_{\mathrm{ang}}\,d\varphi
 =\pi\cos r_{\mathrm W}.
\]
Moreover, \eqref{eq:angular-curvature-lower} and
\eqref{eq:uniform-charge-margin} give
\[
 K_{g_{\mathrm{ang}}}
 \geq\left(\frac{\mathfrak m}{\cos r_{\mathrm W}}\right)^2
 \geq\Gamma_-^2>\Gamma^2.
\]
For the fixed $U$, \eqref{eq:w-asymp} gives
$w_{\varepsilon,U}\to0$ as $\varepsilon\downarrow0$.  Reduce
$\varepsilon_0$ so that
\[
 w_{\varepsilon,U}
 <\min\left\{\rho^2,\frac{\rho^2\Gamma}{V^2}\right\}.
\]
The displayed profile identities, \eqref{eq:q-Gamma-margin}, the choice
of $\rho$, and these curvature and waist bounds verify
\eqref{eq:rho-w}--\eqref{eq:neck-hyp}.  Hence
Proposition~\ref{prop:neck} applies.  At the angular end of the resulting
neck, the induced metric is $g_{\mathrm{ang}}$ and is therefore isometric
to $\Sigma_{\mathrm{neck}}$.  With respect to the outward normal
$+\partial_s$, its shape operator satisfies
\[
 S^{\mathrm{neck}}>q_*\operatorname{Id}.
\]
On the other hand, Theorem~\ref{thm:attachment}\textup{(i)} and
\eqref{eq:angular-circle-bound} give
\[
 S^{\mathrm{hole}}=\operatorname{diag}(1,q_\theta),
 \qquad 0<q_\theta\leq q_*.
\]
Since $q_*>1$,
\[
 S^{\mathrm{hole}}\leq q_*\operatorname{Id},
 \qquad
 S^{\mathrm{neck}}-S^{\mathrm{hole}}>0.
\]
The pair-of-pants outward shape operator at this inner boundary is again
$-S^{\mathrm{hole}}$.  Thus, for the angular-end isometry
$\Phi_{\mathrm N}$,
\[
 \II_{\mathrm{neck}}+\Phi_{\mathrm N}^*\II_{\mathcal P}
 =g_{\mathrm{ang}}
 \bigl((S^{\mathrm{neck}}-S^{\mathrm{hole}})\,\cdot,\cdot\bigr)>0.
\]
Both the neck and the pair of pants are Ricci positive, so
Proposition~\ref{prop:relative-gluing} applies at this seam.

At the other end, Proposition~\ref{prop:neck}\textup{(ii)} and
\eqref{eq:lambda-neck-length} give a round metric of radius
$\rho/\lambda$, outward shape operator
$-\lambda\operatorname{Id}$, and
\[
 \lambda>\frac{\rho}{w_{\varepsilon,U}},
 \qquad \ell_{\mathrm{neck}}<1.
\]
A coordinate curve with $(\varphi,\theta)$ fixed joins the two ends and
has length $\ell_{\mathrm{neck}}$, so their distance is less than $1$.
This proves \textup{(ii)}.
\end{proof}

\section{The quantitative blocks and global construction}

\begin{proposition}
\label{prop:hybrid-block}
There are constants
\[
 Q_0\geq1,\qquad c_*,\rho_*,C_w,c_L,D_{\mathrm{acc}}>0
\]
such that the following holds.  For every $Q\geq Q_0$ there is
$0<w_Q\leq C_w/Q$, and, whenever
$0<\rho<\rho_*$ and $w_Q<c_*\rho^2$, there is a smooth
Ricci-positive annulus
\[
 H_{\mathrm{pre}}(Q,\rho)\cong\mathbb S^2\times[0,1]
\]
with the following properties:
\begin{enumerate}[label=\textup{(H\arabic*)}]
\item its input is round of radius $\rho/\lambda$, with outward shape
operator $-\lambda\operatorname{Id}$, where
\begin{equation}\label{eq:hybrid-input-summary}
 \lambda>\frac{\rho}{w_Q}>\frac1\rho>2;
\end{equation}
\item it contains a compact set $E_Q$, unchanged by all seam smoothings,
such that
\begin{equation}\label{eq:hybrid-action-summary}
 \int_{E_Q}\Scal\,dV\geq c_LQ,\qquad
 \sup_{x\in E_Q}\dist(x,\partial_{\mathrm{in}}H_{\mathrm{pre}})
 \leq D_{\mathrm{acc}}.
\end{equation}
\end{enumerate}
The other boundary admits Ricci-positive outgoing cylindrical extensions of
arbitrarily large radial length, each ending in a round strictly convex
boundary.
\end{proposition}

\begin{proof}
Fix \(r_{\mathrm W}\) satisfying \eqref{eq:r-range}, and put
\[
 q_*=\frac{\tan r_{\mathrm W}}{r_{\mathrm W}}.
\]
Since \(q_*<1/\cos r_{\mathrm W}\), choose
\[
 q_*<V<\Gamma<\frac1{\cos r_{\mathrm W}}.
\]
Recall that
\[
 \mathfrak m_{\varepsilon,U}
 :=\frac{2\kappa_{\varepsilon,U}\ell_{\varepsilon,U}}{\pi}
 =\frac2\pi\int_0^{\ell_{\varepsilon,U}}k_m\,du
\]
is the pole-to-equator meridional bending normalized by the round value
$\pi/2$; here $k_m\equiv\kappa_{\varepsilon,U}$ on the lens boundary.
For each fixed \(U\), equation \eqref{eq:charge-asymp} and
\(C_\varepsilon\to1\) give
\[
 \liminf_{\varepsilon\downarrow0}
 \mathfrak m_{\varepsilon,U}
 =1-\frac{U^3}{8}+O(U^6).
\]
Choose \(U>0\) below the small-\(U\) thresholds in
Proposition~\ref{prop:high-action-lens},
Theorem~\ref{thm:attachment}, and
Corollary~\ref{cor:angular-attachments}, and so small that
\[
 \liminf_{\varepsilon\downarrow0}
 \frac{\mathfrak m_{\varepsilon,U}}{\cos r_{\mathrm W}}
 >\Gamma.
\]
Set
\[
 c_*=\min\left\{1,\frac{\Gamma}{V^2}\right\},
 \qquad
 \rho_*=\min\left\{\frac{r_{\mathrm W}}{\pi},\Gamma^{-1}\right\}.
\]

For \(Q\geq1\), put \(\varepsilon_Q=Q^{-1}\).  After increasing \(Q_0\),
all the fixed-\(U\) conclusions just cited hold for every \(Q\geq Q_0\),
\(7\varepsilon_Q<U\), and
\begin{equation}\label{eq:uniform-Q-charge}
 \frac{\mathfrak m_{\varepsilon_Q,U}}{\cos r_{\mathrm W}}
 >\Gamma.
\end{equation}
Let
\[
 (C_Q,g_Q)=(C_{\varepsilon_Q,U},g_{\varepsilon_Q,U}),\qquad
 E_Q=\mathcal E_{\varepsilon_Q,U},\qquad
 w_Q=w_{\varepsilon_Q,U}.
\]
The waist asymptotic \eqref{eq:w-asymp} and
Proposition~\ref{prop:high-action-lens} give
\(C_w,c_L,D_{\mathrm{dep}}>0\).  These constants are independent of
\(Q\) and satisfy
\begin{equation}\label{eq:packet-family-bounds}
 w_Q\leq\frac{C_w}{Q},\qquad
 \int_{E_Q}\Scal_{g_Q}\,dV_{g_Q}\geq c_LQ,\qquad
 \sup_{x\in C_Q}\dist_{g_Q}(x,\partial C_Q)
 \leq D_{\mathrm{dep}}.
\end{equation}
The same proposition permits the packet gluing collar to be chosen
disjoint from \(E_Q\).

Theorem~\ref{thm:attachment} supplies the sectionally positive angular
pair of pants
\begin{equation}\label{eq:Omega}
 \Omega_Q=\mathcal B_{4r_{\mathrm W}}(o_0)\setminus
 \left(\operatorname{int}\mathcal B_{r_{\mathrm W}}(o_{2r_{\mathrm W}})
       \sqcup
       \operatorname{int}\mathcal B_{r_{\mathrm W}}(o_{-2r_{\mathrm W}})
 \right).
\end{equation}
By Corollary~\ref{cor:angular-attachments}\textup{(i)}, the packet
boundary is isometric to one inner boundary of \(\Omega_Q\), and the sum
of the two outward second fundamental forms is positive definite.
Proposition~\ref{prop:relative-gluing} therefore attaches \(C_Q\) there;
choose its smoothing collar disjoint from \(E_Q\).

Now let \(0<\rho<\rho_*\) and suppose that \(w_Q<c_*\rho^2\).
Write \(h_Q\) for the metric on the remaining inner boundary.
Lemma~\ref{lem:latitude-form} puts \(h_Q\) in the latitude form required
by Proposition~\ref{prop:neck}.  Equation
\eqref{eq:angular-curvature-lower} and
\eqref{eq:uniform-Q-charge} give
\[
 K_{h_Q}>\Gamma^2.
\]
Moreover, the definitions of \(\rho_*\) and \(c_*\) give
\[
 \rho<\min\left\{\frac{r_{\mathrm W}}{\pi},\Gamma^{-1}\right\},
 \qquad
 w_Q<\rho^2,\qquad
 w_Q<\frac{\rho^2\Gamma}{V^2}.
\]
Thus all the hypotheses of Proposition~\ref{prop:neck} hold.  The
resulting companion neck has angular-end shape operator
\[
 S^{\mathrm{neck}}>q_*\operatorname{Id}.
\]
Theorem~\ref{thm:attachment}\textup{(i)} gives
\(0<S^{\mathrm{hole}}\leq q_*\operatorname{Id}\), so
\[
 S^{\mathrm{neck}}-S^{\mathrm{hole}}>0.
\]
Proposition~\ref{prop:relative-gluing} therefore attaches the neck to the
remaining inner boundary.  Its exposed end is round of radius
\(\rho/\lambda\), has outward shape operator
\(-\lambda\operatorname{Id}\), and is at distance less than \(1\) from
the angular end, where \(\lambda>\rho/w_Q\).  Since
\(w_Q<c_*\rho^2\leq\rho^2\) and
\(\rho<r_{\mathrm W}/\pi<1/2\), this proves
\[
 \lambda>\frac{\rho}{w_Q}>\frac1\rho>2.
\]

After these two attachments, the resulting manifold
\(H_{\mathrm{pre}}(Q,\rho)\) is diffeomorphic to
\(\mathbb S^2\times[0,1]\).  Its other boundary is the unchanged outer
boundary of the angular pair of pants.  The metric on \(E_Q\) is
unchanged, so its action satisfies \eqref{eq:hybrid-action-summary} by
\eqref{eq:packet-family-bounds}.

It remains to bound the distance from the round input to \(E_Q\).  The
neck contributes less than \(1\).  In the pair of pants, the nearest
points of the two inner boundaries can be joined through \(o_0\) by a
path of length \(2r_{\mathrm W}\cot r_{\mathrm W}\).

\noindent Any point of the
packet boundary can be joined to the endpoint of this central path by
one of the two pole routes; the shorter route has length at most
\(\pi\cos r_{\mathrm W}\).  Finally, the packet depth costs at most
\(D_{\mathrm{dep}}\).  The relative smoothing collars may be chosen so
thin, and their metrics so \(C^0\)-close to the unsmoothed metrics, that
they add less than \(2\) to the concatenated route.  Hence
\begin{equation}\label{eq:access-explicit}
 D_{\mathrm{acc}}
 =3+2r_{\mathrm W}\cot r_{\mathrm W}
   +\pi\cos r_{\mathrm W}+D_{\mathrm{dep}}
\end{equation}
satisfies the distance estimate in
\eqref{eq:hybrid-action-summary}.  The packet collar was chosen disjoint
from \(E_Q\), while the neck seam lies on the other inner boundary and
the cylinder seam below lies on the outer boundary.  Hence all three
smoothing seams are disjoint from \(E_Q\), so its metric and scalar action
remain unchanged.

Finally, Theorem~\ref{thm:attachment}\textup{(ii)} says that the remaining
boundary has positive Gaussian curvature and a positive least outward
principal curvature \(k_{\mathrm{out}}\).
Lemma~\ref{lem:path-to-round} joins its induced metric to a round metric
through metrics of positive Gaussian curvature.  Apply
Theorem~\ref{thm:path-cylinder} with
\(\delta_{\mathrm{in}}=k_{\mathrm{out}}/2\).  At the initial seam, the
sum of the two outward shape operators is bounded below by
\[
 \frac{k_{\mathrm{out}}}{2}\operatorname{Id}>0,
\]
so Proposition~\ref{prop:relative-gluing} attaches the cylinder.  Since
the radial lengths in Theorem~\ref{thm:path-cylinder} tend to infinity,
the cylinder may be chosen, after reserving its gluing collar, to have
radial length at least any prescribed \(L_0\).  Its far boundary is
round and strictly convex.  This proves the proposition.
\end{proof}

We next record a packet estimate used both in the global induction and in
the proof of Proposition~\ref{prop:constructed-collapse}.

\begin{lemma}
\label{lem:packet-collapse}
Retain the fixed $U$ and the normalized packets
\[
 (C_Q,g_Q)=(C_{Q^{-1},U},g_{Q^{-1},U})
\]
chosen in the proof of Proposition~\ref{prop:hybrid-block}.  There are constants
$V_0,\delta>0$, independent of $Q$, such that, for every sufficiently
large $Q$, there exists $x_Q\in C_Q$ satisfying
\begin{equation}\label{eq:packet-volume-inradius}
 \operatorname{Vol}(C_Q,g_Q)\leq\frac{V_0}{Q},
 \qquad
 \dist_{g_Q}(x_Q,\partial C_Q)\geq\delta.
\end{equation}
\end{lemma}

\begin{proof}
Corollary~\ref{cor:angular-attachments}\textup{(i)} gives
$S^{\mathrm{pkt}}>S^{\mathrm{hole}}$, while
Theorem~\ref{thm:attachment}\textup{(i)} gives
$S^{\mathrm{hole}}>0$.  Hence the packet boundary is strictly outward
convex.  Every packet point has boundary distance at most
$D_{\mathrm{dep}}$ by \eqref{eq:packet-family-bounds}.  The inward normal
exponential map covers $C_Q$ up
to its cut locus.  If $J(t)$ is its normal Jacobian, then, before the cut
time,
\begin{align*}
 J(0)&=1,& (\log J)'(0)&=-H_{\partial C_Q}<0,\\
 (\log J)''&=-|S|^2-\Ric(\partial_t,\partial_t)\leq0.
\end{align*}
Here $S(t)$ is the shape operator of the parallel hypersurface reached by
the inward normal flow.

\noindent
Thus $J(t)\leq1$ up to the cut time.  Every such cut time is at most
$D_{\mathrm{dep}}$, because it is the length of a minimizing segment from
an interior point to the boundary.  The normal exponential map covers the
packet up to a set of measure zero, so the area formula gives
\[
 \operatorname{Vol}(C_Q,g_Q)
 \leq D_{\mathrm{dep}}\operatorname{Area}(\partial C_Q,h_Q).
\]
By Lemma~\ref{lem:latitude-form}, write the boundary metric as
\[
 h_Q=\alpha_Q(\varphi)^2d\varphi^2
      +w_Q^2\cos^2\varphi\,d\theta^2,
 \qquad
 \int_{-\pi/2}^{\pi/2}\alpha_Q(\varphi)\,d\varphi
 =\pi\cos r_{\mathrm W}.
\]
Using \eqref{eq:packet-family-bounds},
\begin{align*}
 \operatorname{Area}(\partial C_Q,h_Q)
 &=2\pi w_Q\int_{-\pi/2}^{\pi/2}
       \alpha_Q(\varphi)\cos\varphi\,d\varphi\\
 &\leq2\pi^2\cos r_{\mathrm W}\,w_Q
 \leq\frac{2\pi^2\cos r_{\mathrm W}\,C_w}{Q}.
\end{align*}
Thus the volume estimate holds with
$V_0:=2\pi^2D_{\mathrm{dep}}C_w\cos r_{\mathrm W}$.

For the inradius estimate, retain $\varepsilon_Q=Q^{-1}$ and take $x_Q$ at
the core coordinates $z=U/2$ and $\phi=0$.  Formula
\eqref{eq:Psi-z} and the local uniform limit established in the proof of
Lemma~\ref{lem:action} give a number $\eta>0$ such
that, for every large $Q$,
\[
 [U/3,2U/3]\times[-\eta,\eta]
 \subset\{(z,\phi):|\phi|<\Psi(z)\}.
\]
On this rectangle the orbit-space part of the scaled packet metric is
\begin{equation}\label{eq:packet-quotient-metric}
 D_{\varepsilon_Q,U}^2\left[
  \bigl(4c_{\varepsilon_Q}\varepsilon_Q
        A_{\varepsilon_Q}(z)z^{-3}\bigr)^2dz^2
  +\left(\frac{2\varepsilon_Q}{z}\right)^2d\phi^2\right].
\end{equation}
Equations \eqref{eq:ell-asymp} and \eqref{eq:D} show that
$D_{\varepsilon_Q,U}\varepsilon_Q$ is bounded above and below by positive
constants depending only on the fixed $U$ and $r_{\mathrm W}$.
Since $A_{\varepsilon_Q}$ is uniformly positive on $[U/3,2U/3]$, the metric in
\eqref{eq:packet-quotient-metric} is bounded below by
$c_0(dz^2+d\phi^2)$ for some $c_0=c_0(U,r_{\mathrm W})>0$.
Projection to the orbit space does not increase length, and every path from
$x_Q$ to the packet boundary, which consists of the two graph faces, must
leave the displayed rectangle.  Its length is therefore at least
$\sqrt{c_0}\min\{U/6,\eta\}$.  Taking this fixed positive number as
$\delta$ gives the second estimate in
\eqref{eq:packet-volume-inradius}.
\end{proof}

\begin{proof}[Proof of Theorem~\ref{thm:main}]
Fix constants
\[
 Q_0,c_*,\rho_*,C_w,c_L,D_{\mathrm{acc}}
\]
and the associated data $w_Q$, $H_{\mathrm{pre}}(Q,\rho)$, and $E_Q$
furnished by Proposition~\ref{prop:hybrid-block}.
After increasing $Q_0$ if necessary, fix the constants $V_0,\delta$ from
Lemma~\ref{lem:packet-collapse} for every $Q\geq Q_0$.
We start with the round cap
\[
 X_0=\left([0,\pi/4]\times\Sph^2,
 du^2+\sin^2u\,g_{\Sph^2}\right),
\]
smoothly collapsed at $u=0$, and take the collapsed point as the
basepoint $p$.  Its boundary $O_0$ is round of radius
$R_0=\sin(\pi/4)=2^{-1/2}$ and has outward principal curvature
$k_0=\cot(\pi/4)=1$.

Suppose inductively that a compact Ricci-positive prefix $X_{j-1}$ has
been constructed, with a single strictly convex round boundary
$O_{j-1}$ of radius $R_{j-1}>0$ and least outward principal curvature
$k_{j-1}>0$.  Set
$D_{j-1}:=\max_{x\in O_{j-1}}d_{X_{j-1}}(p,x)<\infty$.

The parameters at stage $j$ are chosen in the following order.  First choose
\begin{equation}\label{eq:rho-choice}
 0<\rho_j<\min\left\{\rho_*,
                    \frac{k_{j-1}R_{j-1}}2\right\}.
\end{equation}

Next, using $w_Q\leq C_w/Q$, choose $Q_j\geq Q_0$ so large that
\begin{equation}\label{eq:Q-choice}
 w_{Q_j}<c_*\rho_j^2,
 \qquad
 c_LQ_j\frac{2R_{j-1}/\rho_j}
 {D_{j-1}+1+2D_{\mathrm{acc}}R_{j-1}/\rho_j}\geq j.
\end{equation}
All quantities on the right side of \eqref{eq:Q-choice} are fixed before
$Q_j$ is selected.

Apply Proposition~\ref{prop:hybrid-block} with
$(Q,\rho)=(Q_j,\rho_j)$, and let $\lambda_j$ be the number in
\eqref{eq:hybrid-input-summary}.  In constructing this block, choose the
packet-seam smoothing support inside the original packet's
$\delta/4$ boundary-distance collar.  This is possible because the
relative gluing construction allows arbitrarily thin collars.  Scale the
metric of $H_{\mathrm{pre}}(Q_j,\rho_j)$ by $c_j^2$, where
$c_j:=R_{j-1}\lambda_j/\rho_j$.  The input metric and shape then satisfy
\[
 c_j\frac{\rho_j}{\lambda_j}=R_{j-1},\qquad
 -\frac{\lambda_j}{c_j}=-\frac{\rho_j}{R_{j-1}},\qquad
 k_{j-1}-\frac{\rho_j}{R_{j-1}}
 >\frac{k_{j-1}}2>0.
\]
Here the second identity uses the fact that principal curvatures are divided
by $c_j$ under the homothety, and the last inequality follows from
\eqref{eq:rho-choice}.  Relative Ricci-positive gluing joins the
scaled block to $X_{j-1}$.  The collar used at this round input is chosen
disjoint from the compact set $E_j$ defined below.

Let $E_j$ be the image of $E_{Q_j}$ in the scaled stage-$j$ block.  It has
scalar action
\begin{equation}\label{eq:scaled-action}
 \int_{E_j}\Scal\,dV
 =c_j\int_{E_{Q_j}}\Scal_{g_{Q_j}}\,dV_{g_{Q_j}}
 \geq c_Lc_jQ_j.
\end{equation}
The equality uses the three-dimensional homothety law
$\Scal_{c^2g}\,dV_{c^2g}=c\,\Scal_g\,dV_g$.

For each $x\in E_j$, choose an input point realizing the access estimate in
\eqref{eq:hybrid-action-summary}, join that point to $p$ by an old-prefix
path of length at most $D_{j-1}$, and concatenate the two paths.  Choose the
input gluing collar sufficiently thin and its Hermite smoothing sufficiently
$C^0$-close that all the resulting distance bounds increase by less than
$1$.  Their final lengths are therefore strictly smaller than
$D_{j-1}+1+D_{\mathrm{acc}}c_j$.  The metric and scalar action on $E_j$
are unchanged at this seam.  The outgoing-cylinder seam is at the other
boundary of the block and can likewise be chosen disjoint from $E_j$ and
these access paths.  Define
$R_j^{\mathrm{act}}:=D_{j-1}+1+D_{\mathrm{acc}}c_j$.
The radius inequality $\lambda_j>2$ from
\eqref{eq:hybrid-input-summary} implies
\begin{equation}\label{eq:cj-lower}
 c_j>\frac{2R_{j-1}}{\rho_j}.
\end{equation}
Since $c\mapsto c/(D_{j-1}+1+D_{\mathrm{acc}}c)$ is increasing on
$(0,\infty)$,
\eqref{eq:cj-lower} and \eqref{eq:Q-choice} yield
\begin{equation}\label{eq:action-radius-stage}
 \begin{aligned}
 \frac{c_Lc_jQ_j}{R_j^{\mathrm{act}}}
 &=\frac{c_Lc_jQ_j}{D_{j-1}+1+D_{\mathrm{acc}}c_j}\\
 &\geq c_LQ_j\frac{2R_{j-1}/\rho_j}
 {D_{j-1}+1+2D_{\mathrm{acc}}R_{j-1}/\rho_j}
 \geq j.
 \end{aligned}
\end{equation}

Only after \eqref{eq:action-radius-stage}, choose the outgoing cylinder
supplied by Proposition~\ref{prop:hybrid-block} with radial length so large
that its $c_j^2$-homothetic copy, after reserving smoothing collars at both
ends, contains an untouched closed radial slab
\begin{equation}\label{eq:protected-radial-slab}
 \mathcal Z_j=[\tau_j^-,\tau_j^+]\times\Sph^2,
 \qquad \tau_j^+-\tau_j^-\geq1,
\end{equation}
and attach this homothetic copy to the scaled outer boundary.  At its initial
seam, the positive sum of the two outward shape operators is the corresponding
unscaled sum divided by $c_j$, so the relative gluing inequality remains
strict.  Here $s$ is
arclength in the radial direction, so the cylinder metric is
$ds^2+h_j(s)$ on the slab.  The
projection to $s$ is $1$-Lipschitz; hence the two boundary components of
$\mathcal Z_j$ have mutual distance at least $1$.  Reserve collars at the
two ends of the cylinder that are disjoint from $\mathcal Z_j$.  Let
$X_j$ be the resulting
compact prefix, and denote its terminal round radius and least outward
principal curvature by $R_j>0$ and $k_j>0$; their positivity follows from
the final conclusion of Proposition~\ref{prop:hybrid-block}.  The possibly
very large length of this cylinder is now included in the finite number
$D_j$ at the next stage.  Nothing in \eqref{eq:action-radius-stage}
required an upper bound for that downstream length.  The set $E_j$, its
chosen access paths, and the slab $\mathcal Z_j$ all lie outside the
reserved terminal collar and are never changed at later stages.

This completes the induction.

Choose the underlying inclusions so that each $X_j$ embeds in
$\operatorname{int}X_{j+1}$, and let
$M=\bigcup_{j=0}^{\infty}\operatorname{int}X_j$ as an increasing union.
The metric on the
reserved terminal collar of $X_j$ is superseded once, when stage $j+1$ is
attached; away from that collar it is never changed again.  At every stage the relative
smoothing is supported in prescribed collars.  The internal lens and
angular-neck seams are smoothed inside $H_{\mathrm{pre}}$, and the outgoing
path-cylinder seam is smoothed at the end of the corresponding stage.
Attaching stage $j+1$ changes stage $j$
only in its reserved terminal collar and is disjoint from $E_j$ and
$\mathcal Z_j$.  The exhaustion is strict,
$X_j\Subset\operatorname{int}X_{j+1}$ after adjoining a short piece of
the next collar.  Each point is contained in a finite prefix and can meet
at most one later smoothing support.  Thus the metrics stabilize on a
neighborhood of every point and define a unique smooth metric $g$ on $M$,
with $\Ric_g>0$.

The paths and protected sets retained during the induction now give
\begin{equation}\label{eq:ball-containment}
 E_j\subset B_g(p,R_j^{\mathrm{act}}).
\end{equation}
The scalar action in \eqref{eq:scaled-action} is unchanged, and
$\Scal_g>0$ outside $E_j$ as well.  Hence
\begin{equation}\label{eq:final-stage-action}
 \frac1{R_j^{\mathrm{act}}}
 \int_{B_g(p,R_j^{\mathrm{act}})}\Scal_g\,dV_g
 \geq\frac{c_Lc_jQ_j}{R_j^{\mathrm{act}}}\geq j
\end{equation}
by \eqref{eq:action-radius-stage}.

The slabs $\mathcal Z_j$ are pairwise disjoint and ordered along the
product end.  Each $\mathcal Z_j$ separates the initial cap from every
later block.
Any curve escaping all compact prefixes must cross $\mathcal Z_j$ from
one boundary component to the other for every sufficiently large $j$.
By \eqref{eq:protected-radial-slab}, the radial projection is
$1$-Lipschitz and each such crossing has length at least $1$.  The disjoint
crossings therefore give infinite total length.  By the divergent-path
criterion, equivalently Hopf--Rinow, $(M,g)$ is complete.  The same
separation argument applies to any path from $p$ to $E_j$: it must cross
$\mathcal Z_1,\ldots,\mathcal Z_{j-1}$, and therefore
$d(p,E_j)\geq j-1$.  Together with \eqref{eq:ball-containment}, this gives
\begin{equation}\label{eq:radii-infinity}
 R_j^{\mathrm{act}}\longrightarrow\infty.
\end{equation}

The cap $X_0$ is a three-ball and every hybrid block is
$\Sph^2\times[0,1]$.  Choose a product identification of each entire
stage annulus, compatible with its two boundary identifications.  It extends
a diffeomorphism from $X_{j-1}$ onto a round ball to one from $X_j$ onto a
larger round ball.  These extensions may be chosen compatible on smaller
prefixes; their union is a diffeomorphism $M\to\mathbb R^3$.  In
particular, $M$ has one end.

Combining \eqref{eq:radii-infinity} with
\eqref{eq:final-stage-action} proves the assertion for $p$.

It remains to verify the conclusion for every basepoint.  If $p_1\in M$
and $d=d_g(p,p_1)$, then
$B_g(p,R)\subset B_g(p_1,R+d)$.
Using $R=R_j^{\mathrm{act}}$ and \eqref{eq:final-stage-action},
\[
 \frac1{R_j^{\mathrm{act}}+d}
 \int_{B_g(p_1,R_j^{\mathrm{act}}+d)}\Scal_g\,dV_g
 \geq
 j\frac{R_j^{\mathrm{act}}}{R_j^{\mathrm{act}}+d}
 \longrightarrow+\infty.
\]
Thus \eqref{eq:limsup} holds for every basepoint.
\end{proof}

\begin{proof}[Proof of Proposition~\ref{prop:constructed-collapse}]
First note that the induction forces $Q_j\to\infty$.

\noindent Indeed,
equation \eqref{eq:Q-choice} and
\[
 \frac{2R_{j-1}/\rho_j}
 {D_{j-1}+1+2D_{\mathrm{acc}}R_{j-1}/\rho_j}
 <\frac1{D_{\mathrm{acc}}}
\]
give
\begin{equation}\label{eq:Q-divergence}
 Q_j>\frac{D_{\mathrm{acc}}}{c_L}j.
\end{equation}
For every sufficiently large $j$, let $x_j$ denote the image, in the
$c_j^2$-scaled packet, of the point supplied by
Lemma~\ref{lem:packet-collapse} in $C_{Q_j}$.  By the collar choice made
during the stage construction, the corresponding unscaled point has
boundary distance at least $\delta$.  Hence its distance from the chosen
$\delta/4$ packet-boundary collar is at least $3\delta/4$.

\noindent After scaling,
every path from $x_j$ to that collar has length at least
$3c_j\delta/4$ before it reaches a region where the packet metric may have
changed.  Every path that exits the packet must first enter this boundary
collar.  Consequently, if $s_j:=\delta c_j/4$, then the global ball
$B_g(x_j,s_j)$ is contained in the untouched part of the scaled packet, and
\[
 \operatorname{Vol}_g B_g(x_j,s_j)\leq\frac{V_0c_j^3}{Q_j}.
\]
Later stages do not alter this packet.  Hence
\begin{equation}\label{eq:stage-volume-ratio}
 \frac{\operatorname{Vol}_g B_g(x_j,s_j)}{s_j^3}
 \leq\frac{64V_0}{\delta^3Q_j}.
\end{equation}
Equations \eqref{eq:stage-volume-ratio} and \eqref{eq:Q-divergence}
prove \eqref{eq:collapsing-balls-intro} with
$C_{\mathrm{col}}=64V_0/\delta^3$.  Finally, Bishop--Gromov
monotonicity at $x_j$ gives
\[
 \operatorname{AVR}(M,g)
 \leq\frac{\operatorname{Vol}_g B_g(x_j,s_j)}{(4\pi/3)s_j^3}
 \leq\frac{48V_0}{\pi\delta^3Q_j}.
\]
The asymptotic volume ratio is independent of the center.  Letting
$j\to\infty$ proves $\operatorname{AVR}(M,g)=0$.
\end{proof}

\appendix

\section{Relative Ricci-positive boundary tools}
\label{app:boundary-tools}

The Gaussian formulas and path-cylinder construction are recorded in
\cite[Lemmas~3.1--3.2 and Theorem~3.3]{PriorAngularConstruction}.  The
proof of Proposition~\ref{prop:relative-gluing} below gives the localized
form of Perelman's interpolation argument~\cite[Section~4]{PerelmanGluing}
used in this paper.

\begin{lemma}\label{lem:gaussian-ricci}
Let
\begin{equation}\label{eq:gaussian-metric}
 g=ds^2+h(s)
\end{equation}
be a metric on a product collar, and put $L=\tfrac12\dot h$.  Thus $L$ is
the second fundamental form of an $s$-slice with respect to $\partial_s$.
For vectors $X,Y$ tangent to an $s$-slice,
\begin{align}
 \Ric_g(\partial_s,\partial_s)
 &=-\frac12\operatorname{tr}_h\ddot h
   +\frac14\operatorname{tr}
      \bigl((h^{-1}\dot h)^2\bigr),                  \label{eq:gauss-normal}\\
 \Ric_g(\partial_s,X)
 &=(\operatorname{div}_hL)(X)-d(\operatorname{tr}_hL)(X),
                                                        \label{eq:gauss-mixed}\\
 \Ric_g(X,Y)
 &=\Ric_h(X,Y)-\frac12\ddot h(X,Y)
   +\frac12(\dot h\,h^{-1}\dot h)(X,Y)\notag\\
 &\hspace{28mm}
   -\frac14\operatorname{tr}_h(\dot h)\dot h(X,Y).
                                                        \label{eq:gauss-tangent}
\end{align}
\end{lemma}

\begin{proof}
We use the conventions
\[
 \begin{aligned}
 R(U,V)W
 &=\nabla_U\nabla_VW-\nabla_V\nabla_UW-\nabla_{[U,V]}W,\\
 \operatorname{Rm}(U,V,W,Z)&=g(R(U,V)W,Z),
 \end{aligned}
\]
and $(\operatorname{div}_hL)_i=\nabla^jL_{ji}$.  In product coordinates the
dot on the covariant tensor $L$ means differentiation under the product
identification, and the non-slice Christoffel symbols are
\[
 \Gamma^s_{ij}=-L_{ij},\qquad
 \Gamma^k_{si}=\Gamma^k_{is}=(h^{-1}L)^k{}_i.
\]
The Gauss--Codazzi equations give
\begin{align*}
 \operatorname{Rm}_g(\partial_s,X,Y,\partial_s)
   &=-\dot L(X,Y)+(Lh^{-1}L)(X,Y),\\
 \Ric_g(\partial_s,X)
   &=(\operatorname{div}_hL)(X)-d(\operatorname{tr}_hL)(X),\\
 \operatorname{Rm}_g(X,Y,Z,W)
   &=\operatorname{Rm}_h(X,Y,Z,W)-L(X,W)L(Y,Z)\\
   &\hspace{31mm}+L(X,Z)L(Y,W).
\end{align*}
Tracing and substituting $L=\tfrac12\dot h$ proves
\eqref{eq:gauss-normal}--\eqref{eq:gauss-tangent}.
\end{proof}

\begin{proposition}
\label{prop:relative-gluing}
Let $(M_i,g_i)$, $i=1,2$, be compact manifolds with
$\Ric_{g_i}>0$.  Suppose unions of boundary components
$N_i\subset\partial M_i$ are
identified by an isometry $\Phi:N_1\to N_2$.  If, after pullback by
$\Phi$, the sum of the second fundamental forms with respect to the two
outward normals is positive definite,
\begin{equation}\label{eq:relative-shape-hypothesis}
 \II_1+\Phi^*\II_2>0,
\end{equation}
then $M_1\cup_\Phi M_2$ carries a smooth metric of positive Ricci
curvature.  Given arbitrary collars of $N_i$, the new metric can be made
equal to $g_i$ outside those collars.  The collars may be chosen with
arbitrarily small normal width, the replacement can be made arbitrarily
$C^0$-close (and hence pointwise bilipschitz-close) to the continuous glued
metric, and any prescribed compact sets disjoint from the collars are left
unchanged.  The same statement
holds simultaneously for finitely many disjoint seams.
\end{proposition}

\begin{proof}[Proof of Proposition~\ref{prop:relative-gluing}]
Identify the boundaries, write $N$ for the resulting common seam, and use a
signed Fermi coordinate increasing from $M_1$ to $M_2$:
\[
 \begin{gathered}
 g=ds^2+h_-(s)\quad(s\leq0),\qquad
 g=ds^2+h_+(s)\quad(s\geq0),\\
 h_-(0)=h_+(0)=h_0.
 \end{gathered}
\]
The two outward normals are $+\partial_s$ and $-\partial_s$.  Therefore
\begin{equation}\label{eq:normal-derivative-gap}
 G:=\dot h_-(0)-\dot h_+(0)
   =2(\II_1+\Phi^*\II_2)>0.
\end{equation}
For small $\delta>0$, put
\[
 \tau=\frac{s+\delta}{2\delta}\in[0,1].
\]
The tensor-valued cubic Hermite polynomial that matches $h_-$ and
$\dot h_-$ at $s=-\delta$ and $h_+$ and $\dot h_+$ at
$s=\delta$ is
\[
 \begin{aligned}
 h_\delta(s)
 ={}&(2\tau^3-3\tau^2+1)h_-(-\delta)
   +2\delta(\tau^3-2\tau^2+\tau)\dot h_-(-\delta)\\
 &+(-2\tau^3+3\tau^2)h_+(\delta)
   +2\delta(\tau^3-\tau^2)\dot h_+(\delta).
 \end{aligned}
\]
Indeed, since $\partial_s=(2\delta)^{-1}\partial_\tau$, this formula
gives
\[
 \begin{gathered}
 h_\delta(-\delta)=h_-(-\delta),\qquad
 \dot h_\delta(-\delta)=\dot h_-(-\delta),\\
 h_\delta(\delta)=h_+(\delta),\qquad
 \dot h_\delta(\delta)=\dot h_+(\delta).
 \end{gathered}
\]
Differentiating twice yields
\[
 \begin{aligned}
 \ddot h_\delta(s)
 ={}&\frac{6(2\tau-1)}{4\delta^2}
       \bigl(h_-(-\delta)-h_+(\delta)\bigr)\\
 &+\frac1{2\delta}
   \bigl((6\tau-4)\dot h_-(-\delta)
          +(6\tau-2)\dot h_+(\delta)\bigr).
 \end{aligned}
\]
Taylor expansion of the endpoint data at the seam $s=0$ gives
\[
 \begin{aligned}
 h_-(-\delta)-h_+(\delta)
 &=-\delta\bigl(\dot h_-(0)+\dot h_+(0)\bigr)+O(\delta^2),\\
 \dot h_-(-\delta)&=\dot h_-(0)+O(\delta),\qquad
 \dot h_+(\delta)=\dot h_+(0)+O(\delta).
 \end{aligned}
\]
Writing $p=\dot h_-(0)$ and $q=\dot h_+(0)$, the terms of order
$\delta^{-1}$ combine as
\[
 -\frac{3(2\tau-1)}{2\delta}(p+q)
 +\frac{(6\tau-4)p+(6\tau-2)q}{2\delta}
 =\frac{q-p}{2\delta}=-\frac{G}{2\delta}.
\]
Thus all $\tau$-dependence cancels at leading order, uniformly for
$\tau\in[0,1]$.  The same explicit formula gives, uniformly in the
indicated spatial norms,
\begin{equation}\label{eq:hermite-estimates}
 \begin{aligned}
 h_\delta&=h_0+O(\delta)&&\text{in }C^2,\\
 \dot h_\delta&=O(1)&&\text{in }C^1,\\
 \ddot h_\delta&=-\frac{G}{2\delta}+O(1)&&\text{in }C^0.
 \end{aligned}
\end{equation}
Since $h_0$ is positive definite and the seam is compact, the first
estimate shows that $h_\delta$ is a Riemannian metric for all sufficiently
small $\delta$.
Insert \eqref{eq:hermite-estimates} in
\eqref{eq:gauss-normal}--\eqref{eq:gauss-tangent}.  On the open
interpolation collar,
\begin{align}
 \Ric(\partial_s,\partial_s)
   &=\frac1{4\delta}\operatorname{tr}_{h_\delta}G+O(1),\notag\\
 \left.\Ric\right|_{TN\times TN}
   &=\frac{G}{4\delta}+O(1),                         \label{eq:gluing-ricci}\\
 \Ric(\partial_s,\mathord\cdot)&=O(1).\notag
\end{align}
Compactness and $G>0$ give constants $c,C>0$ independent of small
$\delta$ such that
\[
 \Ric(\partial_s,\partial_s)\geq\frac c\delta-C,
 \qquad
 \left.\Ric\right|_{TN\times TN}
 \geq\left(\frac c\delta-C\right)h_\delta,
 \qquad
 \|\Ric(\partial_s,\mathord\cdot)\|_{h_\delta}\leq C.
\]
After decreasing $\delta$, the first two lower bounds are at least
$c/(2\delta)$.  In particular, the tangential restriction
$\left.\Ric\right|_{TN\times TN}$ is positive definite and
\[
 \left(\left.\Ric\right|_{TN\times TN}\right)^{-1}
 \leq\frac{2\delta}{c}h_\delta^{-1}.
\]
Its Schur complement in the splitting
$\mathbb R\partial_s\oplus TN$ is consequently bounded below by
\begin{equation}\label{eq:gluing-schur}
 \begin{aligned}
 &\Ric(\partial_s,\partial_s)
 -\left(\left.\Ric\right|_{TN\times TN}\right)^{-1}
   \left(
    \left.\Ric(\partial_s,\mathord\cdot)\right|_{TN},
    \left.\Ric(\partial_s,\mathord\cdot)\right|_{TN}
   \right)\\
 &\hspace{35mm}\geq
 \frac{c}{2\delta}-\frac{2C^2\delta}{c}>0.
 \end{aligned}
\end{equation}
Thus the $C^1$ Hermite metric has positive Ricci curvature wherever it
is smooth.

It remains to smooth the two joins at $s=\pm\delta$.  Translate one join
to $t=0$ and write the tangential metric as a $C^1$ tensor $k(t)$,
smooth on either side.  For an even nonnegative mollifier
$\varphi_\eta$ supported in $(-\eta,\eta)$, put
$k_\eta:=\varphi_\eta*k$ and let
$\ddot k_{\rm pw}$ denote the classical second derivative of $k$ on its
two smooth sides.  Since $k$ is $C^1$, convolution in $t$ gives
\[
 \ddot k_\eta(t)=\int\varphi_\eta(r)\ddot k_{\rm pw}(t-r)\,dr.
\]
Convolution acts only in the normal variable and commutes with all spatial
derivatives.  On a fixed neighborhood of the join,
$k_\eta\to k$ uniformly in spatial $C^2$ and
$\dot k_\eta\to\dot k$ uniformly in spatial $C^1$.  The Ricci expressions
in Lemma~\ref{lem:gaussian-ricci} are affine in $\ddot k$ and depend
smoothly on the other displayed spatial jets.  Hence, in tensor $C^0$
relative to a fixed background metric on the compact seam,
\begin{equation}\label{eq:gluing-mollified-ricci}
 \Ric_{dt^2+k_\eta(t)}
 =\int\varphi_\eta(r)\Ric_{dt^2+k(t-r)}\,dr+o_\eta(1).
\end{equation}
First fix $\delta$ so that the Ricci tensors on the Hermite and original
sides are bounded below by $\nu_\delta g$ near the join.  Then choose
$\eta$ so that the error in \eqref{eq:gluing-mollified-ricci} is less than
$\nu_\delta/2$.  The mollified metric is Ricci positive.

To localize, choose $\chi_\eta$ equal to one on $|t|\leq2\eta$, supported
in $|t|<\sqrt\eta$, and satisfying
$|\chi_\eta'|\leq C\eta^{-1/2}$ and
$|\chi_\eta''|\leq C\eta^{-1}$.  Define
\[
 \widetilde k_\eta:=k+\chi_\eta(k_\eta-k).
\]
On the transition annulus $2\eta\leq|t|\leq\sqrt\eta$, convolution
samples only one smooth side of $k$.  Evenness of the mollifier and Taylor
expansion therefore give, in spatial norms,
\[
 \|\partial_t^j(k_\eta-k)\|_{C^{2-j}(N)}=O(\eta^2),
 \qquad j=0,1,2.
\]
The product rule then gives
$\|\widetilde k_\eta-k\|_{C^2}=O(\eta)$ on this annulus.  Thus
$dt^2+\widetilde k_\eta$ equals the Ricci-positive mollified metric near
the join, equals the original metric outside $|t|<\sqrt\eta$, and remains
Ricci positive on the transition annulus.

\noindent First choose $\delta$ small
enough that the Hermite collar lies in the prescribed collars and has the
Ricci bounds above; then choose $\eta\ll\delta^2$ small enough for the
mollification and localization estimates.  This also proves the asserted
$C^0$ and support control.  Finitely many disjoint seams are treated
independently.
\end{proof}

\begin{lemma}
\label{lem:path-to-round}
Every smooth metric $h$ on $\mathbb S^2$ with positive Gaussian
curvature is joined to a round metric by a smooth path $(h_q)_{0\leq q\leq1}$
of positive-Gaussian-curvature metrics which is constant near both
endpoints.
\end{lemma}

\begin{proof}
Run the area-normalized Ricci flow from $h$.  Hamilton's surface
theorem~\cite{HamiltonSurface} gives smooth convergence on the fixed sphere
to a constant-positive-curvature metric $h_\infty$.  The flow remains in
the conformal class of $h$, so $h=e^{2u}h_\infty$ for a smooth function
$u$.  For $0\leq s\leq1$, set
\[
 \widetilde h_s=e^{2su}h_\infty.
\]
The conformal curvature formula gives
\[
 K_{\widetilde h_s}
 =e^{-2su}\bigl((1-s)K_{h_\infty}+s e^{2u}K_h\bigr)>0.
\]
Thus $\widetilde h_s$ joins $h_\infty$ to $h$ through metrics of positive
Gaussian curvature.  Choose a smooth nondecreasing map
$\psi:[0,1]\to[0,1]$, constant near both endpoints, with
$\psi(0)=0$ and $\psi(1)=1$, and define
\[
 h_q:=e^{2(1-\psi(q))u}h_\infty.
\]
Then $h_0=h$, $h_1=h_\infty$, the path is constant near its endpoints,
and the preceding curvature formula shows that every $h_q$ has positive
Gaussian curvature.
\end{proof}

\begin{theorem}
\label{thm:path-cylinder}
Let $N^n$ be closed and let $(g_q)_{0\leq q\leq1}$ be a smooth path with
$\Ric_{g_q}>0$, constant near its endpoints.  Given
$\delta_{\rm in}>0$, there are constants $a,\sigma>0$ and
$\ell_0$ such that for every $\ell\geq\ell_0$ there is a
Ricci-positive cylinder $N\times[0,L_\ell]$, where $L_\ell$ is its radial
length.  Writing $A_{\rm in}$ and $A_{\rm out}$ for its two outward shape
operators, its input induced metric is $g_0$, and its input outward shape
operator satisfies $A_{\rm in}>-\delta_{\rm in}\operatorname{Id}$.  Its
output induced metric is $\Lambda_\ell^2g_1$, and its output outward shape
operator satisfies $A_{\rm out}>0$.  More precisely,
\begin{equation}\label{eq:pc-exact-data}
 \begin{aligned}
 \Lambda_\ell&=e^{\ell-\sigma},&
 L_\ell&=\frac{e^\ell-1}{a},\\
 A_{\rm in}&=-a\left(1-\frac{2\sigma}{\ell}\right)\operatorname{Id},&
 A_{\rm out}&=ae^{-\ell}
       \left(1-\frac{\sigma}{2\ell}\right)\operatorname{Id}.
 \end{aligned}
\end{equation}
In particular, $L_\ell\to\infty$ and the output remains strictly convex.
\end{theorem}

\begin{proof}
Choose $\mu,C>0$ so that
\begin{equation}\label{eq:pc-uniform-bounds}
 \Ric_{g_q}\geq\mu g_q,\qquad
 |\dot g_q|+|\nabla^{g_q}\dot g_q|+|\ddot g_q|\leq C.
\end{equation}
Here and below all norms in the second estimate are measured using
$g_q$.  Choose $a,\sigma>0$ with
\begin{equation}\label{eq:pc-parameter-choice}
 a<\delta_{\rm in},\qquad
 \frac{\mu}{a^2}>2n,\qquad
 n\sigma-\frac12\sup_{q,N}
       |\operatorname{tr}_{g_q}\dot g_q|\geq1.
\end{equation}
For $t\in[e^\ell,e^{2\ell}]$ set
\begin{equation}\label{eq:pc-ansatz}
 q_\ell(t)=2\left(1-\frac{\ell}{\log t}\right),\qquad
 f(t)=at\,e^{-\sigma q_\ell(t)},\qquad
 \bar g=dt^2+f(t)^2g_{q_\ell(t)}.
\end{equation}
Writing $q_\ell=q_\ell(t)$ below, one has
\begin{equation}\label{eq:pc-derivatives}
 q_\ell'=\frac{2\ell}{t(\log t)^2},\qquad
 q_\ell''=-\frac{q_\ell'}t
       \left(1+\frac2{\log t}\right),\qquad
 \frac{f'}f=\frac1t-\sigma q_\ell'.
\end{equation}
The slice shape endomorphism with respect to $+\partial_t$ is
\begin{equation}\label{eq:pc-slice-shape}
 \mathcal A_t=\left(\frac1t-\sigma q_\ell'\right)\operatorname{Id}
 +\frac{q_\ell'}2g_{q_\ell}^{-1}\dot g_{q_\ell}.
\end{equation}

Put $T(q)=\operatorname{tr}_{g_q}\dot g_q$.  Substitution in
Lemma~\ref{lem:gaussian-ricci}, using the first two identities in
\eqref{eq:pc-derivatives}, gives
\[
 \Ric_{\bar g}(\partial_t,\partial_t)
 =\left(n\sigma-\frac12T(q_\ell)\right)
   \left(q_\ell''+\frac{2q_\ell'}t\right)
   +O((q_\ell')^2),
\]
where
$q_\ell''+2q_\ell'/t=(q_\ell'/t)(1-2/\log t)$.  Thus the normal component
is
\begin{equation}\label{eq:pc-normal}
 \Ric_{\bar g}(\partial_t,\partial_t)
 =\left(n\sigma-\frac12\operatorname{tr}_{g_{q_\ell}}\dot g_{q_\ell}\right)
   \frac{q_\ell'}t\left(1-\frac2{\log t}\right)
   +O((q_\ell')^2).
\end{equation}
The error constant is independent of $\ell$.  On
$t\in[e^\ell,e^{2\ell}]$,
\[
 \frac1{2t^2\ell}\leq\frac{q_\ell'}t
 \leq\frac2{t^2\ell},
 \qquad
 1-\frac2{\log t}\geq\frac12
 \quad(\ell\geq4).
\]
The last condition in \eqref{eq:pc-parameter-choice} makes the first
factor in \eqref{eq:pc-normal} at least one.  Since
$(q_\ell')^2=O((t^2\ell^2)^{-1})$, the error can be absorbed for all
sufficiently large $\ell$, giving
\begin{equation}\label{eq:pc-normal-lower}
 \Ric_{\bar g}(\partial_t,\partial_t)
 \geq\frac1{8t^2\ell}.
\end{equation}

For the tangential block, raising one index with the slice metric gives
\begin{align*}
 (f^2g_{q_\ell})^{-1}\partial_t(f^2g_{q_\ell})
   &=\frac2t\operatorname{Id}+O((t\ell)^{-1}),\\
 (f^2g_{q_\ell})^{-1}\partial_t^2(f^2g_{q_\ell})
   &=\frac2{t^2}\operatorname{Id}+O((t^2\ell)^{-1}),\\
 \operatorname{tr}_{f^2g_{q_\ell}}\partial_t(f^2g_{q_\ell})
   &=\frac{2n}{t}+O((t\ell)^{-1}).
\end{align*}
Thus the three extrinsic terms in \eqref{eq:gauss-tangent} sum to
\[
 -\frac{n-1}{t^2}\operatorname{Id}+O((t^2\ell)^{-1}).
\]
Consequently, as an endomorphism obtained by raising an index with the
slice metric,
\[
 (f^2g_{q_\ell})^{-1}
 \left.\Ric_{\bar g}\right|_{TN\times TN}
 =f^{-2}g_{q_\ell}^{-1}\Ric_{g_{q_\ell}}
  -\frac{n-1}{t^2}\operatorname{Id}
  +O((t^2\ell)^{-1}).
\]
Since $f\leq at$, \eqref{eq:pc-uniform-bounds} and
\eqref{eq:pc-parameter-choice} imply, for all large $\ell$,
\begin{equation}\label{eq:pc-tangent-lower}
 \left.\Ric_{\bar g}\right|_{TN\times TN}
 \geq\frac n{t^2}f(t)^2g_{q_\ell(t)}.
\end{equation}
The mixed component is exact:
\begin{equation}\label{eq:pc-mixed-exact}
 \Ric_{\bar g}(\partial_t,\mathord\cdot)
 =\frac{q_\ell'}2\left(
   \operatorname{div}_{g_{q_\ell}}\dot g_{q_\ell}
   -d\operatorname{tr}_{g_{q_\ell}}\dot g_{q_\ell}\right),
\end{equation}
Since $0\leq q_\ell\leq1$ and
$q_\ell'\leq2/(t\ell)$, one has
\[
 ae^{-\sigma}t\leq f(t)\leq at.
\]
Taking the norm of \eqref{eq:pc-mixed-exact} with the slice metric therefore
gives
\begin{equation}\label{eq:pc-mixed-bound}
 |\Ric_{\bar g}(\partial_t,\mathord\cdot)|_{f^2g_{q_\ell}}
 \leq\frac{C}{t^2\ell}.
\end{equation}
The inverse of the tangential Ricci block is at most
$\frac{t^2}{n}(f^2g_{q_\ell})^{-1}$.  Its Schur complement therefore obeys
\begin{equation}\label{eq:pc-schur}
 \Ric_{tt}-\Ric_{TN}^{-1}(\Ric_{tN},\Ric_{tN})
 \geq\frac1{8t^2\ell}-\frac{C^2}{nt^2\ell^2}>0
\end{equation}
for every sufficiently large $\ell$.  Equations
\eqref{eq:pc-tangent-lower} and \eqref{eq:pc-schur} prove
$\Ric_{\bar g}>0$.

Finally set $\widehat g=(ae^\ell)^{-2}\bar g$.  At the input the outward
normal is $-\partial_t$, whereas at the output it is $+\partial_t$; under
a homothety $c^2\bar g$, shape operators are multiplied by $c^{-1}$.
Since the path is constant at its endpoints,
\eqref{eq:pc-derivatives} and \eqref{eq:pc-slice-shape} therefore give
\begin{align*}
 \widehat g_{\rm in}&=g_0,&
 A_{\rm in}&=-a\left(1-\frac{2\sigma}{\ell}\right)\operatorname{Id},\\
 \widehat g_{\rm out}&=e^{2(\ell-\sigma)}g_1,&
 A_{\rm out}&=ae^{-\ell}
   \left(1-\frac{\sigma}{2\ell}\right)\operatorname{Id}.
\end{align*}
The radial length after this homothety is
\[
 (ae^\ell)^{-1}(e^{2\ell}-e^\ell)=\frac{e^\ell-1}{a}.
\]
Enlarge $\ell_0$ so that $\ell\geq4$, $\ell>2\sigma$, and all the uniform
error estimates above have been absorbed whenever $\ell\geq\ell_0$.
Then the first shape operator is greater than
$-\delta_{\rm in}\operatorname{Id}$ and the second is positive.  This proves
\eqref{eq:pc-exact-data} and the theorem.
\end{proof}

\end{document}